\documentclass[a4paper,11pt,reqno]{amsart}

\usepackage{amsmath}
\usepackage{bbm}
\usepackage{amssymb}
\usepackage{amsthm}
\usepackage{amscd}
\usepackage{bbold}

\usepackage{hyperref}

\usepackage[top=3cm,left=3cm,right=3cm,bottom=3.5cm]{geometry}

\usepackage[english]{babel}
\usepackage[utf8]{inputenc} 
\usepackage[T1]{fontenc}
\usepackage{ esint }

\usepackage{authblk}

\usepackage[T1]{fontenc}
\usepackage{lmodern}

\usepackage[colorinlistoftodos,prependcaption]{todonotes}

\newcommand{\Addresses}{{
		\bigskip
		\footnotesize
		
	 \textsc{MIT, Dept. of Math., 77 Massachusetts Avenue, Cambridge, MA 02139-4307.}\par\nopagebreak
		\textit{E-mail address}, \texttt{ozuch@mit.edu}
	}}

\newtheorem{thm}{Theorem}[section]

\newtheorem{lem}[thm]{Lemma}
\newtheorem{prop}[thm]{Proposition}
\newtheorem{cor}[thm]{Corollary}

\newtheorem{defn}[thm]{Definition}
\newtheorem{conj}[thm]{Conjecture}
\newtheorem{exmp}[thm]{Example}
\newtheorem{quest}[thm]{Question}

\newtheorem{rem}[thm]{Remark}
\newtheorem{note}[thm]{Note}

\DeclareMathOperator{\Rm}{\textup{Rm}}
\DeclareMathOperator{\R}{\textup{R}}
\DeclareMathOperator{\vol}{\textup{Vol}}
\DeclareMathOperator{\Ric}{\textup{Ric}}

\author{Tristan OZUCH}
\affil{Massachusetts Institute of Technology}
\date{}

\title{Higher order obstructions to the desingularization of Einstein metrics}

\begin{document}
\maketitle
\begin{center}
    \textbf{Tristan Ozuch}\\
    Massachusetts Institute of Technology
\end{center}

    \begin{abstract}
        We find new obstructions to the desingularization of compact Einstein orbifolds by smooth Einstein metrics. These new obstructions, specific to the compact situation, raise the question of whether a compact Einstein $4$-orbifold which is limit of Einstein metrics bubbling out Eguchi-Hanson metrics has to be Kähler. We then test these obstructions to discuss if it is possible to produce a Ricci-flat but not Kähler metric by the most promising desingularization configuration proposed by Page in 1981. We identify $84$ obstructions which, once compared to the $57$ degrees of freedom, indicate that almost all flat orbifold metrics on $\mathbb{T}^4\slash\mathbb{Z}_2$ should not be limit of Ricci-flat metrics with generic holonomy while bubbling out Eguchi-Hanson metrics. Perhaps surprisingly, in the most symmetric situation, we also identify a $14$-dimensional family of desingularizations satisfying all of our $84$ obstructions.
    \end{abstract}
    
    \tableofcontents

\section*{Introduction}

An Einstein metric $\mathbf{g}$ satisfies, for some real number $\Lambda$, the equation
$$\Ric(\mathbf{g})=\Lambda \mathbf{g}.$$
In dimension $4$, they are considered optimal due to the homogeneity of their Ricci curvature but also as critical points of the Einstein-Hilbert functional with fixed volume, $g\mapsto\int_M \R_g dvol_g$, and more importantly as minimizers of the $L^2$-norm of the Riemann curvature tensor, $g\mapsto \int_M |\Rm_g|^2dvol_g$. 

From dimension $4$, even under natural assumptions of bounded diameter (compactness) and lower bound on the volume (non-collapsing) Einstein metrics can develop singularities. One major goal for $4$-dimensional geometry is therefore to understand the compactification of the moduli space of Einstein metrics on a differentiable manifold $M^4$ defined as
\begin{equation}
    \mathbf{E}(M^4) := \left\{(M^4,\mathbf{g})\;|\;\exists \Lambda\in \mathbb{R},\; \Ric(\mathbf{g})=\Lambda \mathbf{g},\; \vol(M^4,\mathbf{g})= 1\right\}\slash\mathcal{D}(M^4)\label{def moduli space}
\end{equation}
where $\mathcal{D}(M^4)$ is the group of diffeomorphisms of $M^4$ acting on metrics by pull-back and to compactify it with a useful structure. The metric spaces which are limit of Einstein $4$-manifolds with uniformly controlled diameter and volume as well as the associated singularity blow-ups have been understood for a long time in the Gromov-Hausdorff sense \cite{and,bkn}: they are respectively \emph{Einstein orbifolds} and \emph{Ricci-flat ALE orbifolds}. The \emph{metric completion} of $(\mathbf{E}(M^4),d_{GH})$ is 
\begin{equation}
    \mathbf{E}(M^4)\;\cup\;\partial_o\mathbf{E}(M^4),\label{def frontière}
\end{equation}
where $\partial_o\mathbf{E}(M^4)$ is the set of orbifold $d_{GH}$-limits with bounded diameter (i.e. at finite $d_{GH}$-distance) of Einstein metrics on $M^4$. 
\\

In the present article, we find new types obstructions to the desingularization of Einstein orbifolds which are special to the compact context. They motivate and indicate positive answers the following questions:
\begin{itemize}
    \item Assume that $ (M_o^4,\mathbf{g}_o) \in \partial_o\mathbf{E}(M^4)$ where $M^4$ has the topology of $M^4_o$ desingularized by Eguchi-Hanson metrics. Is $\mathbf{g}_o$ Kähler-Einstein?
    \item Under the same assumptions, is $\partial_o\mathbf{E}(M^4)$ of codimension $2$ in $\mathbf{E}(M^4)$?
    \item Can we desingularize a flat metric on $\mathbb{T}^4\slash\mathbb{Z}_2$ by smooth Ricci-flat metrics with generic holonomy thanks to Eguchi-Hanson metrics?
\end{itemize}

\subsection*{A new obstruction to the desingularization in the compact case}

\emph{Any} smooth Einstein $4$-manifold close to a compact Einstein orbifold in a mere Gromov-Hausdorff sense has recently been produced by a gluing-perturbation procedure \cite{ozu1,ozu2}. In the present paper, for simplicity, we will focus on an Einstein orbifold $(M_o,\mathbf{g}_o)$ with \emph{integrable} Einstein deformations (like all known $4$-dimensional examples) and consider only the simplest singularities modeled on $\mathbb{R}^4\slash\mathbb{Z}_2\sim \mathbb{C}^2\slash\mathbb{Z}_2$, whose minimal resolution has the topology $T^*\mathbb{S}^2$ of the Eguchi-Hanson metric \cite{eh}.

\begin{rem}
    This is conjectured to be the only possible topology for an Einstein desingularization of $\mathbb{R}^4\slash\mathbb{Z}_2$, see \cite{bkn}.
\end{rem}
Let us denote $M$ the differentiable manifold obtained by minimal resolution (in some orientation) of the $\mathbb{R}^4\slash\mathbb{Z}_2\approx\mathbb{C}^2\slash\mathbb{Z}_2$ singularities of $M_o$. Using \cite{ozu1,ozu2}, we know that if there exists a sequence of Einstein metrics on $M$ converging to $(M_o,\mathbf{g}_o)$ in the Gromov-Hausdorff sense, then an obstruction already noticed in \cite{biq1} is satisfied at every singular point of $M_o$. More precisely, denoting respectively $\mathbf{R}^+_{\mathbf{g}_o}(p)$ and $\mathbf{R}^-_{\mathbf{g}_o}(p)$ the selfdual and anti-selfdual parts of the curvature at a singular point $p\in M_o$ seen as an endomorphism on the space of $2$-forms, we have either $\det \mathbf{R}^+_{\mathbf{g}_o}(p)=0$ or $\det \mathbf{R}^-_{\mathbf{g}_o}(p)=0$. This means that in some basis, we have either
\begin{equation}
    \mathbf{R}^+_{\mathbf{g}_o}(p) =\begin{bmatrix}0&0&0\\
0&*&*\\
0&*&*
\end{bmatrix}\text{ or, }\mathbf{R}^-_{\mathbf{g}_o}(p)=\begin{bmatrix}0&0&0\\
0&*&*\\
0&*&*
\end{bmatrix}.\label{premiere obst}
\end{equation}

We will first show that an additional obstruction holds under weaker assumptions than those of \cite{biq1}. Another physically motivated assumption is that of the \emph{stability} of the Einstein metric. This condition essentially means that the linearization of the Ricci curvature has nonnegative spectrum, see Definition \ref{def stable}.

We are now ready to state the first main result of this paper:
\begin{thm}\label{2e obst}\label{obst generique general}\label{nouvelleobst}
    Let $(M_o^4,\mathbf{g}_o)$ be a compact Einstein orbifold  with $\Ric(\mathbf{g}_o) = \Lambda \mathbf{g}_o$ for $\Lambda\in \mathbb{R}$, with integrable Einstein deformations and with singularities $\mathbb{R}^4\slash\mathbb{Z}_2$. Assume that we have $(M_o^4,\mathbf{g}_o)\in \overline{\mathbf{E}(M^4)}_{GH}$ for $M = M_o \# T^*\mathbb{S}^2\# \ldots\# T^*\mathbb{S}^2$, where $\#$ denotes a gluing in the positive or negative orientation and along the cone $\mathbb{R}^4\slash\mathbb{Z}_2$. Assume additionally \emph{one} of the following properties:
    \begin{enumerate}
        \item there exists a sequence $(\mathbf{g}_n)_n$ of \emph{stable Ricci-flat} (Definition \ref{def stable}) metrics in $\mathbf{E}(M^4)$ converging to $ \mathbf{g}_o$ in the Gromov-Hausdorff sense, or,
        \item $M_o$ has only one singularity and is \emph{rigid} (i.e. does not admit infinitesimal Einstein deformations).
    \end{enumerate}
    
    Then, for any singular point $p$ of $M_o$, there exists a basis of the selfdual $2$-forms or anti-selfdual $2$-forms in which we have (depending on the orientation of the gluing):
    \begin{equation}
    \mathbf{R}_{\mathbf{g}_o}^+(p) =\begin{bmatrix}0&0&0\\
0&0&0\\
0&0&\Lambda
\end{bmatrix}\text{ or }\mathbf{R}_{\mathbf{g}_o}^-(p)=\begin{bmatrix}0&0&0\\
0&0&0\\
0&0&\Lambda
\end{bmatrix}.\label{courbure obst}
\end{equation}
\end{thm}
Notice that a curvature of the form \eqref{courbure obst} is typical of Kähler-Einstein metrics.

Theorem \ref{nouvelleobst} is specific to the compact case and is in sharp contrast with Biquard's desingularization \cite{biq1} in the asymptotically hyperbolic case. Indeed, the degeneration of the asymptotically hyperbolic AdS Taub-Bolt  metrics converges to an orbifold which is rigid and has one singularity $\mathbb{R}^4\slash\mathbb{Z}_2$ but does not satisfy \eqref{courbure obst}, see Example \ref{adstaubbolt}. 

\begin{rem}
        Since the convergence is allowed to be only in the Gromov-Hausdorff sense, the assumptions of Theorem \ref{2e obst} are weaker than in \cite[Corollary 9.3]{biq1}, where the obstruction \eqref{premiere obst} at the singular point was first observed. Moreover, the present proof would also yield an obstruction for other Kleinian singularities by \cite[Lemme 9]{biq2} and \cite[Lemme 12]{biq2}.
\end{rem}

Another goal here is to obtain a new obstruction result under some technical assumption of existence of a sequence of metrics in $\mathbf{E}(M^4)$ sufficiently transverse to the boundary $\partial_o\mathbf{E}(M^4)$ of \eqref{def frontière}. We will say that a sequence of Einstein metrics is a \emph{transverse} desingularization of the orbifold $(M_o,\mathbf{g}_o)$ if the sequence is almost orthogonal to $\partial_o\mathbf{E}(M)$, see Definition \ref{non degenerate and transverse einstein} and Figure \ref{figure transverse}. Almost all Einstein metrics obtained by gluing-perturbation are transverse, see \cite{top1,top2,ls,don,bk}. Note that \cite{biq1,biq2,biq3} are notable exceptions. A weaker assumption than being transverse for an Einstein desingularization is being \emph{nondegenerate}. We will say that a sequence of Einstein metrics is a \emph{nondegenerate} Einstein desingularization of $(M_o,\mathbf{g}_o)\in \partial_o\mathbf{E}(M)$ if the sequence approaches $(M_o,\mathbf{g}_o)$ without being ``too tangent'' to $\partial_o\mathbf{E}(M)$, see Definition \ref{non degenerate and transverse einstein} and Figure \ref{figure non degenerate}. This notion is in the spirit of the assumption of \emph{genericity} of \cite{spo} and \emph{nondegeneracy} of \cite{br15}.

\begin{prop}\label{2e obst nondeg}
    Let $(M_o^4,\mathbf{g}_o)$ be a compact Einstein orbifold  with $\Ric(\mathbf{g}_o) = \Lambda \mathbf{g}_o$ for $\Lambda\in \mathbb{R}$, with integrable Einstein deformations and with singularities $\mathbb{R}^4\slash\mathbb{Z}_2$. Assume that we have $(M_o^4,\mathbf{g}_o)\in \overline{\mathbf{E}(M^4)}_{GH}$ for $M = M_o \# T^*\mathbb{S}^2\# \ldots\# T^*\mathbb{S}^2$, where $\#$ denotes a gluing in the positive or negative orientation and along the cone $\mathbb{R}^4\slash\mathbb{Z}_2$. Assume additionally \emph{one} of the following properties:
    \begin{enumerate}
        \item there exists a \emph{transverse} (Definition \ref{non degenerate and transverse einstein}) sequence $(\mathbf{g}_n)_n$ of metrics in $\mathbf{E}(M^4)$ converging to $ \mathbf{g}_o$ in the Gromov-Hausdorff sense,
        \item $\Lambda=0$ and there exists a \emph{nondegenerate} (Definition \ref{non degenerate and transverse einstein}) sequence $(\mathbf{g}_n)_n$ of metrics in $\mathbf{E}(M^4)$ converging to $ \mathbf{g}_o$ in the Gromov-Hausdorff sense,
        \item $M_o$ has only one singularity and there exists a \emph{nondegenerate} sequence $(\mathbf{g}_n)_n$ of metrics in $\mathbf{E}(M^4)$ converging to $ \mathbf{g}_o$ in the Gromov-Hausdorff sense.
    \end{enumerate}
    
    Then, for any singular point $p$ of $M_o$, $\mathbf{g}_o$ satisfies the obstruction \eqref{courbure obst}.
\end{prop}
\begin{rem}
     The proof of Proposition \ref{2e obst nondeg} relies on the fact that our \emph{compact} deformations add as many degrees of freedom as new obstructions to satisfy. The obstruction \eqref{courbure obst} is not satisfied by the orbifolds of the desingularizations of \cite{biq1,biq2,biq3} or Example \ref{adstaubbolt}, which rely on nondegenerate (but not transverse) Einstein deformations coming from the \emph{conformal} infinity. These deformations from the conformal infinity add an infinite number of degrees of freedom \emph{without} adding any obstruction. 
\end{rem}

It is somewhat expected that the completion $\mathbf{E}(M)\cup\partial_o\mathbf{E}(M)$ of the moduli space $\mathbf{E}(M)$ is real-analytic or real-subanalytic, see \cite{andsurv}. This would imply that the orbifold metrics can be approached by \emph{curves} of Einstein metrics. An initial intuition for Theorem \ref{2e obst} is \cite{biq3} which shows that if $\Lambda\neq 0$ and the obstruction \eqref{courbure obst} is not satisfied, then the orbifold metrics cannot be approached by curves of Einstein metrics. Our proofs however do not use this fact and even work in the case when $\Lambda=0$. They rely heavily on the computations of the second variations of the Ricci curvature of \cite{biq2} and the formalism of \cite{fine}.

\subsection*{Is $\partial_o\mathbf{E}(M)$ a boundary or a filling?}

An interesting problem on which Theorem \ref{2e obst} sheds some light is the size of $\partial_o\mathbf{E}(M^4)$ the set of singular metrics in the completion \eqref{def frontière} of the moduli space of Einstein metrics. Anderson proposes the following ``optimistic'' (in his own words) conjecture.
\begin{conj}[{\cite{andsurv}}]\label{conj codim 2}
    The subspace $\partial_o\mathbf{E}(M^4)$ is of codimension $2$ in $\overline{\mathbf{E}(M^4)}_{GH}$.
\end{conj}
This conjecture means that we should not think of $\partial_o\mathbf{E}(M^4)$ as a \emph{boundary} of $\mathbf{E}(M^4)$ but rather as a filling of missing pieces as it was pictured in \cite{andL2} in the case of $M^4=K3$.
\begin{rem}
    This is false in the asymptotically hyperbolic context where the desingularizations of \cite{biq1} applied to the AdS Taub-Bolt orbifold (Example \ref{adstaubbolt}) show that $\partial_o\mathbf{E}(M^4)$ is of codimension $1$.
\end{rem}

With our present vocabulary, Conjecture \ref{conj codim 2} means that given a compact Einstein orbifold $(M_o,\mathbf{g}_o)\in \partial_o\mathbf{E}(M^4)$, the space of Einstein desingularizations of $M_o$ which are \emph{transverse} to $\partial_o\mathbf{E}(M^4)$ is at least $2$-dimensional.

Theorem \ref{2e obst} can be seen as a step towards this conjecture in the situation where we desingularize by an Eguchi-Hanson metric. Indeed, its proof shows that:
\begin{itemize}
    \item if $\mathbf{R}^+_{\mathbf{g}_o}(p_o)$ is invertible, then there is no desingularization which is Einstein up to a $o(t)$ error,
    \item if the kernel of $\mathbf{R}^+_{\mathbf{g}_o}(p_o)$ is of dimension $1$, then there is no transverse desingularization which is Einstein up to a $o(t^2)$ error and there is exactly a one-dimensional set of desingularizations which are Einstein up to a $o(t)$ error,
    \item if the kernel of $\mathbf{R}^+_{\mathbf{g}_o}(p_o)$ is of dimension $2$ or $3$, then the transverse desingularizations which are Einstein up to a $o(t^2)$ error are of dimension $2$ or $3$ and in correspondence with the kernel of $\mathbf{R}^+_{\mathbf{g}_o}(p_o)$. 
\end{itemize}

Finally, recall that Einstein metrics with a curvature tensor of the form \eqref{courbure obst} for $\Lambda = 0$ at \emph{all} points are locally hyperkähler metrics. Since we have only identified the first two obstructions of an infinite list, we might wonder if the remaining obstructions impose this condition at all points. We ask the following question.
\begin{quest}\label{conj kahler desing}
    Let $(M_o^4,\mathbf{g}_o)$ be a Ricci-flat orbifold, and assume that it can be desingularized in the Gromov-Hausdorff sense by Einstein metrics $(M^4,\mathbf{g}_n)_{n\in \mathbb{N}^*}$ forming trees of Kähler Ricci-flat ALE metrics. Are the orbifold $(M_o^4,\mathbf{g}_o)$ and the smooth metrics $(M^4,\mathbf{g}_n)$ necessarily quotients of hyperkähler metrics?
\end{quest}
A positive answer to the above question combined with the folklore conjecture that all Ricci-flat ALE orbifolds are Kähler indicates that it may not be possible to produce Ricci-flat metrics which are not Kähler by gluing constructions on Ricci-flat orbifolds. This pessimistic perspective is related to one of the main issue in Riemannian geometry which is to understand the structure of Ricci-flat metrics on compact manifolds.

\subsection*{Ricci-flat metrics and reduced holonomy}

 For a long time, it was believed that compact Ricci-flat metrics should be flat. The resolution of Calabi's conjecture by Yau \cite{yau} provided many counter-examples. These so-called Calabi-Yau metrics are the only currently known examples  and they have reduced holonomy (that is different from the generic holonomy $SO(n)$ in dimension $n$). A lingering question since then is therefore the following one.
\begin{quest}\label{question Ricci flat holonomie}
    Do all compact Ricci-flat manifolds have reduced holonomy?
\end{quest}
A physically relevant additional property that one can ask from a Ricci-flat metric, is that of being \emph{stable}, see \cite{ach}.
\begin{quest}[{\cite{ach}}]\label{question Ricci flat holonomie stable}
    Do all compact \emph{stable} Ricci-flat manifolds have reduced holonomy?
\end{quest}

\subsection*{Construction of hyperkähler metrics by gluing-perturbation}
Compact non flat Ricci-flat metrics are not explicit at all, but a classical way to approach them and to obtain a more concrete picture is to produce them by gluing various explicit noncompact and singular blocks together into an approximate Ricci-flat metric and perturb it into an actual Ricci-flat metric. 

Gibbons-Pope \cite{gp} and Page \cite{pag} proposed a conjectural and physically motivated picture for some hyperkähler, hence Ricci-flat, metrics on some K3 surfaces. The idea was to desingularize the orbifold $\mathbb{T}^4\slash\mathbb{Z}_2$ which has $16$ singularities of type $\mathbb{R}^4\slash\mathbb{Z}_2$ by gluing $16$ Eguchi-Hanson metrics \emph{in the same orientation} and to perturb the result to a hyperkähler metric. This was rigorously obtained for Eguchi-Hanson metrics glued at comparable scales by Topiwala \cite{top1,top2} and Lebrun-Singer \cite{ls}, see also \cite{don}. The proof heavily relies on Kähler geometry arguments.

\subsection*{A Ricci-flat but not locally hyperkähler metric?}
In 1981, Page \cite{pag81} asks a new question.
\begin{quest}\label{question page}
Is it possible to perturb a gluing of Eguchi-Hanson metrics in \emph{different} orientations to the orbifold $\mathbb{T}^4\slash\mathbb{Z}_2$ into a Ricci-flat metric?
\end{quest}
\begin{rem}
    It is often conjectured, see for instance \cite{bkn}, that the only Ricci-flat metric asymptotic to  $\mathbb{R}^4\slash\mathbb{Z}_2$ is the Eguchi-Hanson metric. This would therefore describe any possible Einstein desingularization of $\mathbb{T}^4\slash\mathbb{Z}_2$.
\end{rem}
Having different orientations for the Eguchi-Hanson metrics prevents the use of Kähler geometry techniques. A positive answer to Question \ref{question page} would provide a Ricci-flat metric with generic holonomy and a negative answer to Question \ref{question Ricci flat holonomie}.
\subsection*{Brendle-Kapouleas configuration}
In \cite{bk}, Brendle and Kapouleas studied Question \ref{question page} in the most symmetric situation: the orbifold metric comes from the regular lattice $\mathbb{Z}^4$, the points where the gluing are done in the positive or negative orientation follow a so-called \emph{chessboard pattern}, and the Eguchi-Hanson metrics in the same orientation are glued with the same $SO(4)$ or $O(4)\backslash SO(4)$ parameter and the same size. In this $1$-dimensional set of configurations, they exhibit \emph{one} interaction between the bubbles glued in different orientation. They remarkably use this obstruction to construct an intriguing solution to the Ricci flow which exhibits a new behavior: it is an ancient Ricci-flow desingularizing the orbifold $\mathbb{T}^4\slash\mathbb{Z}_2$.

\subsection*{Partial answer to Page's question.}
In this article, we consider the general situation of the desingularization of a flat metric on $\mathbb{T}^4\slash\mathbb{Z}_2$ obtained from a lattice $L(\mathbb{Z}^4)$ for $L\in GL(4,\mathbb{R})$ and desingularized by Eguchi-Hanson metrics glued in various directions, sizes and orientations. For each configuration of positively and negatively oriented Eguchi-Hanson metrics, this yields a $57$-dimensional space of candidates to a Ricci-flat gluing-perturbation as in \cite{ozu2} and therefore of potential positive answers to Question \ref{question page}.

We however identify a great number of interactions between these metrics which induce obstructions interpreted as curvature conditions similar to those of \eqref{premiere obst} and \eqref{courbure obst}. 

\begin{thm}[Informal, Corollaries \ref{list obst} and \ref{cor freedom constraints}]\label{thm informel obst tore}
We have the following numbers of constraints in this $57$-dimensional space of desingularization configurations.
    \begin{enumerate}
        \item For \emph{general} parameters, there are $57$ \emph{necessary} polynomial equations that the gluing parameters of the Eguchi-Hanson metrics should satisfy for an Einstein desingularization to exist. $48$ of these equations are analogous to the obstruction \eqref{premiere obst}.
        \item For \emph{general} parameters, there are $84$ \emph{necessary} polynomial equations that the gluing parameters should satisfy for a \emph{stable} Ricci-flat desingularization to exist. $80$ of these equations are analogous to the obstruction \eqref{courbure obst}.
        \item For parameters corresponding to a \emph{nondegenerate} desingularization, there are $84$ \emph{necessary} polynomial equations that the gluing parameters of the Eguchi-Hanson metrics should satisfy for an Einstein desingularization to exist. $80$ of these equations are analogous to the obstruction \eqref{courbure obst}.
    \end{enumerate}
    The coefficients of these (quadratic and quartic) polynomial equations depend on the flat metric on $\mathbb{T}^4\slash\mathbb{Z}_2$ only.
\end{thm}

 This indicates that generic flat metrics on $\mathbb{T}^4\slash\mathbb{Z}_2$, should not be Gromov-Hausdorff limit of Einstein metrics bubbling out Eguchi-Hanson metrics in different orientations.
\begin{rem}
    In the configuration proposed by Brendle and Kapouleas \cite{bk}, \emph{none} of the $57$ obstructions of the first point is satisfied.
\end{rem}

A simpler consequence is that a lot of configurations of positively and negatively oriented Eguchi-Hanson metrics are impossible. We in particular obtain the following obstruction under a topological assumption.
\begin{thm}[Theorem \ref{obst un EH}]\label{obst desing un eh}
    There does not exist a sequence of Einstein metrics $d_{GH}$-converging to the regular torus (obtained from the lattice $\mathbb{Z}^4$) by bubbling out exactly \emph{one} positively oriented Eguchi-Hanson metric and 15 negatively oriented Eguchi-Hanson metrics. 
\end{thm}
\begin{rem}
    It is likely that this result holds without assuming that the lattice is the usual $\mathbb{Z}^4$, see Conjecture \ref{conj 1 eh+}.
\end{rem}
    \begin{rem}
        The desingularization by \emph{stable} Ricci-flat metrics is much more restrictive. A direct extension of the proof of Theorem \ref{obst desing un eh} states that the desingularization is impossible with less than $3$ positively oriented Eguchi-Hanson metrics and the rest of negatively oriented, see Corollary \ref{obst 3 positive generale stable}.
    \end{rem}
    
\subsection*{A configuration satisfying the first obstructions.}
Seeing the above obstruction results, one would be tempted to try and prove that it is impossible to obtain a Ricci-flat but not hyperkähler metric by gluing-perturbation of Eguchi-Hanson metrics to a flat $\mathbb{T}^4\slash\mathbb{Z}_2$. One would maybe expect that the above $84$ obstructions should be enough to show that every configuration is obstructed. 

Perhaps surprisingly, we find a set of solutions to our $84$ equations. More precisely, there exists an explicit $14$-dimensional set of configurations of gluing of Eguchi-Hanson metrics which satisfy the $84$ above equations. The orientations of the Eguchi-Hanson metrics follow the chessboard pattern of \cite{bk}, but their directions and sizes are not constrained.

\subsection*{Acknowledgments}
The author would like to thank his PhD advisor, Olivier Biquard for his help and interest in the early stages of this project.

\section{Main definitions}

\begin{note}
    All along this article, when the name of a metric is in bold, then it is Einstein.
\end{note}
\subsection{Einstein orbifolds and ALE spaces}
For $\Gamma$ a finite subgroup of $SO(4)$ acting freely on $\mathbb{S}^3$, let us denote $(\mathbb{R}^4\slash\Gamma,\mathbf{e})$ the flat orbifold obtained by the quotient of the Euclidean metric on $\mathbb{R}^4$ by the action of $\Gamma$, and $r_\mathbf{e}:= d_\mathbf{e}(.,0)$.

\begin{defn}[Orbifold (with isolated singularities)]\label{orb Ein}
    We will say that a metric space $(M_o,g_o)$ is an orbifold of dimension $n\geqslant 2$ if there exists $\epsilon_0>0$ and a finite number of points $(p_k)_k$ of $M_o$ called \emph{singular} such that we have the following properties:
    \begin{enumerate}
        \item the space $(M_o\backslash\{p_k\}_k,g_o)$ is a manifold of dimension $n$,
        \item for each singular point $p_k$ of $M_o$, there exists a neighborhood of $p_k$, $ U_k\subset M_o$, a finite subgroup acting freely on $\mathbb{S}^{n-1}$, $\Gamma_k\subset SO(n)$, and a diffeomorphism $ \Phi_k: B_\mathbf{e}(0,\epsilon_0)\subset\mathbb{R}^n\slash\Gamma_k \to U_k\subset M_o $ for which, the pull-back of $\Phi_k^*g_o$  on the covering $\mathbb{R}^n$ is smooth. Here, \emph{diffeomorphism} is understood in the orbifold sense, i.e. it is a diffeomorphism between the cover $\mathbb{R}^4$ and the associated cover of $U_k$.
    \end{enumerate}
\end{defn}
\begin{rem}\label{analysis orbifold}
    The analysis on an orbifold is exactly the same as the analysis on a manifold up to using finite local coverings at the singular points.
\end{rem}

\begin{defn}[The function $r_o$ on an orbifold]\label{ro}
    We define $r_o$, a smooth function on $M_o$ satisfying $r_o:= (\Phi_k)_* r_\mathbf{e}$ on each $U_k$, and such that on $M_o\backslash U_k$, we have $\epsilon_0\leqslant r_o<1$ (the different choices will be equivalent for our applications).
    
    We will denote, for $0<\epsilon\leqslant\epsilon_0$, 
    \begin{equation}
        M_o(\epsilon):= \{r_o>\epsilon\} = M_o\backslash  \Big(\bigcup_k \Phi_k\big(\overline{B_\mathbf{e}(0,\epsilon)}\big) \Big).\label{eq Moepsilon}
    \end{equation}
    
\end{defn}

Let us now turn to ALE metrics.

\begin{defn}[ALE orbifold (with isolated singularities)]\label{def orb ale}
    An ALE orbifold of dimension $n\geqslant 4$, $(N,g_b)$ is a metric space for which there exists $\epsilon_0>0$, singular points $(p_k)_k$ and a compact $K\subset N$ for which we have:
    \begin{enumerate}
        \item $(N,g_b)$ is an orbifold of dimension $n$,
        \item there exists a diffeomorphism $\Psi_\infty: (\mathbb{R}^n\slash\Gamma_\infty)\backslash \overline{B_\mathbf{e}(0,\epsilon_0^{-1})} \to N\backslash K$ such that $$r_\mathbf{e}^l|\nabla^l(\Psi_\infty^* g_b - \mathbf{e})|_{\mathbf{e}}\leqslant C_l r_\mathbf{e}^{-n}.$$
    \end{enumerate}
\end{defn}

\begin{defn}[The function $r_{b}$ on an ALE orbifold]\label{rb}
We define $r_{b}$ a smooth function on $N$ satisfying $r_{b}:= (\Psi_k)_* r_\mathbf{e}$ on each $U_k$, and $r_{b}:= (\Psi_\infty)_* r_\mathbf{e}$ on $U_\infty$, and such that $\epsilon_0\leqslant  r_{b}\leqslant \epsilon_0^{-1}$ on the rest of $N$ (the different choices are equivalent for our applications).

For $0<\epsilon\leqslant\epsilon_0$, we will denote
\begin{equation}
    N(\epsilon):= \{\epsilon<r_b<\epsilon^{-1}\} = N\backslash  \Big(\bigcup_k \Psi_k\big(\overline{B_\mathbf{e}(0,\epsilon)}\big) \cup \Psi_\infty \big((\mathbb{R}^4\slash\Gamma_\infty)\backslash B_\mathbf{e}(0,\epsilon^{-1})\big)\Big).\label{eq Njepsilon}
\end{equation}
\end{defn}
Now, consider a subset $S_o$ of the singular points of $M_o$ (respectively $S$ of $N$).
\begin{defn}[Functionals $r_{o,S_o}$ and $r_{b,S}$]\label{def roSo rbS}
    We define the functional $r_{o,S_o}$ (respectively $r_{b,S}$) exactly like in Definitions \ref{ro} (respectively \ref{rb}) by only considering the sets $U_k$ containing points of $S_o$ (respectively $S$).
\end{defn}

A metric $g$ close to $\mathbf{g}_o$ (with some adapted decay at infinity in the noncompact case) is Einstein and is in Bianchi gauge with respect to $\mathbf{g}_o$, that is satisfies $B_{\mathbf{g}_o} g := (\delta_{\mathbf{g}_o} + \frac{1}{2}d \textup{tr}_{\textbf{g}_o})g =0$, if and only if it satisfies
$$\mathbf{\Phi}_{\mathbf{g}_o}(g) := \Ric(g) -\Lambda g + \delta^*_gB_{\mathbf{g}_o}g =0,$$
see for instance \cite[Section 6]{biq1}.
The linearization of $\mathbf{\Phi}_{\mathbf{g}_o}$ at $\textbf{g}_o$ is
$$P_{\mathbf{g}_o}:= \frac{1}{2}\nabla_{\textbf{g}_o}^*\nabla_{\textbf{g}_o} - \mathring{\R}_{\textbf{g}_o},$$
where $\mathring{\R}$ is the action of curvature on symmetric $2$-tensors: for an orthonormal basis $(e_i)_i$,
		$$\mathring{\R}(h)(X,Y)= \sum_i h\big(\Rm(e_i,X)Y,e_i\big).$$
\begin{defn}[Infinitesimal Einstein deformations]
    For a compact or ALE Einstein orbifold $(M,\mathbf{g})$ with $\Ric(\mathbf{g}) = \Lambda\mathbf{g}$, we define $\mathbf{O}(\mathbf{g})$ as the kernel of $P_{\mathbf{g}}$ on $L^2(\mathbf{g})$.
\end{defn}

\begin{defn}[Einstein modulo obstructions deformations]\label{def einst mod obst orbifold}
    According to \cite{ozu2} (see also \cite{koi}), there exists $\epsilon>0$ such that for any $v\in \mathbf{O}(\mathbf{g})$ with $\|v\|_{L^2(\mathbf{g})}<\epsilon$, there exists a unique metric $\Bar{g}_v$ satisfying
    \begin{enumerate}
        \item $\mathbf{\Phi}_{\mathbf{g}}(\Bar{g}_v)\in \mathbf{O}(\mathbf{g})$,
        \item $\Bar{g}_v-(\mathbf{g}+v)\perp_{\mathbf{g}} \mathbf{O}(\mathbf{g})$,
        \item $\|\Bar{g}_v-\mathbf{g}\|_{L^2(\mathbf{g})}\leqslant 2 \epsilon$.
    \end{enumerate}
\end{defn}

All along this article, in order to simplify the exposition and for lack of nonintegrable examples, we will only consider Einstein metrics with integrable Einstein deformations.

\begin{defn}[Einstein metric with integrable deformations]\label{einstein integrable deformations}
    An Einstein metric $\mathbf{g}$ only has \emph{integrable} infinitesimal Einstein deformations if for any $v\in \mathbf{O}(\mathbf{g})$, the metric $\Bar{g}_v = \mathbf{g}_v$ actually satisfies 
    \begin{enumerate}
        \item $\mathbf{\Phi}_{\mathbf{g}}(\mathbf{g}_v)=0$,
        \item $\mathbf{g}_v-(\mathbf{g}+v)\perp_{\mathbf{g}} \mathbf{O}(\mathbf{g})$,
        \item $\|\mathbf{g}_v-\mathbf{g}\|_{L^2(\mathbf{g})}\leqslant 2 \epsilon$.
    \end{enumerate}
\end{defn}
An additional class of metrics we will be interested in is that of \emph{stable Ricci-flat} metrics.
\begin{defn}[{Stable Ricci-flat metrics}]\label{def stable}
    We will say that $\mathbf{g}$ is a \emph{stable Ricci-flat metric} if it is Ricci-flat and satisfies: for all smooth compactly supported  symmetric $2$-tensor $h$:
    $$\big\langle P_{\mathbf{g}}h,h\big\rangle_{L^2(\mathbf{g})}\geqslant 0.$$
\end{defn}
All of the known examples of compact Ricci-flat metrics are \emph{stable}.

\subsection{Eguchi-Hanson metrics}\label{section eguchi hanson}

Let us describe the Eguchi-Hanson metric \cite{eh}. In this article, the Ricci-flat ALE metrics  $(N,g_b)$ we will consider will almost always be homothetic transformations of the Eguchi-Hanson metric.

The Eguchi-Hanson metric, which we will denote $\mathbf{eh}$, is defined on $T^*\mathbb{S}^2$. It is asymptotic to the flat cone $\mathbb{R}^4\slash\mathbb{Z}_2$. Denote $(x_1,x_2,x_3,x_4)$ coordinates in an orthonormal basis of $\mathbb{R}^4$, and define $r:= \sqrt{x_1^2+x_2^2+x_3^2+x_4^2}$, and a basis of invariant $1$-forms on the sphere $\mathbb{S}^3$, $(\alpha_1,\alpha_2,\alpha_3)$ by $$\alpha_1 :=\frac{1}{r^2}(x_1dx_2-x_2dx_1+x_3dx_4-x_4dx_3)$$ and the other by cyclic permutation of the indices $\{2,3,4\}$.

Outside of the zero section of $T^*\mathbb{S}^2$ represented by $r=0$, the metric $\mathbf{eh}$ has the following expression (with the identification $T^*\mathbb{S}^2\backslash \mathbb{S}^2\approx (\mathbb{R}^4\slash\mathbb{Z}_2) \backslash \{0\}$):
\begin{equation}
    \mathbf{eh}:= \sqrt{\frac{r^4}{1+r^4}}(dr^2+r^2\alpha_1^2) + \sqrt{1+r^4}(\alpha_2^2+\alpha_3^2).\label{eguchihanson}
\end{equation}
The metric extends to $ T^*\mathbb{S}^2 $ by adding the zero section $\mathbb{S}^2$ with metric $\alpha_2^2+\alpha_3^2$ at $r=0$. We will always denote $\mathbb{S}^2$ that zero-section of $ T^*\mathbb{S}^2 $.

This gives the asymptotic development at infinity $$\mathbf{eh}= \mathbf{e} - \frac{1}{2r^4}(dr^2 +r^2\alpha_1^2 -r^2\alpha_2^2-r^2\alpha_3^2)+\mathcal{O}(r^{-8}),$$
where $H^4 := -\frac{1}{2r^4}(dr^2 +r^2\alpha_1^2 -r^2\alpha_2^2-r^2\alpha_3^2)$ is divergence-free and trace-free with respect to the Euclidean metric $\mathbf{e}=dr^2+r^2(\alpha_1^2+\alpha_2^2+\alpha_3^2)$.
\subsubsection{Curvature of Eguchi-Hanson metrics.}
On $\mathbb{R}^4$, we define the $2$-forms $$\theta_1^\pm :=  r dr\wedge \alpha_1 \pm r^2\alpha_2\wedge\alpha_3,$$ and similarly $\theta_2^\pm$ and $\theta_3^\pm$ by cyclic permutations, and we define the usual bases $$\omega^\pm_1 := dx_1\wedge dx_2\pm dx_3\wedge dx_4$$ and similarly $\omega_2^\pm$ and $\omega_3^\pm$ by cyclic permutations. The $\theta_i^+$ and $\omega_i^+$ are selfdual and the $\theta_i^-$ and $\omega_i^-$ are anti-selfdual. Direct computations yield the following relation between them: $\theta_i^+ = \omega_i^+,$
    and for $x \in \mathbb{R}^4$,
    \begin{align}
        \theta_i^-(x) = \sum_{j=1}^3 \frac{x^T(\omega_i^+\circ\omega_j^-)x}{|x|^2}\omega_j^-,\label{rotation 2 forms}
    \end{align}
    where $ \omega_1^+\circ\omega_i^- $ is the symmetric traceless matrix given by the (commuting) product of the antisymmetric matrices associated to $\omega_i^+$ and $\omega_j^-$, and where $x^T$ is the transpose of $x$.
\begin{rem}
    The product $\omega^+_i\circ\omega^-_j =\omega^-_i\circ\omega^+_j$ is a trace-free involution and therefore it is characterized by two planes of eigenvalues $1$ and $-1$.
\end{rem}
\begin{rem}\label{identification traceless 2tensor 2 formes}
    We will often use the identification of the traceless symmetric $2$-tensors $\textup{Sym}^2_0(TM)$ and $\Omega^+\otimes \Omega^-$ where $(\omega^+,\omega^-)\in\Omega^+\times \Omega^-$ is associated to $\omega^+\circ \omega^- = \omega^-\circ \omega^+$.
\end{rem}

For any $\zeta= (\zeta_1,\zeta_2,\zeta_3)\in \mathbb{R}^3$, we denote
$$\zeta^\pm :=\zeta_1\omega_1^\pm +\zeta_2\omega_2^\pm +\zeta_3\omega_3^\pm\in \Omega^\pm.$$
\begin{defn}\label{def rho -}
        We define for any $x\in \mathbb{R}^4$ the linear transformation $\rho_x$ which to $\zeta =(\zeta_1,\zeta_2,\zeta_3)\in\mathbb{R}^3$ associates $\rho_x(\zeta)\in \mathbb{R}^3$ whose $j$'s coordinate is
        $$\rho_x(\zeta)_j =  \frac{x^T(\zeta^+\circ\omega_j^-)x}{|x|^2}.$$
    \end{defn}
    \begin{rem}
        Identifying $\mathbb{R}^4$ with the space of quaternions and $\mathbb{R}^3$ with the space of pure quaternions, for any $x\in \mathbb{R}^4$, the map $\rho_x: \mathbb{R}^3\to \mathbb{R}^3$ is exactly the rotation produced on the pure quaternionic part by conjugation by $x$. More precisely, identifying $ \zeta\in \mathbb{R}^3 $ with the pure quaternion $(0,\zeta)$ in the usual basis, we have $$(0,\rho_x(\zeta))=x\cdot(0,\zeta)\cdot x^{-1}.$$
    \end{rem}

    The metric $\mathbf{eh}$ is Ricci-flat, anti-selfdual and has the following anti-selfdual curvature:
    \begin{equation}
        \mathbf{R}^-(\mathbf{eh}) = \frac{1}{r^6}\begin{bmatrix}
8 & 0 & 0 \\
0 & -4 & 0\\
0 & 0 & -4
\end{bmatrix}
    \end{equation}
    in a basis asymptotic to the basis $(\theta_i^{-})_i$, see \cite{eh}. We also remark that we have $H^4 = -\frac{\omega_1^+\circ\theta_1^-}{2r^4}$ where $H^4$ is defined by $\mathbf{eh} = \textbf{e} +H^4 +\mathcal{O}(r^{-8})$ at infinity and where $\omega_1^+\circ\theta_1^-$ is the (commutating) composition of the antisymmetric matrices associated to $\omega_1^+$ and $\theta_1^-$.
    
    More generally, for any $\zeta = (\zeta_1,\zeta_2,\zeta_3)\in \mathbb{R}^3\backslash \{0\}$, consider $(\zeta^+_{(k)})_{k\in\{1,2,3\}}$ an orthogonal basis of $\Omega^+$ with constant length $\sqrt{2}=|\omega_1^+|$ and $\zeta^+_{(k)}=\zeta^+/|\zeta|=\sqrt{2}\zeta^+/|\zeta^+|$. We may replace the above $1$-forms $\alpha_k = \omega_k^+(dr)$ by the $1$-forms $\zeta^+_{(k)}(dr)$. By a change of variables $u=\frac{r}{\sqrt{|\zeta|}}$, one checks that the metric
    \begin{equation}
        \mathbf{eh}_{\zeta^+}:= \sqrt{\frac{r^4}{|\zeta|^2+r^4}}\left(dr^2+r^2\zeta^+_{(1)}(dr)^2\right) + \sqrt{|\zeta|^2+r^4}\left(\zeta^+_{(2)}(dr)^2+\zeta^+_{(3)}(dr)^2\right)\label{eh zeta+}
    \end{equation}
    is homothetic to $\mathbf{eh}$ and more precisely satisfies the following properties.

    \begin{prop}\label{differents zeta}
        Let $\zeta= (\zeta_1,\zeta_2,\zeta_3)\in \mathbb{R}^3\backslash \{0\}$ and denote $\zeta^+ := \zeta_1\omega_1^+ +\zeta_2\omega_2^+ +\zeta_3\omega_3^+$. We have the following properties:
        \begin{enumerate}
            \item at infinity, we have
            \begin{equation}
                \mathbf{eh}_{\zeta^+} = \mathbf{e} - \frac{\rho(\zeta)^-\circ \zeta^+}{2r^4} + \mathcal{O}(|\zeta|^4r^{-8}),\label{dvp eh}
            \end{equation}
            \item $\mathbf{eh}_{\zeta^+}$ is isometric to $|\zeta|\cdot\mathbf{eh}$, 
    \item $\mathbf{eh}_{(1,0,0)^+} = \mathbf{eh}$, and
    \item on $\mathbb{S}^2$, denoting $(\zeta^+_{(k)})_{k\in\{1,2,3\}}$ a basis of the selfdual $2$-forms of constant length with $\zeta^+_{(1)}=\zeta^+/|\zeta|$, the zero-section of $T^*\mathbb{S}^2$, the metric is $ |\zeta^+|(\zeta^+_{(2)}(dr)^2+ \zeta^+_{(3)}(dr)^2)$.
        \end{enumerate}
    \end{prop}
    \begin{proof}[]
    
    \end{proof}
    \begin{rem}
        These metrics $\mathbf{eh}_{\zeta^+}$ for $\zeta\in \mathbb{R}^3\backslash\{0\}$ reach \emph{all} of the metrics obtained by \emph{orientation-preserving} rotations and rescaling of $\mathbf{eh}$ \emph{up to isometry}.
    \end{rem}
    
    We will need to understand their curvature in the last sections of this article. From the computations of \cite{eh} and Proposition \ref{differents zeta}, we have the following result.
    
        For $\zeta = (\zeta_1,\zeta_2,\zeta_3)$, denote $ H^4(\zeta^+) := -\frac{\rho(\zeta)^-\circ \zeta^+}{2r_\textbf{e}^4}$. Then, the linearization of the curvature at $\mathbf{e}$ in the direction $H^4(\zeta^+)$ is Ricci-flat anti-selfdual and we have
        \begin{equation}
            d_{\mathbf{e}}\mathbf{R}^-(H^4(\zeta^+)) = \frac{|\zeta|^2}{r^6}\begin{bmatrix}
8 & 0 & 0 \\
0 & -4 & 0\\
0 & 0 & -4
\end{bmatrix}\label{courbure variation h4 matrix}
        \end{equation}
        in any orthogonal basis with constant length of the anti-selfdual $2$-forms of $\mathbb{R}^4$ whose first vector is $\frac{\rho(\zeta)^-}{|\zeta|}$.
        
        Identifying the space of endomorphisms of $\Omega^-$, $\textup{End}(\Omega^-)$ and $\Omega^-\otimes\Omega^-$, this can be rewritten as 
        \begin{equation}
             d_{\mathbf{e}}\mathbf{R}^-(H^4(\zeta^+)) = \frac{ 12 \pi_{ \textup{tr}}(\rho(\zeta)^-\otimes\rho(\zeta)^-)}{r^6}\label{courbure variation h4}
        \end{equation}
        where $\pi_{\textup{tr}}h := h - \frac{\textup{tr}h}{3}\textup{I}_3$ is a projection on trace-free matrices.

\begin{rem}
    In the constant basis $(\omega_1^-,\omega_2^-,\omega_3^-)$, the coefficients of $d_{\mathbf{e}}\mathbf{R}^-(H^4(\zeta^+))$ (seen as a $3\times 3$ matrix) are harmonic by \eqref{rotation 2 forms} as expected from \cite[Lemme 3]{biq2}.
\end{rem}
\subsubsection{Infinitesimal variations.}
The deformations of the Eguchi-Hanson metric are given by the variations of $\zeta$ in \eqref{eh zeta+}.

For $\zeta = (1,0,0)$, an orthogonal basis of the $L^2(\mathbf{eh})$-kernel of $P_{\mathbf{eh}}$ denoted $\mathbf{O}(\mathbf{eh})$ may be computed. Denote $(\mathbf{o}_{1},\mathbf{o}_{2},\mathbf{o}_{3})$ given by the infinitesimal variations of $\zeta$ respectively in the directions $(1,0,0)$, $(0,1,0)$ and $(0,0,1)$. They have the following developments in the coordinates of the above development \eqref{dvp eh},
\begin{equation}
        \mathbf{o}_{k} = O_{k}^{4} + \mathcal{O}(r_\mathbf{e}^{-8}) \text{ with } O_{k}^{4}=-\frac{\theta_1^-\circ \theta_k^+}{r^4}.\label{def O^4_k}
    \end{equation}
   
    \begin{rem}\label{O14 prop a H4}
        The symmetric $2$-tensor $O_{1}^4 $ is equal to twice to the asymptotic term $H^4$ of $\mathbf{eh}$, see \eqref{dvp eh}.
    \end{rem}
    The infinitesimal deformations $\mathbf{o}_1$, $\mathbf{o}_2$ and $\mathbf{o}_3$ respectively induce the following infinitesimal changes of anti-selfdual curvature in the basis $(\theta_1^-,\theta_2^-,\theta_3^-)$:
    \begin{equation}
             \frac{1}{r^6}\begin{bmatrix}
8 & 0 & 0 \\
0 & -4 & 0\\
0 & 0 & -4
\end{bmatrix}, \;\; \frac{1}{r^6}\begin{bmatrix}
0 & 4 & 0 \\
4 & 0 & 0\\
0 & 0 & 0
\end{bmatrix},\; \text{ and }\; \frac{1}{r^6}\begin{bmatrix}
0 & 0 & 4 \\
0 & 0 & 0\\
4 & 0 & 0
\end{bmatrix}.
\end{equation}

For another $\zeta\in \mathbb{R}^3\backslash\{0\}$, let us consider, thanks to the second point of Proposition \ref{differents zeta}, a diffeomorphism $\phi$ such that $\phi^*\mathbf{eh}_{\zeta^+} = |\zeta|\cdot \mathbf{eh}$. Then, $(\mathbf{o}_k(\zeta^+))_k:=(\phi_*\mathbf{o}_k)_k$ forms an orthogonal basis of $\mathbf{O}(\mathbf{eh}_{\zeta^+})$ with $\|\phi_*\mathbf{o}_k\|_{L^2(\mathbf{eh}_{\zeta^+})}=\|\mathbf{o}_k\|_{L^2(\mathbf{eh})}$
since the $L^2$-norm of $2$-tensors is invariant by rescaling in dimension $4$. Recalling that in the same coordinates at infinity, the first $r_\mathbf{e}^{-4}$-terms of $|\zeta|\phi_*\mathbf{eh}$ and of $\mathbf{eh}_{\zeta^+}$ coincide, we find 
$$|\zeta|\cdot\phi_*\Big(\frac{\omega_1^+\circ\theta_1^-}{2r_\mathbf{e}^4}\Big) =\frac{\zeta^+\circ\rho(\zeta)^-}{2r_\mathbf{e}^4}$$
and therefore, denoting $(\zeta^+_{(k)})_{k\in\{1,2,3\}}$ a basis of the selfdual $2$-forms of constant length with $\zeta^+_{(1)}=\zeta^+/|\zeta|$, the development at infinity of the $\phi_*\mathbf{o}_k$ is
\begin{equation}
    \mathbf{o}_k(\zeta^+)=\phi_*\mathbf{o}_k = \frac{\zeta^+_{(k)}\circ\rho(\zeta)^-}{r_\mathbf{e}^4} + \mathcal{O}(|\zeta^+|^3r_\mathbf{e}^{-8}).\label{dvp obst avec zeta}
\end{equation}
\begin{rem}\label{rem deformation eh zeta et noyau}
    We check that $\mathbf{o}_k(\zeta^+)=\partial_{\zeta^+_{(k)}}\mathbf{eh}_{\zeta^+},$
    where we denoted $\partial_{\zeta^+_{(k)}}\mathbf{eh}_{\zeta^+}$ the differential of $\zeta^+\mapsto \mathbf{eh}_{\zeta^+}$ at $\zeta^+$ in the direction $\zeta^+_{(k)}$.
\end{rem}

\subsubsection{Negative orientation.}

We will also consider \emph{negatively oriented} Eguchi-Hanson metrics. For $\zeta = (\zeta_1,\zeta_2,\zeta_3) \in \mathbb{R}^3\backslash\{0\}$, we define
\begin{equation}
    \mathbf{eh}_{\zeta^-}:= \sqrt{\frac{r^4}{|\zeta|^2+r^4}}\left(dr^2+r^2\zeta^-_{(1)}(dr)^2\right) + \sqrt{|\zeta|^2+r^4}\left(\zeta^-_{(2)}(dr)^2+\zeta^-_{(3)}(dr)^2\right).\label{eh zeta-}
\end{equation}

\noindent It is isometric to $|\zeta|\cdot\mathbf{eh}$ but has the opposite orientation.
At infinity, we have the development
$$ \mathbf{eh}_{\zeta^-}=\mathbf{e}-\frac{\zeta^-\circ \rho(\zeta)^+}{2r^4}+\mathcal{O}(|\zeta|^4r^{-8}).$$

\subsection{Naïve desingularizations}\label{reecriture controle}

Let us now recall the definition of a naïve desingularization of an orbifold from \cite{ozu1}.

\subsubsection*{Gluing of ALE spaces to orbifold singularities.}

    Let $0<2\epsilon<\epsilon_0$ be a fixed constant, $t>0$, $(M_o,g_o)$ an orbifold and $\Phi: B_e(0,\epsilon_0)\subset\mathbb{R}^4\slash\mathbb{Z}_2\to U$ a local chart of Definition \ref{orb Ein} around a singular point $p\in M_o$. Let also $(N,g_b)$ be an ALE metric asymptotic to $\mathbb{R}^4\slash\mathbb{Z}_2$, and $\Psi_\infty: (\mathbb{R}^4\slash\mathbb{Z}_2)\backslash B_e(0,\epsilon_0^{-1}t^\frac{1}{2}) \to N\backslash K$ a chart at infinity of Definition \ref{def orb ale}.
    
    For $s>0$, define $\phi_s: x\in \mathbb{R}^4\slash\Gamma\to sx \in \mathbb{R}^4\slash\Gamma$. For $t<\epsilon_0^4$, we define $M_o\#N$ as $N$ glued to $M_o$ thanks to the diffeomorphism 
$$ \Phi\circ\phi_{\sqrt{t}}\circ\Psi^{-1} : \Psi(A_\mathbf{e}(\epsilon_0^{-1},\epsilon_0t^{-1/2}))\to \Phi(A_\mathbf{e}(\epsilon_0^{-1}\sqrt{t},\epsilon_0)).$$
Consider $\chi:\mathbb{R}^+\to\mathbb{R}^+$, a $C^\infty$ cut-off function supported on $[0,2]$ and equal to $1$ on $[0,1]$.
    
\begin{defn}[Naïve gluing of an ALE space to an orbifold]\label{def naive desing}
    We define a \emph{naïve gluing} of $(N,g_b)$ at scale $0<t<\epsilon^4$ to $(M_o,g_o)$ at the singular point $p$, which we will denote $(M_o\#N,g_o\#_{p,t}g_b)$ by putting $g_o\#_{p,t}g_b=g_o$ on $M\backslash U$, $g_o\#_{p,t}g_b=tg_b$ on $K$, and
    $$g_o\#_{p,t}g_b =  \chi(t^{-\frac{1}{4}}r_e)\Psi_\infty^*g_b + \big(1-\chi(t^{-\frac{1}{4}}r_e)\big)\Phi^*g_o$$
    on $A_\mathbf{e}(\epsilon^{-1}\sqrt{t},2\epsilon)$.
\end{defn}

\begin{defn}[Function $r_D$ on a naïve desingularization]\label{rD}
    On a naïve $(M,g^D)$, we define the function smooth function $r_D$ in the following way:
    \begin{enumerate}
        \item $r_D = r_o$ on $M_o(\epsilon_0)$ defined in \eqref{eq Moepsilon},
        \item $r_D =\sqrt{t_j}r_{b_j}$ on each $N_j(\epsilon_0^{-1})$ defined in \eqref{eq Njepsilon} and
        \item $r_D=r_\mathbf{e}$ on $A_\mathbf{e}(\epsilon_0^{-1}\sqrt{t},\epsilon_0)$.
    \end{enumerate}
\end{defn}

Let us fix a notation for the desingularization by Eguchi-Hanson metrics.

\begin{defn}[Naïve desingularization by Eguchi-Hanson metrics]\label{desing eh}
    Let $(M_o,\mathbf{g}_o)$ be a compact Einstein orbifold with integrable Einstein deformations (Definition \ref{einstein integrable deformations}) and $\mathbb{R}^4\slash\mathbb{Z}_2$ singularities at points $j\in S$, and let $v\in \mathbf{O}(\mathbf{g}_o)$ and $\zeta=(\zeta_j^{\epsilon_j})_j$ with $\zeta_j\in\mathbb{R}^3\backslash\{0\}$ and $\epsilon_j\in \{+,-\}$.
    We define $g^D_{\zeta,v}$ the naïve gluing (as in Definition \ref{def naive desing}) of $\mathbf{eh}_{\zeta_j^{\epsilon_j}/|\zeta_j|}$ to $\mathbf{g}_v$ at scale $|\zeta_j|$ using the ALE coordinates of \eqref{eh zeta+} and \eqref{eh zeta-}.
\end{defn}

\begin{defn}[Approximate obstructions $\tilde{\mathbf{O}}(g^D_{\zeta,v})$]\label{def tilde O gDv zeta}
        We define $\tilde{\mathbf{O}}(g^D_{\zeta,v})$ the linear space spanned by the infinitesimal variations of $(v,\zeta)\mapsto g^D_{\zeta,v}$.
    \end{defn}
    \begin{rem}\label{rem OgD ozu}
    The space $\tilde{\mathbf{O}}(g^D_{\zeta,v})$ is very close to the approximate obstruction space of \cite{ozu2} and even yields better estimates. The results of \cite{ozu2} therefore hold when using $\tilde{\mathbf{O}}(g^D_{\zeta,v})$ instead of the similar space $\tilde{\mathbf{O}}(g^D)$ of \cite{ozu2}.
\end{rem}
    
We next define the notion of \emph{nondegenerate} sequence of naïve desingularizations in the spirit of \cite{br15} (see also the notion of \emph{generic} smoothings in \cite{spo}). It intuitively means that the desingularization sequence reaches an element $\partial_o\mathbf{E}(M)$ without being too tangent to $\partial_o\mathbf{E}(M)$.
\begin{defn}[Nondegenerate naïve desingularization]\label{non degenerate}
    We will say that a sequence of naïve desingularizations $(M,g^D_{t_n,v_n})$ of a compact Einstein orbifold $(M_o,\mathbf{g}_o)$ is \emph{nondegenerate} if, denoting $t_{n,\max} = \max t_{n,j}$ and $t_{n,\min} = \min t_{n,j}$, we have
    \begin{enumerate}
        \item $\lim_{n\to +\infty}\frac{t_{n,\min}}{t_{n,\max}}>0$ and
        \item $\frac{\|v_{n}\|^2_{L^2(\mathbf{g}_o)}}{t_{n,\min}}\to 0$.
    \end{enumerate}
\end{defn}
This technical definition essentially means that the gluing scales are comparable and the Einstein deformations of the orbifold are not too large. A stronger notion is that of \emph{transverse desingularization}.

\begin{defn}[Transverse naïve desingularization]\label{transverse}
We will say that a sequence of naïve desingularizations $(M,g^D_{t_n,v_n})$ of a compact Einstein orbifold $(M_o,\mathbf{g}_o)$ is \emph{transverse} if we have
\begin{enumerate}
        \item $\lim_{n\to +\infty}\frac{t_{n,\min}}{t_{n,\max}}>0$ and
        \item $\frac{\|v_{n}\|_{L^2(\mathbf{g}_o)}}{t_{n,\min}}\to 0$.
    \end{enumerate}
\end{defn}
\begin{rem}
    Usual desingularizations are transverse and often even \emph{orthogonal}, with $v_n = 0$. The desingularization of \cite{biq1} is a notable exception.
\end{rem}

Let us finally recall that according to \cite{ozu1,ozu2}, for any Einstein orbifold $(M_o,\mathbf{g}_o)$, and any $\delta>0$, there exists $\epsilon = \epsilon((M_o,\mathbf{g}_o),\delta)>0$ such that if an Einstein metric $(M,\textbf{g})$ satisfies 
$$ d_{GH}\big((M_o,\mathbf{g}_o),(M,\textbf{g})\big) <\epsilon, $$
then, there exists a naïve desingularization $(M,g^D)$ of $(M_o,\mathbf{g}_o)$ and a diffeomorphism $\phi:M\to M$ with
$$\|\phi^*\textbf{g}-g^D\|_{C^{2,\alpha}_{\beta,*}(g^D)}<\delta,$$
where the weighted Hölder norms $ C^{2,\alpha}_{\beta,*}(g^D) $ are defined in Section \ref{function spaces} in the Appendix.

We finally define the notions of \emph{nondegenerate} and \emph{transverse} sequence of Einstein metrics desingularizing the orbifold $(M_o,\mathbf{g}_o)\in \partial_o\mathbf{E}(M)$.
\begin{defn}[Nondegenerate and transverse Einstein desingularization]\label{non degenerate and transverse einstein}
    Let $(M^4,\mathbf{g}_n)_{n\in\mathbb{N}}$ be a sequence of Einstein metrics $d_{GH}$-converging to a compact Einstein orbifold $(M_o,\mathbf{g}_o)\in \partial_o\mathbf{E}(M^4)$. By \cite{ozu1,ozu2}, there exists a sequence of naïve desingularizations $(g^D_{t_n,v_n})_n$ such that up to acting on $\mathbf{g}_n$ by a diffeomorphism, we have
        \begin{enumerate}
            \item $\|\mathbf{g}_n-g^D_{t_n,v_n}\|_{C^{2,\alpha}_{\beta,*}(g^D_{t_n,v_n})}\leqslant \epsilon$ for some $\epsilon = \epsilon(\mathbf{g}_o)>0$ determined in \cite{ozu2},
            \item $\mathbf{g}_n$ is in reduced divergence-free gauge (defined in \cite{ozu2}) with respect to $g^D_{t_n,v_n}$            \item $\mathbf{g}_n-g^D_{t_n,v_n}$ is $L^2(g^D_{t_n,\zeta_n})$-orthogonal to $\tilde{\mathbf{O}}(g^D_{t_n,\zeta_n})$ of Definition \ref{def tilde O gDv zeta}
        \end{enumerate}
        
        We will say that $\mathbf{g}_n$ is a \emph{nondegenerate} (respectively \emph{transverse}) Einstein desingularization of $(M_o,\mathbf{g}_o)$ if the sequence $(g^D_{\zeta_n,v_n})_n$ is \emph{nondegenerate} (respectively \emph{transverse}) in the sense of Definitions \ref{non degenerate} and \ref{transverse}.
\end{defn}
\begin{figure}[hbt!]
\vspace{-5pt}
\caption{A \textbf{nondegenerate} desingularization $(M,\mathbf{g}_n)_{n\in\mathbb{N}}$ approaches $(M_o,\mathbf{g}_o)$ without being ``too tangent'' to the boundary $\partial_o\mathbf{E}(M)$: it stays in the plain green region.}\label{figure non degenerate} 
\begin{center}
\vspace{-13pt}
 \includegraphics[trim=0 180 0 180, clip, scale=.3]{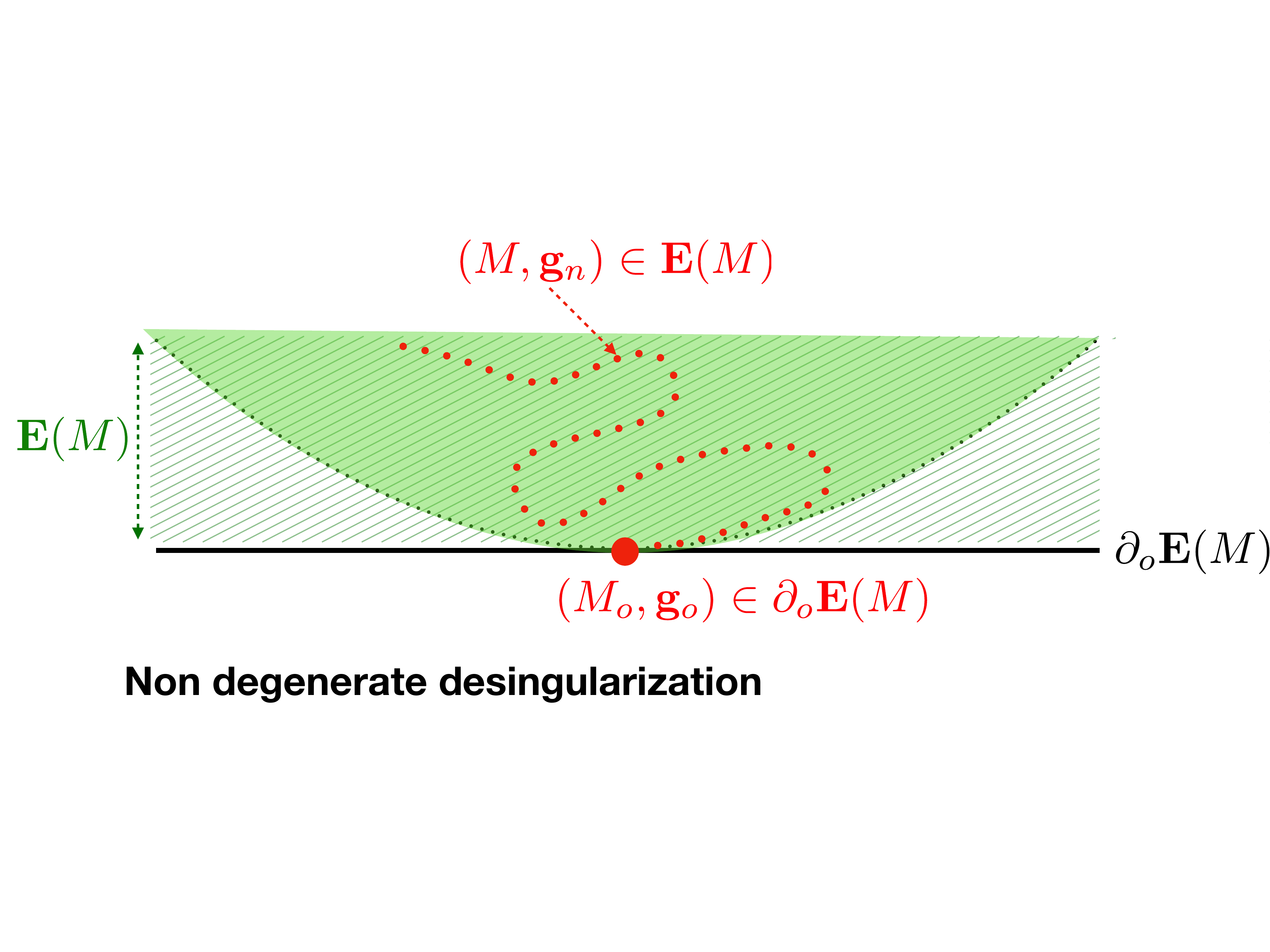} 
 \vspace{-13pt}
\end{center}
\end{figure}
\begin{figure}[hbt!]
\caption{A \textbf{transverse} desingularization $(M,\mathbf{g}_n)_{n\in\mathbb{N}}$ becomes almost orthogonal to the boundary $\partial_o\mathbf{E}(M)$ as it approaches $(M_o,\mathbf{g}_o)$: it stays in the plain green region.}\label{figure transverse}
\begin{center}
\vspace{-13pt}
 \includegraphics[trim=0 200 0 180, clip, scale=.3]{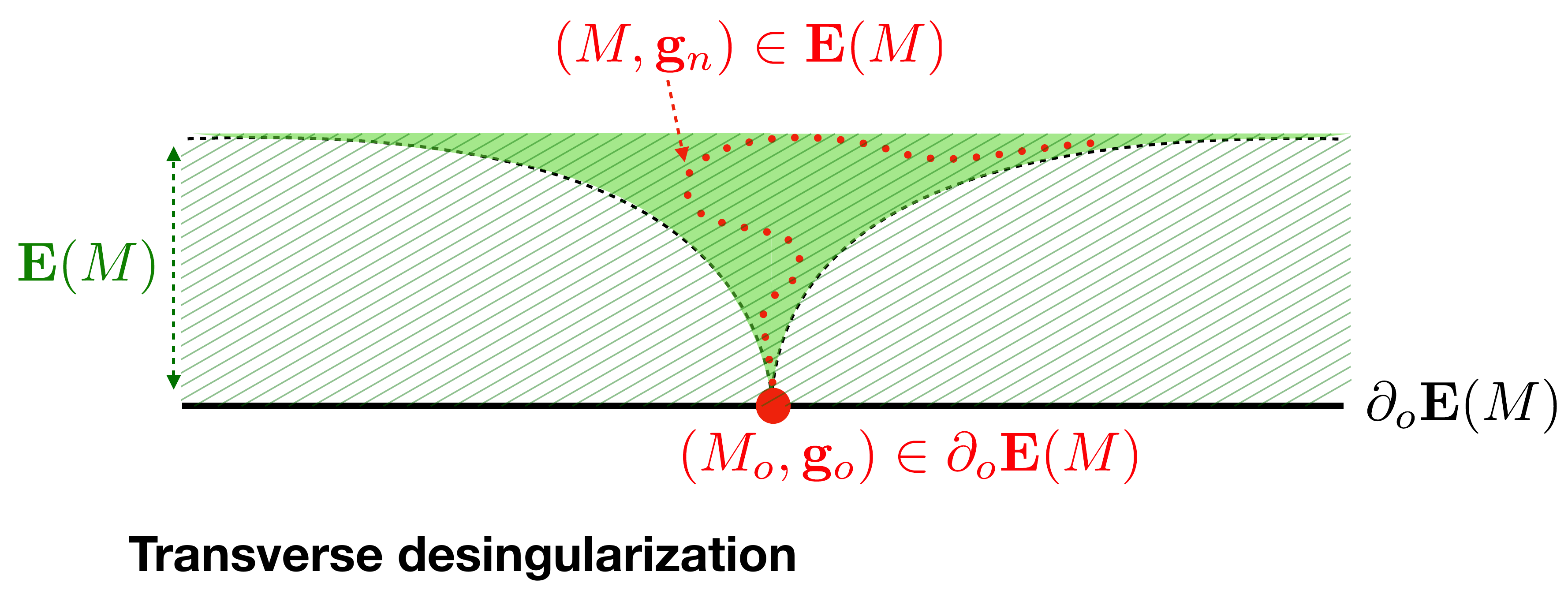}
 \vspace{-15pt}
\end{center}
\end{figure}
\newpage
\section{New obstructions to the desingularization of Einstein orbifolds}\label{section new section}

    We first deduce a new type of obstruction to the Gromov-Hausdorff desingularization of Einstein orbifolds which differs from the ones of \cite{biq1,biq2,ozu2}. The goal of this section is to prove Theorem \ref{2e obst} and Proposition \ref{2e obst nondeg}. 
    
    This obstruction is special to the compact case, or at least not present in the asymptotically hyperbolic (AH) context if one allows deformations of the conformal infinity (which are infinitely many new degrees of freedom which do not add obstructions). 
    
    \begin{exmp}\label{adstaubbolt}
        The AdS Taub-Bolt family of metrics constructed by Page-Pope \cite{pp} are asymptotically hyperbolic (AH) Einstein metrics and converge to an AH selfdual Einstein orbifold while bubbling out an Eguchi-Hanson metric (see \cite[Chapter B]{ozuthese} for a precise description of the degeneration). This orbifold has only one singularity $\mathbb{R}^4\slash\mathbb{Z}_2$ at which we have $\Ric^0= 0$ because it is Einstein, and in the usual bases of $\Omega^\pm$, we have
$$\mathbf{R}^- = \begin{bmatrix}
     -1&0 & 0 \\
     0 & -1 & 0\\
     0&0&-1
  \end{bmatrix},$$ because it is selfdual, and finally $$\mathbf{R}^+ = \begin{bmatrix}
     0&0&0 \\
     0&-\frac{3}{2}&0\\
     0&0&-\frac{3}{2}
  \end{bmatrix},$$
    only has a one-dimensional kernel. 
    
    The AdS Taub-Bolt orbifold is $L^2$-rigid. If we consider the Einstein deformations induced by the conformal changes at infinity, then, there exists an infinite dimensional space of \emph{nondegenerate} Einstein desingularizations, see \cite{biq1}. This shows that the compactness assumption in Proposition \ref{2e obst nondeg} is necessary.
    \end{exmp}

    \subsection{Degeneration and obstructions}

    Let $(M_o,\mathbf{g}_o)$ be a compact Einstein orbifold with integrable Einstein deformations (Definition \ref{einstein integrable deformations}) and with a $\mathbb{R}^4\slash\mathbb{Z}_2$ singularity at $p\in M_o$, and let $(N,\mathbf{eh})$ be an Eguchi-Hanson metric. At $p$, we can choose orbifold coordinates (Definition \ref{orb Ein}) in which the metric $\mathbf{g}_o$ is in Bianchi gauge with respect to the flat metric $\mathbf{e}$ (see \cite{ozu2} for a proof in the more difficult case of neck regions), that is satisfies $B_\mathbf{e}\mathbf{g}_o=0$. Since the metric $\mathbf{g}_o$ is smooth in the above orbifold coordinates and because of the $\mathbb{Z}_2$-invariance, we have the following development,
    \begin{equation}
        \mathbf{g}_o = \mathbf{e} + H_2+H_4+\mathcal{O}(r_\mathbf{e}^6),\label{dvp go}
    \end{equation}
    where the $H_i$ are homogeneous symmetric $2$-tensors with $|H_i|_{\mathbf{e}}\sim r_\mathbf{e}^i$ with $B_\mathbf{e} H_2 = 0$. At the infinity of $N$, in the coordinates of \eqref{eguchihanson}, we have the following development,
    \begin{equation}
        \mathbf{eh} = \mathbf{e} + H^4+ \mathcal{O}(r_\mathbf{e}^{-8}),\label{dvp gEH}
    \end{equation}
    where $|H^4|_{\mathbf{e}}\sim r_\mathbf{e}^{-4}$ and with $B_eH^4=0$.
    
        For a metric $g$, let us denote $Q_g^{(2)}$ the bilinear terms given by the second derivative of $$h\mapsto \mathbf{\Phi}_g(h) :=\Ric(h)+\delta^*_hB_gh$$ at $g$ (whose linearization is $P_g$). According to \cite[Section 3, Section 10 and Lemma 14.1]{biq1}, for any symmetric $2$-tensors $ |H_2|_\textbf{e} \sim r_\mathbf{e}^2$ with $P_\mathbf{e}H_2 = \Lambda \mathbf{e}$, and $ |H_4|_\textbf{e} \sim r_\mathbf{e}^4$ with $P_\mathbf{e}H_4 = \Lambda H_2 - Q_\mathbf{e}^{(2)}(H_2,H_2)$, there exist 
        \begin{enumerate}
            \item a \emph{unique} symmetric $2$-tensor $\underline{h}_2$ and reals $(\lambda_k)_{k\in\{1,2,3\}}$ on $N$ satisfying:
            \begin{equation}
  \left\{
      \begin{aligned}
        P_{\mathbf{eh}}\underline{h}_2&= \Lambda \mathbf{eh}+\sum_k \lambda_k \mathbf{o}_{k},\\
        \underline{h}_2 &= H_2 + \mathcal{O}(r^{-2}),\\
        \int_{\mathbb{S}^2}\langle\underline{h}_2, \mathbf{o}_k\rangle_{\mathbf{eh}} dv_{\mathbf{eh}_{|\mathbb{S}^2}}&= 0 \text{ for all  } k\in\{1,2,3\}. 
      \end{aligned}
    \right.\label{def lambda k}
\end{equation}

            \item and a symmetric $2$-tensor $\underline{h}_4$ and real numbers $(\mu_k)_{k\in\{1,2,3\}}$ on $N$ satisfying:
            \begin{equation}
            \left\{
      \begin{aligned}
        P_{\mathbf{eh}}\underline{h}_4&= \Lambda \underline{h}_2- Q_{\mathbf{eh}}^{(2)}(\underline{h}_2,\underline{h}_2)+\sum_k \mu_k\mathbf{o}_{k},\\
        \underline{h}_4 &= H_4 + \mathcal{O}(r^{\epsilon}),\\
        \int_{\mathbb{S}^2}\langle\underline{h}_4, \mathbf{o}_k\rangle_{\mathbf{eh}} dv_{\mathbf{eh}_{|\mathbb{S}^2}}&= 0 \text{ for all  } k\in\{1,2,3\}.
      \end{aligned}
    \right.\label{def mu k}
\end{equation}
        for all $\epsilon>0$.
\end{enumerate}

\begin{rem}
    The above tensor $\underline{h}_4$ is not unique since there are $2$-tensors which are $\mathcal{O}(1)$ at infinity in the kernel of $P_\mathbf{eh}$. There will be a best choice of $\underline{h}_4$ in our construction which will ensure that the asymptotics on the orbifold and the ALE match, see Remark \ref{H^4_4 terme de h_4}.
\end{rem}

\subsubsection{More precision on the first obstruction}

Let $\zeta\in\mathbb{R}^3\backslash \{0\}$, and denote $\mathbf{o}_k(\zeta^\pm)$ the basis of $\mathbf{O}(\mathbf{eh}_{\zeta^\pm})$ of \eqref{dvp obst avec zeta}. Since the metric $\mathbf{eh}_{\zeta^\pm}$ is homothetic to $\mathbf{eh}$, we also find solutions $\underline{h}_2$ and $\underline{h}_4$ to
\begin{enumerate}
            \item a unique symmetric $2$-tensor $\underline{h}_2(\zeta^\pm)$ on $N$ satisfying the following equations
            \begin{equation}
  \left\{
      \begin{aligned}
        P_{\mathbf{eh}_{\zeta^\pm}}\underline{h}_2(\zeta^\pm)&= \Lambda \mathbf{eh}_{\zeta^\pm}+\sum_k \lambda_k(\zeta^\pm) \mathbf{o}_{k}(\zeta^\pm),\\
        \underline{h}_2(\zeta^\pm) &= H_2 + \mathcal{O}(r^{-2}),\\
        \int_{\mathbb{S}^2}\langle\underline{h}_2(\zeta^\pm), \mathbf{o}_k(\zeta^\pm)\rangle_{\mathbf{eh}_{\zeta^\pm}} &dv_{\mathbf{eh}_{\zeta^\pm|\mathbb{S}^2}}= 0 \text{ for all  } k\in\{1,2,3\}. 
      \end{aligned}
    \right.\label{def lambda k zeta}
\end{equation}

            \item and a symmetric $2$-tensor $\underline{h}_4(\zeta^\pm)$ on $N$ satisfying the following equations
            \begin{equation}
            \left\{
      \begin{aligned}
        P_{\mathbf{eh}_{\zeta^\pm}}\underline{h}_4(\zeta^\pm)&= \Lambda \underline{h}_2(\zeta^\pm)- Q_{\mathbf{eh}_{\zeta^\pm}}^{(2)}(\underline{h}_2(\zeta^\pm),\underline{h}_2(\zeta^\pm))+\sum_k \mu_k(\zeta^\pm) \mathbf{o}_{k}(\zeta^\pm),\\
        \underline{h}_4(\zeta^\pm) &= H_4 + \mathcal{O}(r^{\epsilon}),\\
        \int_{\mathbb{S}^2}\langle\underline{h}_4(\zeta^\pm), \mathbf{o}_k&(\zeta^\pm)\rangle_{\mathbf{eh}_{\zeta^\pm}} dv_{\mathbf{eh}_{\zeta^\pm|\mathbb{S}^2}}= 0 \text{ for all  } k\in\{1,2,3\}. 
      \end{aligned}
    \right.\label{def mu k zeta}
\end{equation}
        for all $\epsilon>0$.
\end{enumerate}
We moreover have the following values for the $\lambda_k(\zeta^\pm)$: for some $c>0$,
\begin{equation}
    \lambda_k(\zeta^\pm) = c\left\langle\mathbf{R}^\pm(H_2)\zeta^\pm,\zeta^\pm_{(k)} \right\rangle,\label{valeur obst lambda zeta}
\end{equation}
where $(\zeta^\pm_{(k)})_{k\in \{1,2,3\}}$ is an orthogonal basis with $|\zeta^\pm_{(k)}|=\sqrt{2}$ and $\zeta^\pm_{(1)}=\zeta^\pm/|\zeta|$. 
\\

We can be more precise about the first obstruction and the symmetric $2$-tensor $\underline{h}_2$.
\begin{lem}\label{def H42}
    Let $\underline{h}_2$ be one of the above solution of \eqref{def lambda k zeta} (where we omit the dependence in $\zeta^\pm$ for simpler equations). Then, we have
    \begin{equation}
        \underline{h}_2 = H_2+H^4_2+ \mathcal{O}(r^{-4+\epsilon}) \text{ for any }\epsilon>0,
    \end{equation}
    where $H^4_2$ is a homogeneous solution to
    \begin{equation}
        P_\mathbf{e}H^4_2 = -Q_\mathbf{e}^{(2)}(H^4,H_2) + \Lambda H^4+\sum_k \lambda_k O_{k}^4, \text{ with } |H^4_2|_\mathbf{e} \sim r_\mathbf{e}^{-2}.\label{equation H^4_2}
    \end{equation}
    Moreover, when $\lambda_k = 0$, then $H^4_2 = 0$.
\end{lem}
\begin{proof}
    At infinity, in ALE coordinates we have $\underline{h}_2 = H_2 + h'$ with $h' = \mathcal{O}(r_\mathbf{e}^{-2+\epsilon})$ for all $\epsilon>0$. Since $\mathbf{eh} = \mathbf{e}+ H^4 + \mathcal{O}(r_\mathbf{e}^{-8})$, we develop the equation \eqref{def lambda k} as:
    \begin{align*}
        P_{\mathbf{e}}H_2 + P_\mathbf{e}(h') + Q_\mathbf{e}^{(2)}(H^4,H_2) + \mathcal{O}(r_\mathbf{e}^{-8+\epsilon}) = \Lambda \mathbf{e} + \Lambda H^4 +\sum_k \lambda_k O_{k}^4 +  \mathcal{O}(r_\mathbf{e}^{-8+\epsilon}),
    \end{align*}
    and using the assumption that $P_{\mathbf{e}}H_2 = \Lambda \mathbf{e}$, we find 
    \begin{align*}
        P_\mathbf{e}(h') + Q_\mathbf{e}^{(2)}(H^4,H_2)= \Lambda H^4 +\sum_k \lambda_k O_{k}^4 +  \mathcal{O}(r_\mathbf{e}^{-8+\epsilon}).
    \end{align*}
    which is the stated equation. Now let us consider $H^2$ with $|H^2|_\mathbf{e}\sim r_\mathbf{e}^{-2}$ such that 
    $$P_e H^2 = -Q_\mathbf{e}^{(2)}(H^4,H_2) + \Lambda H^4+\sum_k \lambda_k O_{k}^4.$$
    It exists by Lemma \ref{resolution equation laplacien} in the appendix (one checks that there is no $\log$ term). Then, we see that $P_\mathbf{e}(h'- H^2) = \mathcal{O}(r_\mathbf{e}^{-8+\epsilon})$ while $h' = \mathcal{O}(r_\mathbf{e}^{-2+\epsilon})$. By the theory of elliptic operators between weighted Hölder spaces (see for instance \cite{pr} or \cite{lm}), we therefore find that $h'- H^2 = \frac{K_0}{r_\mathbf{e}^2} + \mathcal{O}(r_\mathbf{e}^{-6+\epsilon})$ for a constant symmetric $2$-tensor $K_0$, and consequently
    $$ h' = H^4_2+\mathcal{O}(r_\mathbf{e}^{-4+\epsilon}) = H^2 + \frac{K_0}{r_\mathbf{e}^2}+\mathcal{O}(r_\mathbf{e}^{-4+\epsilon}) $$
    which proves the stated result.
    \\
    
    Let us now prove that if the obstructions $\lambda_k$ of \eqref{def lambda k} vanish, then $H^4_2=0$.
    
    The idea is to use the uniqueness of $ \underline{h}_2 $, and show that $\mathcal{L}_{r\partial_r}\underline{h}_2$ satisfies the same equation. Indeed, we have: 
    \begin{equation}
    \begin{aligned}
        0=\mathcal{L}_{r\partial_r} \big(P_\mathbf{eh}  \underline{h}_2 - \Lambda \mathbf{eh}\big)&= Q_\mathbf{eh}^{(2)}(\mathcal{L}_{r\partial_r}\mathbf{eh}, \underline{h}_2) + P_\mathbf{eh}\mathcal{L}_{r\partial_r}\underline{h}_2 - \Lambda \mathcal{L}_{r\partial_r}\mathbf{eh},\\
        &= P_\mathbf{eh}\mathcal{L}_{r\partial_r}\underline{h}_2 - 4 \Lambda\mathbf{eh}\\
        &\;+Q_\mathbf{eh}^{(2)}(\mathcal{L}_{r\partial_r}\mathbf{eh}-2\mathbf{eh}, \underline{h}_2)-\Lambda(\mathcal{L}_{r\partial_r}\mathbf{eh}-2\mathbf{eh})
    \end{aligned}\label{Lrdr Peh h2}
    \end{equation}
    \begin{rem}
        In the last line, we used the fact that $Q_\mathbf{eh}^{(2)}(\mathbf{eh}, \underline{h}_2) + \Lambda \mathbf{eh} =0$, which comes from the (Bianchi-free version of) the identity $\Ric(s g) = \Ric(g)$ as $2$-tensors for all $s>0$. This implies $D_g\Ric(g) = 0$ for all $g$ and differentiating again at $g$ in the direction $h$ gives:
        $$ D^2_g\Ric(h,g) + D_g\Ric(h) = 0. $$
        We use this formula with $g = \mathbf{eh}$, $h = \underline{h}_2$ and use the assumption that $D_\mathbf{eh}\Ric(\underline{h}_2) = P_{\mathbf{eh}}\underline{h}_2 = \Lambda \mathbf{eh}$ because the $\lambda_k$ vanish.
    \end{rem}
   In \eqref{Lrdr Peh h2}, we recognize $-2 \mathbf{o}_1 = \mathcal{L}_{r\partial_r}\mathbf{eh}-2\mathbf{eh}$.
    Let us show that the last term vanishes. Using \cite[Lemme 3]{biq2}, we find for $c\neq 0$, $$Q_\mathbf{eh}^{(2)}(\mathbf{o}_1, \underline{h}_2) - \Lambda \mathbf{o}_1 = c\sum_j R_{1j}^+(H_2)\mathbf{o}_j$$ and by \cite{biq1}, if $\lambda_1=\lambda_2=\lambda_3=0$, then, $ R_{11}^+(H_2) = R_{12}^+(H_2)=R_{13}^+(H_2)=0$, hence, $$Q_\mathbf{eh}^{(2)}(\mathbf{o}_1, \underline{h}_2) - \Lambda \mathbf{o}_1 = 0.$$
    There remains $P_\mathbf{eh}\mathcal{L}_{r\partial_r}\underline{h}_2 = 4\Lambda \mathbf{eh}$ in \eqref{Lrdr Peh h2}. Since $ \mathcal{L}_{r\partial_r}\underline{h}_2 \sim \mathcal{L}_{r\partial_r}H_2 = 4 H_2$ at infinity, we find the same first two equalities of \eqref{def lambda k} up to a constant $4$. Since $4\underline{h}_2$ is the unique solution up to $\mathbf{O}(\mathbf{eh})$, we find that $ \mathcal{L}_{r\partial_r}\underline{h}_2 = 4\underline{h}_2 +\mathcal{O}(r_\mathbf{e}^{-4}) $.
    
    The point is that now $ \mathcal{L}_{r\partial_r} H^4_2 = 0$ hence $\mathcal{L}_{r\partial_r}\underline{h}_2 = 4H_2 + \mathcal{O}(r_\mathbf{e}^{-4})$, and consequently $ \underline{h}_2 = H_2 + \mathcal{O}(r_\mathbf{e}^{-4})$.
    
\begin{rem}
    Let us note that the vanishing of $H^4_2$ is dependent on the coordinate system we consider for Eguchi-Hanson metric. It seems that the crucial property if these coordinates we use here is that the vector field $r\partial_r$ is harmonic. Since we will always take $\mathbf{eh}$ in the coordinates of \eqref{eh zeta+}, we will always have $H^4_2=0$ as long as the $\lambda_k$ vanish.
\end{rem}

    \end{proof}
    
    Later on, we will consider deformations of our orbifold as well and we will need to consider metrics which are not quite asymptotic to our target Euclidean metric, but a nearby one, i.e. $\mathbf{e}+H_0$ for some constant $2$-tensor $H_0$.

Similarly, let us consider coordinates in which we have $\mathbf{g}_o=\mathbf{e}+H_0+H_2+\mathcal{O}(r_\mathbf{e}^4)$ for $H_0$ a traceless constant $2$-tensor and let us denote $\psi:\mathbb{R}^4\slash\mathbb{Z}_2\to\mathbb{R}^4\slash\mathbb{Z}_2$ a linear isomorphism for which $\psi^*\mathbf{e}=\mathbf{e}+H_0$, in order to extend $H_2$ to $\psi^*\mathbf{eh}_{\zeta^\pm}$. 
\begin{rem}\label{meaning diffeo}
    Here, by $\psi^*\mathbf{eh}_{\zeta^\pm}$, we really mean the same diffeomorphism applied to $\mathbf{eh} = \mathbf{e} + \sum_{k\geqslant 2}\tilde{H}_k(\theta)r^{-2k}$ in the polar coordinates $\theta\in \mathbb{RP}^3$ and $r\in [0,+\infty)$ as in \cite[Proposition 3.14]{kro}, where $(\mathbb{R}^4\slash\mathbb{Z}_2,\mathbf{e})= (\mathbb{R}^+\times \mathbb{RP}^3, dr^2 + dr^2g_{\mathbb{RP}^3})$. The metric has to be ``closed-up'' by a different $\mathbb{S}^2$ at $\psi(0)=0$.
\end{rem}
Then the obstructions are:
\begin{equation}
    \lambda_k(\zeta^\pm) = c\left\langle\mathbf{R}^\pm(\psi_*H_2)\zeta^\pm,\zeta^\pm_{(k)} \right\rangle. \label{valeur lambda k psi zeta}
\end{equation}
    
    We will moreover need to extend constant $2$-tensors to our Eguchi-Hanson metrics.
    \begin{lem}\label{def h0}
        Let $H_0$ be a constant $2$-tensor on $\mathbb{R}^4\slash \mathbb{Z}_2$, i.e. constant on $\mathbb{R}^4$. Then, there exists a unique $2$-tensor $\underline{h}_0$ on $T^*\mathbb{S}^2$ satisfying:
        \begin{equation}
            \left\{\begin{aligned}
                &\;P_\mathbf{eh} \underline{h}_0 = 0,\\
                &\;\;\underline{h}_0 = H_0 +H^4_0+ \mathcal{O}(r^{-6+\epsilon}) \text{ for all } \epsilon>0\\
                &\;\;\underline{h}_0 \perp_{L^2(\mathbf{eh})}\mathbf{O}(\mathbf{eh}).
            \end{aligned}\right.
        \end{equation}
    \end{lem}
    \begin{proof}
    Let $H_0 = \mathcal{L}_{X_1}\mathbf{e}$ be a constant symmetric $2$-tensor with $X_1$ a linear vector field. Then, since $P_\mathbf{e}H_0=0$, for $\chi$ a cut-off function supported in a neighborhood of infinity of $\mathbf{eh}$, we have 
    $P_\mathbf{eh}(\chi H_0) = \mathcal{O}(r_\mathbf{e}^{-6})$, and an integration by parts against the $\mathbf{o}_k$ proves that $P_\mathbf{eh}(\chi H_0)\perp \mathbf{O}(\mathbf{eh})$ and there is no obstruction to finding $ h'' $ with $P_\mathbf{eh}(\chi H_0+h'')= 0$ and $ h''= \mathcal{O}(r_\mathbf{e}^{-4+\epsilon})$ for all $\epsilon>0$, see \cite[Proposition 2.1]{biq1}. 
    This lets us find a solution $\underline{h}_0$ to 
    \begin{equation}
        P_\mathbf{eh}\underline{h}_0=0, \text{ with } \underline{h}_0 = H_0 + \mathcal{O}(r_\mathbf{e}^{-4}).
    \end{equation}
    with $B_\mathbf{eh}\underline{h}_0=0$, $\textup{tr}_\mathbf{eh}\underline{h}_0=0$ by the maximum principle since $\textup{tr}_\mathbf{eh}P_\mathbf{eh} = \frac{1}{2}\nabla_\mathbf{eh}^*\nabla_\mathbf{eh}\textup{tr}_\mathbf{eh}$ and $B_\mathbf{eh}P_\mathbf{eh} = B_\mathbf{eh}\delta_\mathbf{eh}^*B_\mathbf{eh}=\frac{1}{2}\nabla_\mathbf{eh}^*\nabla_\mathbf{eh}B_\mathbf{eh}$. We can make $\underline{h}_0$ unique by additionally assuming: $\underline{h}_0\perp \mathbf{O}(\mathbf{eh})$
    and we find a development:
    \begin{equation}
        \underline{h}_0=H_0+H_0^4+\mathcal{O}(r_\mathbf{e}^{-6+\epsilon}) \text{ with } |H_0^4|_\mathbf{e}\sim r_\mathbf{e}^{-4}.\label{asympt H_0^4}
    \end{equation}
    The $L^2$-product between $\underline{h}_0$ and $\mathbf{o}\in\mathbf{O}(\mathbf{eh})$ is a priori ill-defined because $\mathbf{o}=\mathcal{O}(r^{-4})$ and $\underline{h}_0 = \mathcal{O}(1)$, but since the coefficients of the leading terms restricted to spheres $\{ r = cst\} $ belong to different eigenspaces of the Laplacian on $\mathbb{RP}^3$, this is therefore well-defined.
    \begin{rem}
        We can prove that we have $\underline{h}_0 = \mathcal{L}_{\underline{x}_1} \mathbf{eh}$ for some vector field $\underline{x}_1$ satisfying 
        $0=B_{\mathbf{eh}}\mathcal{L}_{\underline{x}_1}\mathbf{eh} = \nabla_{\mathbf{eh}}^*\nabla_{\mathbf{eh}} \underline{x}_1$ and $\underline{x}_1 = X_1 + \mathcal{O}(r_\mathbf{e}^{-3+\epsilon})$. We however do not have any application for this at the moment.
    \end{rem}
    The link between the next asymptotic term $H^4_2$ of the Einstein modulo obstructions deformation $\underline{h}_2 = H_2 + H^4_2 + \ldots$ of Lemma \ref{def H42} and the development of $\underline{h}_0 = H_0 + H_0^4 + \ldots$ is given by the following integration by parts:
    \begin{align}
        0=&\int_N\big\langle P_{\mathbf{eh}}\underline{h}_2- \Lambda \mathbf{eh}-\sum_k \lambda_k \mathbf{o}_{k},\underline{h}_0 \big\rangle_\mathbf{eh}dv_\mathbf{eh}\nonumber\\
        =& \int_N\langle P_{\mathbf{eh}}\underline{h}_2,\underline{h}_0 \rangle_\mathbf{eh}dv_\mathbf{eh}\nonumber\\
        =& \int_N\langle \underline{h}_2,P_{\mathbf{eh}}\underline{h}_0 \rangle_\mathbf{eh}dv_\mathbf{eh}+\frac{1}{2}\lim_{r\to +\infty}\int_{\{r_\mathbf{e}=r\}}\langle \nabla_{\partial_{r_\mathbf{e}}} (\underline{h}_2),\underline{h}_0\rangle_{\mathbf{eh}} - \langle \underline{h}_2,\nabla_{\partial_{r_\mathbf{e}}} \underline{h}_0\rangle_{\mathbf{eh}}dv_{\{r_\mathbf{e}=r\}} \nonumber\\
        0=& -3\int_{\mathbb{S}^3\slash\mathbb{Z}_2}\langle H_2,H^4_0 \rangle_\mathbf{e}dv_{\mathbb{S}^3\slash\mathbb{Z}_2}-\int_{\mathbb{S}^3\slash\mathbb{Z}_2}\langle H^4_2,H_0 \rangle_\mathbf{e}dv_{\mathbb{S}^3\slash\mathbb{Z}_2},\label{egalite H24 H0}
    \end{align}
    where we used that by homogeneity, we have $\nabla_{\partial_{r_\mathbf{e}}} H_2 = \frac{2}{r_\mathbf{e}}H_2$, $\nabla_{\partial_{r_\mathbf{e}}} H_0 =0$, $\nabla_{\partial_{r_\mathbf{e}}} H^4_2 = -\frac{2}{r_\mathbf{e}}H^4_2$ and $\nabla_{\partial_{r_\mathbf{e}}} H^4_0 = -\frac{4}{r_\mathbf{e}}H^4_0$.
    
    Finally, if $\lambda_k = 0$, using the fact that $ H^4_2=0 $ from Lemma \ref{def H42}, we find: 
    \begin{equation}
        \int_{\mathbb{S}^3\slash\mathbb{Z}_2}\langle H_2,H^4_0 \rangle_\mathbf{e}dv_{\mathbb{S}^3\slash\mathbb{Z}_2}=0\label{orth H_2 H^0_4}
    \end{equation}
\end{proof}

\subsubsection{A first obstruction on the orbifold}

We have a similar result for the extension of tensors on the orbifold. Consider an infinitesimal Ricci-flat variation of $\mathbf{eh}_{\zeta_j^\pm}$ in the direction $\mathbf{o}_m=\mathbf{o}_m(\zeta_j^\pm)\in\mathbf{O}(\mathbf{eh}_{\zeta_j^\pm})$ where $\mathbf{o}_m(\zeta_j^\pm)$ is the $2$-tensor defined in \eqref{dvp obst avec zeta} and Remark \ref{rem deformation eh zeta et noyau}. In particular, $\mathbf{o}_m$ has a development 
$ \mathbf{o}_m=O_{m}^4+\mathcal{O}(|\zeta|^3r_\mathbf{e}^{-8}) $ with $|O_{m}^4|_\mathbf{e}\sim |\zeta|r^{-4}$.

\begin{lem}
    We define $\overline{\mathbf{o}}_m^4$ for $m\in\{1,2,3\}$ some solutions to the following equation
\begin{equation}
  \left\{
      \begin{aligned}
        P_o \overline{\mathbf{o}}_m^4 & \in \mathbf{O}(\mathbf{g}_o),\\
        \overline{\mathbf{o}}_m^4 &= O_{m}^4 +\mathcal{O}(r_o^{-2-\epsilon}) \text{ for any } \epsilon>0 \text{ at the singular point } j\\
        \overline{\mathbf{o}}_m^4&\perp \mathbf{O}(\mathbf{g}_o).\label{eq prol O4}
      \end{aligned}
    \right.
\end{equation}
They are determined up to the kernel of $P_{\mathbf{g}_o}$ in $\mathcal{O}(r_\mathbf{e}^{-2})$ at the singular points and this choice will not alter the result at this level of precision.
\end{lem}
\begin{proof}
Let us consider a cut-off function $\chi$ on $M_o$ supported in a small enough neighborhood of a singular point $p$. We have $P_{\mathbf{g}_o} (\chi O^4_m) = \mathcal{O}(r^{-4})$, hence there exists a $2$ tensor $|O^4_{m,2}| = \mathcal{O}(r^{-2-\epsilon})$ so that $$P_{\mathbf{g}_o} (\chi (O^4_m+O^4_{m,2})) = \mathcal{O}(r^{-2-\epsilon}).$$ Finally, using the Fredholm properties of $P_{\mathbf{g}_o}$ proven in \cite{ozu2}, there exists $|h'| =  \mathcal{O}(r^{-\epsilon})$ for all $\epsilon>0$ so that:
 $$P_{\mathbf{g}_o} (\chi (O^4_m+O^4_{m,2}) + h') \in \mathbf{O}(\mathbf{g}_o).$$
 
 Now, we can uniquely choose it by imposing $(\chi (O^4_m+O^4_{m,2}) + h')\perp_{L^2(\mathbf{g}_o)} \mathbf{O}(\mathbf{g}_o)$, where the $L^2$ product is a priori ill-defined as in the proof of Lemma \ref{def h0}, but makes sense via the orthogonality of the different eigenspaces of the Laplacian on $\mathbb{RP}^3$. 
\end{proof}

\begin{prop}\label{premiere obst orbifold 1 pt}
    Let us denote $(\mathbf{v}_l)_l$ an orthonormal basis of $\mathbf{O}(\mathbf{g}_o)$. Consider the development of the symmetric $2$-tensor $\underline{h}_2 = H_2+ H^4_2 + \mathcal{O}(r_\mathbf{e}^{-6+\epsilon})$ of \eqref{def lambda k zeta} (again, we will omit the dependence in $\zeta^\pm$).
    
    For $\zeta^\pm\in \Omega^\pm$ and the symmetric $2$-tensor $H^4=H^4(\zeta^\pm)$ with $P_\mathbf{e}H^4 = 0$ and $B_\mathbf{e}H^4 = 0$, there exists $(\overline{h}^4,(\nu_l)_l)$ with a symmetric $2$-tensor $\overline{h}^4$ and real numbers $(\nu_l)_l$ on $M_o$ satisfying the following equations
\begin{equation}
  \left\{
      \begin{aligned}
        P_o \overline{h}^4 &= \sum_l \nu_l\mathbf{v}_l+\sum_k \lambda_k\overline{\mathbf{o}}_k^4,\\
        \overline{h}^4 &= H^4 + H^4_2 + \mathcal{O}(|\zeta|^2r_o^{-\epsilon}) \text{ for any } \epsilon>0 \text{ at the singular point } j\\
        \overline{h}^4&\perp \mathbf{O}(\mathbf{g}_o).\label{eq prol H4}
      \end{aligned}
    \right.
\end{equation}

We moreover have the following value for the obstruction when $\lambda_k = 0$ for $k\in\{1,2,3\}$
    \begin{equation}
    \begin{aligned}
        \nu_l &=3 \int_{\mathbb{S}^3\slash\mathbb{Z}_2}\langle H^4,V_{l,2} \rangle_\mathbf{e} dv_{\mathbb{S}^3\slash\mathbb{Z}_2}.
    \end{aligned}\label{valeur nu l}
    \end{equation}
    where we have $\mathbf{v}_l = V_{l,0}+ V_{l,2}+ \mathcal{O}(r_\mathbf{e}^4)$ with $|V_{l,m}|_\mathbf{e} \sim r_\mathbf{e}^m$.
\end{prop}

\begin{rem}\label{H^4_4 terme de h_4}
    The next term in the development of $\overline{h}^4$ is $H^4_4$ with $|H^4_4|_\mathbf{e}=\mathcal{O}(r_\mathbf{e}^\epsilon+r_\mathbf{e}^{-\epsilon})$ for any $\epsilon>0$ (it will typically be logarithmic in $r_\mathbf{e}$) for which we have 
$$\overline{h}^4 = H^4 + H^4_2+H^4_4+ \mathcal{O}(|\zeta|^2r_o^{2-\epsilon}).$$
It satisfies 
\begin{equation}
    P_\mathbf{e} H_4^4 + Q_\mathbf{e}^{(2)}(H^4,H_4) + Q_\mathbf{e}^{(2)}(H^4_2,H_2)+Q_\mathbf{e}^{(3)}(H^4,H_2,H_2)=0.\label{def H^4_4}
\end{equation}
for $Q_\mathbf{e}^{(m)}$ the $m$-linear terms of the development of $h\mapsto(\Ric-\Lambda + \delta^*_{\mathbf{e}+h}B_\mathbf{e})(\mathbf{e}+h)$ around $h=0$. The equality \eqref{def H^4_4} is exactly the equation satisfied by the next term in the development of $\underline{h}_4$ defined in \eqref{def mu k} or \eqref{def mu k zeta}. That was what made $\underline{h}_4$ non unique: one could add any element of the kernel of $P_\mathbf{eh}$ asymptotic to a constant $2$-tensor to it. We can make it unique (and still existing) by imposing that we have for the above $H^4_4$: $$\underline{h}_4=H_4 +H^4_4+ \mathcal{O}(r^{-2+\epsilon}).$$
\end{rem}
\begin{proof}
    Consider $\chi$ a cut-off function supported in the neighborhood of the singular point $j\in M_o$. From $ P_\mathbf{e}H^4 =0 $,  $r_\mathbf{e}^k|\nabla_\mathbf{e}^k(\mathbf{g}_o-\mathbf{e})|_\mathbf{e}=\mathcal{O}(r_\mathbf{e}^2)$ and using \eqref{equation H^4_2}, we find $$ P_o(\chi (H^4+H^4_2))-\sum_{k}\lambda_k\overline{\mathbf{o}}^4_k \in r_o^{-2}C^\alpha_\beta(\mathbf{g}_o)$$ for $0<\beta<1$ where the norms are defined in the Appendix \ref{function spaces}.
    
    Moreover, by integration by parts and using $\overline{\mathbf{o}}^4_k\perp_{L^2(\mathbf{g}_o)} \mathbf{O}(\mathbf{g}_o)$, we find
    \begin{equation}
    \begin{aligned}
        \int_{M_o}\big\langle P_o&(\chi (H^4+H^4_2)),\mathbf{v}_l\big\rangle_{\mathbf{g}_o}dv_{\mathbf{g}_o}\\
        =&\int_{M_o}\big\langle \chi (H^4+H^4_2),P_o(\mathbf{v}_l)\big\rangle_{\mathbf{g}_o}dv_{\mathbf{g}_o}\\
        &+\frac{1}{2}\lim_{r\to 0}\int_{\{r_o=r\}}\langle \nabla_{\partial_{r_o}} (\chi H^4),\mathbf{v}_l\rangle_{\mathbf{g}_o} - \langle \chi H^4,\nabla_{\partial_{r_o}} \mathbf{v}_l\rangle_{\mathbf{g}_o}dv_{\{r_o=r\}} \\
        &+\frac{1}{2}\lim_{r\to 0}\int_{\{r_o=r\}}\langle \nabla_{\partial_{r_o}} (\chi H^4_2),\mathbf{v}_l\rangle_{\mathbf{g}_o} - \langle \chi H^4_2,\nabla_{\partial_{r_o}} \mathbf{v}_l\rangle_{\mathbf{g}_o}dv_{\{r_o=r\}}\\
        =& 3 \int_{\mathbb{S}^3\slash\mathbb{Z}_2}\langle H^4,V_{l,2} \rangle_\mathbf{e} dv_{\mathbb{S}^3\slash\mathbb{Z}_2}+ \int_{\mathbb{S}^3\slash\mathbb{Z}_2} \langle H^4_ 2,V_{l,0}\rangle_\mathbf{e}dv_{\mathbb{S}^3\slash\mathbb{Z}_2}
    \end{aligned}\label{dvp 1st obst orb}
    \end{equation}
    Indeed, since once restricted to $\mathbb{S}^3\slash\mathbb{Z}_2$, the coefficients of $H^4$ and $V_{l,0}$ are eigenfunctions of the Laplace-Beltrami operator and associated to different eigenvalues, the product against the first asymptotic term $V_{l,0}$ vanishes. Note that this also justifies the orthogonality of \eqref{eq prol H4}. For the second term $V_{l,2}$, we use the homogeneity of $H^4$ and $V_{l,2}$ which yield $\nabla_{\partial_{r_\mathbf{e}}}H^4 = -\frac{4}{r_\mathbf{e}}H^4$ and $\nabla_{\partial_{r_\mathbf{e}}}V_{l,2} = \frac{2}{r_\mathbf{e}}V_{l,2}$ and the above value. Thanks to the Fredholm properties of $P_o$ proved in \cite{ozu2}, we have the stated result.
    
    When $\lambda_k =0$ for $k\in \{1,2,3\}$, then, we can use Lemma \ref{def H42}, hence plug $H^4_2=0$ into \eqref{dvp 1st obst orb} and find the expected value.
\end{proof}

Similarly, if $(M_o,\mathbf{g}_o)$ has several $\mathbb{R}^4\slash\mathbb{Z}_2$ singularities, denote $\overline{h}^4_{j}$ the symmetric $2$-tensors of \eqref{eq prol H4} asymptotic to $H^4(\zeta^\pm_j) + H_2^4(\zeta^\pm_j)$ at the singular point $j$. The symmetric $2$-tensor $$\overline{h}^4:=\sum_{j \text{ singular}}\overline{h}^4_{j}$$ satisfies
\begin{equation}
  \left\{
      \begin{aligned}
        P_o \overline{h}^4 &= \sum_k \nu_l\mathbf{v}_l+\sum_{j}\sum_k \lambda_k^j\overline{\mathbf{o}}_{j,k}^4 ,\\
        \overline{h}^4&= H^4(\zeta^\pm_j) + H^4_2(\zeta^\pm_j)+ \mathcal{O}(r_o^{-\epsilon}) \text{ for any } \epsilon>0 \text{ at each singular point } j\\
        \overline{h}^4&\perp \mathbf{O}(\mathbf{g}_o).\label{obstructino prol H4 a tous les pts}
      \end{aligned}
    \right.
\end{equation}
\subsubsection{First variation of the initial obstruction}\label{section first var lambda}
By \cite[Proposition 5.17]{ozu2}, around a singular point $j$, in coordinates where $\mathbf{g}_o = \mathbf{e} + H_{j,2}+\mathcal{O}(r_\mathbf{e}^4)$, for $v\in\mathbf{O}(\mathbf{g}_o)$ small enough and with $v= V_{j,0}+V_{j,2} + \mathcal{O}(r_\mathbf{e}^4)$ in the same coordinates, we have
$$\mathbf{g}_v = (\mathbf{e}+\tilde{V}_{j,0}) + (H_{j,2} + \tilde{V}_{j,2}) + \mathcal{O}(r_\mathbf{e}^4),$$
with $ \tilde{V}_{j,0}=V_{j,0} + \mathcal{O}(\|v\|^2_{L^2(\mathbf{g}_o)}) $ and $\tilde{V}_{j,2} = V_{j,2} + \mathcal{O}(\|v\|^2_{L^2(\mathbf{g}_o)}r_\mathbf{e}^2)$. We also denote $\psi(\tilde{V}_{j,0}) : \mathbb{R}^4\slash\mathbb{Z}_2\to\mathbb{R}^4\slash\mathbb{Z}_2$ the linear isomorphism for which $ \psi(\tilde{V}_{j,0})^*\mathbf{e} = \mathbf{e}+\tilde{V}_{j,0} $.

For $k\in \{1,2,3\}$ and a singular point $j$, and $\zeta = (\zeta_j)_j$ with $\zeta_j\in \mathbb{R}^3\backslash\{0\}$, we will denote $\lambda_k^j(\zeta,v)$ the real numbers of \eqref{valeur lambda k psi zeta} obtained for $ H_0 = \tilde{V}_{j,0} $ and $ H_2 =H_{j,2} + \tilde{V}_{j,2} $, that is for $c>0$,
\begin{equation}
    \lambda_k^j(\zeta,v)= c\big\langle \mathbf{R}^\pm\big(\psi(\tilde{V}_{j,0})_*(H_{j,2} + \tilde{V}_{j,2} )\big)\zeta_j^\pm\;,\; \zeta_{j,(k)}^\pm \big\rangle.\label{exp lambda jk}
\end{equation}

We find a link between the variations of the first obstruction $\lambda_1$ on the Eguchi-Hanson metric and the first obstructions $\nu_l$ on the orbifold.
\begin{cor}
    Consider $(M_o,\mathbf{g}_o)$ an integrable Einstein orbifold with only $\mathbb{R}^4\slash\mathbb{Z}_2$ singularities. Consider also a small enough $v\in \mathbf{O}(\mathbf{g}_o)$, the orthonormal basis $(\mathbf{v}_l)_l$ of Proposition \ref{premiere obst orbifold 1 pt} and the Einstein deformation $\mathbf{g}_v$ of $\mathbf{g}_o$ of Definition \ref{einstein integrable deformations}. Assume that we have $\lambda_k^j(\zeta,0) = 0$ for all $k,j$.
    
    Then, for some constant $C\in\mathbb{R}^*$, we have the following control:
    \begin{equation}
        \sum_{j \text{ singular}}|\zeta_j|\lambda^j_1(\zeta,v)= C\sum_l\langle v,\mathbf{v}_l\rangle\nu_l(\zeta,0) + \mathcal{O}(\|v\|_{L^2}^2|\zeta|^2),\label{estimée somme lambda 1}    \end{equation}
    where the $\nu_l$ are real numbers of \eqref{obstructino prol H4 a tous les pts}, and where $|\zeta| = \max_j|\zeta_j|$.
\end{cor}
\begin{proof}
    Let $(M_o,\mathbf{g}_o)$ be like in the statement above and consider a small enough $ v\in \mathbf{O}(\mathbf{g}_o)$ and $\mathbf{g}_v$ the Einstein deformation of $\mathbf{g}_o$ of Definitions \ref{def einst mod obst orbifold} and \ref{einstein integrable deformations}, denote $\lambda_1^j(\zeta,v)$ the constant in \eqref{exp lambda jk} above. Our goal is to show that the first variation of $v\mapsto \sum_j|\zeta_j|\lambda_1^j(\zeta,v)$ around $(\zeta,0)$ is proportional to $\sum_k\langle v,\mathbf{v}_l\rangle\nu_l(\zeta,0)$.
    
    From the expression \eqref{exp lambda jk} and the formula \cite[Proposition 3.2]{biq1}, we see that for the singular point $j$, the first variation of $v\mapsto\sum_j|\zeta_j|\lambda_1^j(\zeta,v)$ is
    $$\sum_j|\zeta_j|\int_{\mathbb{S}^3\slash\mathbb{Z}_2} \big\langle V_2,O_{j,1}^4 \big\rangle_\mathbf{e} dv_{\mathbb{S}^3\slash\mathbb{Z}_2},$$
    up to a multiplicative constant. Using the equality $H^4(\zeta_j^\pm) =\frac{|\zeta_j|}{2}O_{j,1}^4(\zeta_j^\pm)$ together with \eqref{obstructino prol H4 a tous les pts}, we see that this equals $\sum_k\langle v,\mathbf{v}_l\rangle\nu_l$ up to a multiplicative constant and the estimate \eqref{estimée somme lambda 1} follows.
\end{proof}

\subsubsection{Values for the second obstruction}

    Let $H_2$ be a quadratic symmetric $2$-tensor with $P_\mathbf{e}H_2 = \Lambda \mathbf{e}$, and $H_4$ with $ |H_4|_\textbf{e} \sim r_\mathbf{e}^4$ and $P_\mathbf{e}H_4 = \Lambda H_2 - Q_\mathbf{e}^{(2)}(H_2,H_2)$, and recall the notation $\mathbf{R}^+(H_2)$ as the common selfdual part of the curvature at $r_\mathbf{e}=0$ of Riemannian metrics with a development $\mathbf{e} + H_2 +\mathcal{O}(r_\mathbf{e}^3)$. As seen above, we have $\lambda_k:= c\big\langle\mathbf{R}^+(H_2)(\omega_1^+),\omega_k^+\big\rangle$ for the $\lambda_k$ of \eqref{def lambda k} and $c>0$. In particular, if for all $k$, $\lambda_k = 0$, then we have
        $$\mathbf{R}^+(H_2)=\begin{bmatrix}0&0&0\\
0&R_{22}&R_{23}\\
0&R_{32}&R_{33}
\end{bmatrix},
    $$ and in this case, for a positive constant $C>0$, by \cite[Lemme 9]{biq2} we have 
    \begin{equation}
        \mu_1 = C\det \begin{bmatrix}
R_{22}&R_{23}\\
R_{32}&R_{33}
\end{bmatrix}\label{expresion mu 1}
    \end{equation}
    
    \begin{rem}\label{signe mu1 ricci plat}
        If $\Ric(H_2)=0$, then we have $\textup{tr}\mathbf{R}^+(H_2)=0$ and therefore $$\det \begin{bmatrix}
R_{22}&R_{23}\\
R_{32}&R_{33}
\end{bmatrix}\leqslant0$$
with equality if and only if $\mathbf{R}^+(H_2)=0$.
    \end{rem}
    
    In general, we might only know that the $\lambda_k$ are small rather than exactly vanishing. We still have the following estimate.
    \begin{lem}\label{mu1 approche}
        Let $H_2$ and $H_4$ be homogeneous symmetric $2$-tensors as above. We have
    \begin{equation}
        \mu_1 = C\det \begin{bmatrix}
R_{22}&R_{23}\\
R_{32}&R_{33}
\end{bmatrix} + \mathcal{O}\big(\sum_k|\lambda_k|\cdot|r_e^{-2}H_2|_{g_e}\big).\label{expresion mu 1 approx}
    \end{equation}
    \end{lem}
    \begin{proof}
    According to \cite[Lemma 7]{biq2}, we have
         $$\mathbf{R}^+(H_2)=\begin{bmatrix}c^{-1}\lambda_1&c^{-1}\lambda_2&c^{-1}\lambda_3\\
        c^{-1}\lambda_2&R_{22}&R_{23}\\
        c^{-1}\lambda_3&R_{32}&R_{33}
\end{bmatrix}.
    $$
    Exactly like in the proof of \cite[Lemma 8]{biq2}, we use the expression of $Q_e^{(2)}$ given in \cite[Lemma 3]{biq2} in order to express $\mu_1$. The value of $\mu_1$ is a quadratic form in the coefficients of $\mathbf{R}^+(H_2)$. It is composed of products of a $\lambda_k$ and a $\lambda_j$, products of the $\lambda_k$ with the remaining $R_{ij}$, and its last part involving $R_{22}$, $R_{23}$, $R_{32}$ and $R_{33}$ is 
    $$C\det \begin{bmatrix}
R_{22}&R_{23}\\
R_{32}&R_{33}
\end{bmatrix},$$
as proven in \cite[Lemma 8]{biq2}. Since all of the $\lambda_k$ as well as the $R_{ij}$ are linear in $H_2$ and are therefore are controlled by $|r_e^{-2}H_2|_{g_e}$. We deduce the stated estimate.
    \end{proof}
    For another Eguchi-Hanson metric $\mathbf{eh}_{\zeta^\pm}$, one has similar estimates with the matrix expression of $\mathbf{R}^+(H_2)$ in the basis $ (\zeta^\pm,|\zeta|\cdot\zeta^\pm_{(2)},|\zeta|\cdot\zeta^\pm_{(3)}) $ on $\mathbb{R}\cdot\zeta^\pm$ and its orthogonal. Note that in that case, we have $\lambda_k\sim |\zeta|$ and $\mu_1\sim |\zeta|^2$.
    
    \subsection{Approximate Einstein metrics and obstructions}
    
    Let $(M_o,\mathbf{g}_o)$ be a compact Einstein orbifold with integrable Einstein deformations and only $\mathbb{R}^4\slash\mathbb{Z}_2$ singularities. For $v\in\mathbf{O}(\mathbf{g}_o)$ small enough and $\zeta = (\zeta_j^\pm)_j$ with $\zeta_j\in \mathbb{R}^3\backslash\{0\}$, we recall the notations:
    \begin{enumerate}
        \item $\mathbf{g}_v$ the Einstein deformation of $\mathbf{g}_o$ of Definition \ref{einstein integrable deformations},
        \item for a singular point $j$, symmetric $2$-tensors $\underline{h}_{j,2}(v,\zeta)$ and $\underline{h}_{j,4}(v,\zeta)$ of \eqref{def lambda k zeta} and \eqref{def mu k zeta} by extending the asymptotic terms of $\mathbf{g}_v$ on $\mathbf{eh}_{\zeta_j^\pm}$ and
        \item a symmetric 
        $2$-tensor $\overline{h}^4(v,\zeta)$ of \eqref{obstructino prol H4 a tous les pts} extending the $r^{-4}$ term of $\mathbf{eh}_{\zeta_j^\pm}$ on $(M_o,\mathbf{g}_v)$.
        \item $g^D$ or $g^D$ for the naïve desingularization and $r_D$ the associated radial parameter, see Definitions \ref{desing eh} and \ref{rD}
    \end{enumerate}

    \begin{defn}[Approximate Einstein modulo obstructions metric $g^A_{\zeta,v}$]\label{def gAv zeta}
        We define an approximate Einstein modulo obstructions metric $g^A_{\zeta,v}$ as the naïve gluing of the metrics $\mathbf{eh}_{\zeta_j^\pm}+\underline{h}_{j,2}(v,\zeta)+\underline{h}_{j,4}(v,\zeta)$ at the singular points $j$ of $(M_o,\mathbf{g}_v+\overline{h}^4(v,\zeta))$.
        
    \end{defn}
    
    There exists $C=C(\mathbf{g}_o)>0$ such that we have
    \begin{itemize}
        \item $r_D^k|\nabla^k_{\mathbf{eh}_{\zeta_j^\pm}}\underline{h}_2|\leqslant Cr_D^2$,
        \item $r_D^k|\nabla^k_{\mathbf{eh}_{\zeta_j^\pm}}\underline{h}_4|\leqslant Cr_D^4$, and
        \item $r_D^k|\nabla^k_{\mathbf{eh}_{\zeta_j^\pm}}\overline{h}^4|\leqslant C|\zeta|^2r_D^{-4}$,
    \end{itemize}
    we have the following controls:
    on $M_o(\epsilon)$ (defined in \eqref{eq Moepsilon})
    \begin{equation}
        r_D^k\big|\nabla^k(g^A_{\zeta,v}-g^D_{\zeta,v})\big|_{g^D_{\zeta,v}}\leqslant C |\zeta|^2,\label{gA-gD 1}
    \end{equation}
    and on the connected component of $M\backslash M_o(\epsilon)$ containing the singular point $j\in S$:
    \begin{align}
        r_D^k\big|\nabla^k(g^A_{\zeta,v}-g^D_{\zeta,v})\big|_{g^D_{\zeta,v}}\leqslant C \left(\mathbb{1}_{\{r_D>|\zeta_j|^\frac{1}{4}\}}|\zeta_j|^2r_D^{-4} + \mathbb{1}_{\{r_D<2|\zeta_j|^\frac{1}{4}\}}r_D^2\right).\label{gA-gD 2}
    \end{align}
    where $\mathbb{1}_{A}$ is the indicator function for $A$.
    
    We now approximate the kernel of the linearization of the Einstein operator at $g^A_{\zeta,v}$.
    \begin{defn}[Approximate obstructions $\tilde{\mathbf{O}}(g^A_{\zeta,v})$]\label{def tilde O gAv zeta}
        We define $\tilde{\mathbf{O}}(g^A_{\zeta,v})$ the linear space spanned by the infinitesimal variations of $(v,\zeta)\mapsto g^A_{\zeta,v}$.
        
        More precisely for the element $ \mathbf{o}_k \in \mathbf{O}(\mathbf{eh}_{\zeta_j^\pm})$ with $\mathbf{o}_k = \partial_{\zeta_{j,(k)}}\mathbf{eh}_{\zeta_j^\pm}$, we define $$\tilde{\mathbf{o}}_k:=\partial_{\zeta_{j,(k)}}g^A_{\zeta,v}$$
        and, for $\mathbf{w}\in \mathbf{O}(\mathbf{g}_v)$, we define $\tilde{\mathbf{w}}$ the associated infinitesimal deformation of $v\mapsto g^A_{\zeta,v}$.
    \end{defn}
\begin{rem}
    As in Remark \ref{rem OgD ozu}, the space $\tilde{\mathbf{O}}(g^A_{\zeta,v})$ is very close to the space $\tilde{\mathbf{O}}(g^D_{\zeta,v})$ of \cite{ozu2} and even yields better estimate. The results of \cite{ozu2} therefore hold when replacing  $\tilde{\mathbf{O}}(g^D_{\zeta,v})$ by $\tilde{\mathbf{O}}(g^A_{\zeta,v})$.
\end{rem}
    
    We have the following estimates which show that the space $\tilde{\mathbf{O}}(g^A_{\zeta,v})$ is an approximate $L^2$-cokernel of the linearization of the Einstein equation.
    \begin{lem}\label{proj contre obst}
    For $\tilde{\mathbf{o}}_k\in \tilde{\mathbf{O}}(g^A_{\zeta,v})$ and $\tilde{\mathbf{w}}\in \tilde{\mathbf{O}}(g^A_{\zeta,v})$ as in Definition \ref{def tilde O gAv zeta} above, we have for any symmetric $2$-tensor $h$ on M,
        \begin{equation}
            \big|\langle P_{g^D_{\zeta,v}}h,\tilde{\mathbf{o}}_k \rangle_{L^2(g^D_{\zeta,v})}\big|\leqslant C |\zeta|\cdot \|h\|_{C^{2,\alpha}_{\beta,*}(g^D_{\zeta,v})}\|\mathbf{o}_k\|_{L^2(\mathbf{eh}_{\zeta_j^\pm})} \text{ and } \label{projection contre obst ale}
        \end{equation}
        \begin{equation}
            \big|\langle P_{g^D_{\zeta,v}}h,\tilde{\mathbf{w}} \rangle_{L^2(g^D_{\zeta,v})}\big|\leqslant C |\zeta|\cdot \|h\|_{C^{2,\alpha}_{\beta,*}(g^D_{\zeta,v})}\|\mathbf{w}\|_{L^2(\mathbf{g}_o)}.\label{projection contre obst orb}
        \end{equation}
    \end{lem}
    \begin{proof}
    
    By integration by parts and symmetry of $P_{g^D_{\zeta_v}}$ on the closed manifold $M$, it is equivalent to control $$\int_{M} \left\langle h, P_{g^D_{\zeta,v}} \tilde{\mathbf{o}}_k \right\rangle_{g^D_{\zeta,v}} dv_{g^D_{\zeta,v}} \text{ and } \int_{M} \left\langle h, P_{g^D_{\zeta,v}} \tilde{\mathbf{w}} \right\rangle_{g^D_{\zeta,v}} dv_{g^D_{\zeta,v}}.$$
    
       On the region where $g^D_{\zeta,v}=\mathbf{eh}_{\zeta_j^\pm}$, we have $P_{g^D_{\zeta,v}}\tilde{\mathbf{o}}_k=0$. There remains the region where $g^D_{\zeta,v}=\mathbf{g}_v$ on which by differentiating \eqref{obstructino prol H4 a tous les pts} with respect to $\zeta_j^\pm$, we find 
       $$ P_{g^D_{\zeta,v}}\tilde{\mathbf{o}}_k = \sum_l\partial_{\zeta_{j,(k)}} \nu_l \mathbf{w}_l + \sum_{i,m}\partial_{\zeta_{j,(k)}} (\lambda_i^m \overline{\mathbf{o}}_{i,m}^4), $$
       where $(\mathbf{w}_l)_l$ denotes an orthonormal basis of $\mathbf{O}(\mathbf{g}_v)$, and
       where the $\partial_{\zeta_{j,(k)}} \nu_l$ are the derivatives of the $\nu_l=\nu_l(v,\zeta)$ of \eqref{obstructino prol H4 a tous les pts} in the direction $\zeta_{j,(k)}$. Given the expressions \eqref{valeur obst lambda zeta}, \eqref{dvp obst avec zeta} and \eqref{obstructino prol H4 a tous les pts}, that gives for some $C = C(\mathbf{g}_o)>0$,
       $$\big\| P_{g^D_{\zeta,v}}\tilde{\mathbf{o}}_k\big\|_{L^2( \mathbf{eh}_{\zeta_j^\pm})} \leqslant C|\zeta|\cdot \|\mathbf{o}_k(\zeta_j^\pm)\|_{L^2( \mathbf{eh}_{\zeta_j^\pm})}. $$
       We find that $$\int_{\{r_D>2|\zeta_j^\pm|^{1/4}\}} \left\langle h, P_{g^D_{\zeta,v}} \tilde{\mathbf{o}}_k \right\rangle_{g^D_{\zeta,v}} dv_{g^D_{\zeta,v}}\leqslant C |\zeta|\cdot\|h\|_{C^0(g^D_{\zeta,v})}\|\mathbf{o}_k(\zeta_j^\pm)\|_{L^2( \mathbf{eh}_{\zeta_j^\pm})}.$$
       
       Finally, on the gluing region where $|\zeta_j^\pm|^\frac{1}{4}<r_D<2|\zeta_j^\pm|^\frac{1}{4}$, we have for all $l\in \mathbb{N}$ and $C_l>0$ $$r_D^l|\nabla^l(g^D_{\zeta,v}-\mathbf{eh}_{\zeta_j^\pm})|_{g^D_{\zeta,v}}\leqslant C_l \big(|\zeta_j^\pm|^2r_D^{-4}+r_D^2\big)$$
       and $$r_D^l|\nabla^l(\tilde{\mathbf{o}}_k-\mathbf{o}_k)|_{g^D_{\zeta,v}}\leqslant C \big(|\zeta_j^\pm|^4r_D^{-8}+|\zeta_j^\pm|^2r_D^{-2}\big),$$
       and we find that on the gluing region where $|\zeta_j^\pm|^\frac{1}{4}<r_D<2|\zeta_j^\pm|^\frac{1}{4}$, we have
       \begin{align*}
           \left|P_{g^D_{\zeta,v}} \tilde{\mathbf{o}}_k\right|_{g^D_{\zeta,v}}&\leqslant\left|P_{g^D_{\zeta,v}}( \tilde{\mathbf{o}}_k - \mathbf{o}_k) \right|_{g^D_{\zeta,v}}+\big|\big(P_{g^D_{\zeta,v}}-P_{\mathbf{eh}_{\zeta_j^\pm}}\big)(\mathbf{o}_k) \big|_{g^D_{\zeta,v}}\\ 
           &\leqslant C|\zeta_j^\pm|\cdot\|\mathbf{o}_k\|_{L^2(\mathbf{eh}_{\zeta_j^\pm})}.
       \end{align*}
       We therefore globally find the estimate \eqref{projection contre obst ale}.
       \\
       
       For the second estimate, note first that on the region where $g^D_{\zeta,v}=\mathbf{g}_v $, we have 
       $$P_{g^D_{\zeta,v}}\mathbf{w}=0,$$
       and by a similar argument to the above one, by differentiating \eqref{def lambda k zeta} and \eqref{def mu k zeta}, on each region where $g^D_{\zeta,v}= \mathbf{eh}_{\zeta_j^\pm}$, we have
        $$P_{g^D_{\zeta,v}}\tilde{\mathbf{w}}=\sum_k \partial_{\mathbf{w}}(\lambda_k+\mu_k) \mathbf{o}_k(\zeta_j^\pm),$$
        and we find that there exists $C=C(\mathbf{g}_o)>0$ such that $|\partial_\mathbf{w}(\lambda_k+\mu_k)| \leqslant C |\zeta|$. Together with a control on the gluing region like the above one, one obtains the estimate \eqref{projection contre obst orb}.
    \end{proof}

    \begin{prop}\label{est einst op gAv zeta}
        Assume that there exists an Einstein metric $\mathbf{g}_{\zeta,v}$ satisfying:
        \begin{enumerate}
            \item $\mathbf{\Phi}_{g^D_{\zeta,v}}(\mathbf{g}_{\zeta,v})= 0$,
            \item $\mathbf{g}_{\zeta,v}- g^A_{\zeta,v}\perp \tilde{\mathbf{O}}(g^A_{\zeta,v})$
            \item $\| \mathbf{g}_{\zeta,v}- g^A_{\zeta,v} \|_{C^{2,\alpha}_{\beta,*}(g^D_{\zeta,v})}<\epsilon$ for $\epsilon>0$ small enough determined in \cite{ozu2} depending only on $\mathbf{g}_o$.
        \end{enumerate}
        Then, we actually have for all $l$ for the obstructions of \eqref{eq prol H4}:
        \begin{equation}
            \nu_l = \mathcal{O}(|\zeta|^\frac{5}{2}),
        \end{equation}
        and for each $j$, and all $k\in \{1,2,3\}$ for the obstructions of \eqref{def lambda k zeta} and \eqref{def mu k zeta}:
        \begin{equation}
            \lambda_k^j+\mu_k^j = \mathcal{O}(|\zeta_j|^\frac{5}{2}).
        \end{equation}
    \end{prop}
    \begin{proof}
        Let us denote $\pi_{\tilde{\mathbf{O}}(g^A_{\zeta,v})^\perp}$ the $L^2(g^A_{\zeta,v})$-orthogonal projection on $\tilde{\mathbf{O}}(g^A_{\zeta,v})^\perp$. Thanks to the obstruction result of \cite[Proposition 5.1]{ozuthese}, we only have to understand the value of 
        $\big\|\pi_{\tilde{\mathbf{O}}(g^A_{\zeta,v})^\perp}\mathbf{\Phi}_{g^D_{\zeta,v}}(g^A_{\zeta,v})\big\|_{r_D^{-2}C^{\alpha}(g^D_{\zeta,v})},$
        in order to control both $\|g^A_{\zeta,v}-\mathbf{g}_{\zeta,v}\|_{C^{2,\alpha}_{\beta,*}(g^D_{\zeta,v})}$ and the obstructions, where the norms are defined in Definition \ref{norme a poids M} in the Appendix.
    
        The estimate on the different Eguchi-Hanson regions has been done in \cite[Section 14]{biq1}. There remains to understand the orbifold region. Using the definition of $\overline{h}^4$, we control the first variation $$P_{\mathbf{g}_o}\overline{h}^4= \sum_l \nu_l\mathbf{v}_l +\sum_{j}\sum_{k= 1}^3 \lambda_k^j\overline{\mathbf{o}}_{j,k}^4,$$
        and on the region where $g^A_{\zeta,v}= \mathbf{g}_o + \overline{h}^4$ with $\overline{h}^4=\mathcal{O}(|\zeta|^2r_D^{-4})$, we have:
        $$ \mathbf{\Phi}_{g^D_{\zeta,v}}(g^A_{\zeta,v}) = \sum_l \nu_l\mathbf{v}_l +\sum_{j}\sum_{k= 1}^3 \lambda_k^j\overline{\mathbf{o}}_{j,k}^4+ \mathcal{O}(|\zeta|^4r_D^{-10}) $$ which is exactly analogous to the estimate \cite[(108)]{biq1}. The gluing region is treated exactly as in \cite[Section 14]{biq1} for instance, see also the above proof of Lemma \ref{proj contre obst}. We conclude like in \cite[(109)]{biq1} that 
        \begin{equation}
            \big\|\pi_{\tilde{\mathbf{O}}(g^A_{\zeta,v})^\perp}\mathbf{\Phi}_{g^D_{\zeta,v}}(g^A_{\zeta,v})\big\|_{r_D^{-2}C^{\alpha}(g^D_{\zeta,v})} \leqslant C |\zeta|^{\frac{3}{2}-\frac{\beta}{4}}.\label{est Phi }
        \end{equation}
        We conclude exactly like in order to prove \cite[(112)]{biq1} that we have the stated control (see also \cite[Proposition 5.1]{ozuthese}).
    \end{proof}
    
    \subsection{A new obstruction to the desingularization of Einstein orbifolds}\label{new obst section}
    
    By \cite{ozu1,ozu2}, we know that Einstein metrics which are close to an Einstein orbifold $(M_o,\mathbf{g}_o)$ result from a gluing-perturbation procedure. To prove that it cannot be desingularized, we just need to prove that an obstruction to this procedure does not vanish.
    
    \subsubsection{Approximate obstructions}
    Let us denote $\zeta = t\cdot \phi$ with $t=(t_j)_j = (|\zeta_j|)_j$ and $\phi = (\phi_j)_j = \big(\zeta_j^\pm/|\zeta_j|\big)_j$. By Proposition \ref{est einst op gAv zeta}, we find the following estimate for any singular point $j$, any $k\in \{1,2,3\}$:
    \begin{equation}
     t_j\lambda_k^j(\phi_j,v)+t_j^2\mu_k^j(\phi_j,v) = \mathcal{O}(t_{\max}^\frac{5}{2}), \text{ and }\label{contrôle obst}
    \end{equation}
    \begin{equation}
        \nu_l(t\cdot \phi,v) = \mathcal{O}(t_{\max}^\frac{5}{2}), \text{ for each $l$, }\label{contrôle obst orb}
    \end{equation}
        where the $\lambda_k^j(\phi_j,v)$ and $\mu_k^j(\phi_j,v)$ and $\nu_l(t\cdot\phi, v)$ are the constants of Lemma \eqref{def lambda k zeta}, \eqref{def mu k zeta} and \eqref{obstructino prol H4 a tous les pts} for each $\textbf{eh}_{\zeta_j^\pm}$ associated to the development of $\mathbf{g}_v$ at the singular point $j$.

    Our goal is to prove that for all $j$, both all of the $\lambda_k^j$ and $\mu_1^j$ have to vanish when $t\to 0$. The main concern is to rule out the compensation of $ t_j^2\mu_1^j $ by $t_j\lambda_1^j$ just like in the desingularization case of \cite[Theorem 14.3]{biq1}. 
    
\begin{proof}[Proof of Theorem \ref{nouvelleobst} and Proposition \ref{2e obst nondeg}]
    Let us consider $(M_o,\mathbf{g}_o)\in \overline{\mathbf{E}(M)}_{GH}$ for $M$ having the topology of $M_o$ desingularized by $T^*\mathbb{S}^2$ in an orientation at each of its singular point and assume that $\mathbf{g}_o$ only has integrable Einstein deformations (see Definition \ref{einstein integrable deformations}). Let us assume that there exists a sequence of Einstein metrics $(M,\mathbf{g}_n)_n$ desingularizing $(M_o,\mathbf{g}_o)$. We will discuss at the end of the proof in which situation we shall assume that this sequence is nondegenerate, transverse or stable.
    
    Then, by \cite{ozu1,ozu2}, there exists a sequence of naïve desingularizations $g^D_{\zeta_n,v_n}$ such that the metrics $(M,\mathbf{g}_n)$ are Einstein perturbations of $g^D_{\zeta_n,v_n}$ with $\zeta_n=(\zeta_{n,j})_j$. We will use the notations $t_n = (t_{n,j})_j$, with $t_{n,j} = |\zeta_{n,j}| >0$ and $\phi_n = (\phi_{n,j})_j$, $\phi_{n,j} = \frac{\zeta^\pm_{n,j}}{|\zeta_{n,j}|}$ satisfying $\zeta_{n,j}^\pm=t_{n,j}\cdot\phi_{n,j}$, and we will denote $t_n\cdot \phi_n = (t_{n,j}\phi_{n,j})_j$.
    
    By compactness, up to taking a subsequence, we can assume that for any $j$, $\phi_{n,j}\to \phi_{\infty,j}$ on the sphere while $t_n\to 0$ and $v_n\to0$ since by assumption the limit is $\mathbf{g}_o$.
    
    Denote $\lambda_{k}^j(\phi_{n,j},v_n)$ and $\mu_{k}^j(\phi_{n,j},v_n)$ the real numbers of \eqref{def lambda k} and \eqref{def mu k} for these configurations. According to \eqref{contrôle obst}, we have the following controls: for all $j$ and $k$,
    \begin{equation}
        t_{n,j}\lambda_{k}^j(\phi_{n,j},v_n) + t_{n,j}^2\mu_{k}^j(\phi_{n,j},v_n) = \mathcal{O}(t_{n,\max}^\frac{5}{2}). \label{dvp obst bulle}
    \end{equation}

    The first obstruction of \cite{biq1,ozu2} means that at the limit, for any singular point $j$, we have $\lambda^j_k(\phi_{\infty,j}, 0)= 0$ for $k\in\{1,2,3\}$, there exists a basis of the selfdual or anti-selfdual (depending on the orientation of the gluing) $2$-forms starting with $\phi_{\infty,j} =\lim_{j\to +\infty}\zeta_{n,j}^\pm/|\zeta_{n,j}|$ in which we have
    \begin{equation}
    \mathbf{R}^\pm_{g_o}(j)=\begin{bmatrix}0&0&0\\
0&a^j_2&0\\
0&0&a^j_3
\end{bmatrix}.\label{courbure limite 1e obstr}
    \end{equation}
    
    Our goal is to prove that for all $j$ we have, $\mu_{1}^j(\phi_{\infty,j},0) = 0$. By Lemma \eqref{mu1 approche}, this will then imply that either $a_2^j= 0$ or $a_3^j= 0$ in \eqref{courbure limite 1e obstr}.
    
    Assume towards a contradiction that both $a_2^j\neq 0$ and $a_3^j\neq0$. Let us use the variations of the $\phi\mapsto\lambda_k^j(\phi,0)$ computed in \cite[Lemma 12.2]{biq1} at $\phi_{\infty,j}$, we first have the controls:
    \begin{align}
        \lambda_1^j(\phi_{n,j},0) =  \mathcal{O}(|\phi_{n,j}-\phi_{\infty,j}|^2),\label{lambda1variations}
    \end{align}
    because the first variation of the smooth map $\phi\mapsto\lambda_1^j(\phi,0)$ vanishes at $\phi_{\infty,j}$. Similarly since we assumed towards a contradiction that $a^j_2\neq0$ and $a^j_3\neq 0$, there exists $c>0$ depending on $(a^j_2,a^j_3)$, so that
    \begin{equation}
        \big|\big(\lambda_2^j(\phi_{n,j},0),\lambda_3^j(\phi_{n,j},0)\big)\big|\geqslant c|\phi_{n,j}-\phi_{\infty,j}|,\label{borne inf lambda 2 3}
    \end{equation}
    since the first order variations of $\phi\mapsto(\lambda_2^j(\phi,0),\lambda_3^j(\phi,0))$ is invertible at $\phi_{\infty,j}$. For the variations induced by $v_n\in\mathbf{O}(\mathbf{g}_o)$, for any $k\in \{1,2,3\}$, from the expression \eqref{exp lambda jk}, we have
     $$ |\lambda_k^j(\phi_{n,j},v_n)-\lambda_k^j(\phi_{n,j},0)| = \mathcal{O}(\|v_n\|_{L^2(\mathbf{g}_o)}).$$
     For the $\mu_k^j$, we have the control $\mu_{k}^j(\phi_{n,j},v_n) = \mu_k^j(\phi_{\infty,j},0) + \mathcal{O}(|\phi_{n,j}-\phi_{\infty,j}|) + \mathcal{O}(\|v_n\|_{L^2(\mathbf{g}_o)})$. Therefore for $k = 2$ or $k=3$, \eqref{dvp obst bulle} gives
    \begin{equation}
        t_{n,j}\lambda_k^j(\phi_{n,j},v_n) = - t_{n,j}^2\big(\mu_k^j(\phi_{\infty,j},v_n) + \mathcal{O}(|\phi_{n,j}-\phi_{\infty,j}|)+\mathcal{O}(\|v_n\|_{L^2(\mathbf{g}_o)})\big) + \mathcal{O}(t_{n,\max}^\frac{5}{2}). \label{dvp obst bulle approx k}
    \end{equation}
    At any singular point $j$, we must therefore have $\big|\big(\lambda_2^j(\phi_{n,j},v_{n}),\lambda_3^j(\phi_{n,j},v_{n})\big)\big|= \mathcal{O} (t_{n,j}+\frac{t_{n,\max}^{5/2}}{t_{n,j}})$, and by \eqref{borne inf lambda 2 3}, this implies that $|\phi_{n,j}-\phi_{\infty,j}| = \mathcal{O}\big(\|v_n\|_{L^2(\mathbf{g}_o)}+t_{n,j}+\frac{t_{n,\max}^{5/2}}{t_{n,j}}\big)$.
    \begin{enumerate}
        \item In the \emph{transverse} (Definition \ref{transverse}) situation, where we have $t_{n,\max}=\mathcal{O}(t_{n,j})$ and $\|v_n\|_{L^2(\mathbf{g}_o)}= o(t_{n,\min})$. By rewriting \eqref{dvp obst bulle} for $k=1$, including the bound \eqref{lambda1variations} with the new information that  $|\phi_{n,j}-\phi_{\infty,j}| = \mathcal{O}(t_{n,j})$ and using $|\lambda_1(\phi_{\infty,j},v_{n,j})-\lambda_1(\phi_{\infty,j},0)|= \mathcal{O}(\|v_{n,j}\|_{L^2(\mathbf{g}_o)})=o(t_{n,\min})$, we consequently find
    \begin{equation}
        o(t_{n,j}^2) + t_{n,j}^2\big(\mu_{1}^j(\phi_{\infty,j},0) + \mathcal{O}(|\phi_{n,j}-\phi_{\infty,j}|)+ \mathcal{O}(\|v_n\|_{L^2(\mathbf{g}_o)})\big) = \mathcal{O}(t_{n,\max}^\frac{5}{2}),
    \end{equation}
    and therefore $\mu_1^j(\phi_{\infty,j},0)= 0$.
        
        \item 
    As proven in \cite[Proof of Theorem 3]{biq3}, and thanks to the estimate \eqref{est Phi }, we know that the linearization of $\Ric$ in Bianchi gauge at $\mathbf{g}_n$ denoted $P_{\mathbf{g}_n}$, satisfies for some $o_{n,j,k}\in L^2(\mathbf{g}_n)$ (constructed from $\mathbf{o}_{n,j,k}\in\mathbf{O}(\mathbf{eh}_{\zeta_{n,j}})$) for $k\in\{1,2,3\}$:
    $$\big\langle P_{\mathbf{g}_n}o_{n,j,1},o_{n,j,1}\big\rangle_{L^2(\mathbf{g}_n)} = o(1)\cdot\|o_{n,j,1}\|_{L^2(\mathbf{g}_n)}, \text{ and}$$
    $$\big\langle P_{\mathbf{g}_n}o_{n,j,k},o_{n,j,k}\big\rangle_{L^2(\mathbf{g}_n)} = (a_k^j+o(1))\cdot\|o_{n,j,k}\|_{L^2(\mathbf{g}_n)},$$
    for $k\in\{2,3\}$. That means that the metric cannot be \emph{stable} as in Definition \ref{def stable} if $a^j_2 = -a^j_3\neq 0$. Consequently, we necessarily have $a^j_2=a^j_3= 0$ for stable Ricci-flat desingularizations.

        \item If the orbifold satisfies $\mathbf{O}(\mathbf{g}_o)= \{0\}$ and only has one $\mathbb{R}^4\slash\mathbb{Z}_2$ singularity, then the transversality assumption is automatically satisfied and the result follows.
        
        \item In the Ricci-flat situation, then the \emph{nondegeneracy} assumption is enough because we can use Remark \ref{signe mu1 ricci plat}. Indeed, first, the nondegeneracy assumption gives $|\phi_{n,j}-\phi_{\infty,j}|=o(t_{n,j}^\frac{1}{2})$, hence $|\lambda_1(\phi_{n,j},0)|= o(t_{n,j})$ by \eqref{lambda1variations}; second from \eqref{estimée somme lambda 1}, we moreover have a better control over the sum of the $t_j^2\lambda_1^j(\zeta,v)$, namely:
    \begin{equation}
        \sum_jt_j^2\lambda_1^j(\phi_{n,j},v_{n})-t_j^2\lambda_1^j(\phi_{n,j},0)= \mathcal{O}(t_{\max}^\frac{5}{2}\|v_n\|_{L^2(\mathbf{g}_o)})+ \mathcal{O}(t_{\max}^2\|v_n\|_{L^2(\mathbf{g}_o)}^2). \label{somme lambda1 est preuve}
    \end{equation}
        which together with \eqref{lambda1variations} and \eqref{contrôle obst} and using $ \|v_n\|^2_{L^2(\mathbf{g}_o)} = o(t_{n,j}) $ yields the estimate
        \begin{align*}
            \sum_j o\Big(&t_{n,j}^3 + t_{n,\max}^\frac{5}{2}t_{n,j}^\frac{1}{2} +t_{n,\max}^2 t_{n,j}\Big)\\
            &+ \sum_jt_{n,j}^3\big(\mu_{1}^j(\phi_{\infty,j},0) + \mathcal{O}(|\phi_{n,j}-\phi_{\infty,j}|)+\mathcal{O}(\|v_n\|_{L^2(\mathbf{g}_o)})\big) 
            = \mathcal{O}(t_{n,j}t_{n,\max}^\frac{5}{2}),
        \end{align*}
        which with $t_{n,\max}=\mathcal{O}(t_{n,j})$ from the nondegeneracy assumption implies that
        $$ \sum_j\mu_1^j(\phi_{\infty,j},0)=0 $$
        and since all of the $\mu_1^j(\phi_{\infty,j},0)$ are nonpositive by Remark \ref{signe mu1 ricci plat}, for any $j$, we have $\mu_1^j(\phi_{\infty,j},0) = 0$. 
        
        \item Similarly for Einstein but non Ricci-flat orbifolds with only one singularity, we use \eqref{estimée somme lambda 1} as in the previous point since there is only one term in the sum.
    \end{enumerate}
\begin{rem}
    The estimate \eqref{contrôle obst orb} also implies that: for any $l$, one has
    $$ \lim_{n\to+\infty} \nu_l\Big(\frac{t_n}{t_{n,\max}}\cdot \phi_n\;,\;v_n\Big) = \lim_{n\to+\infty} \frac{\nu_l(t_n\cdot \phi_n\;,\;v_n)}{t_{n,\max}^2} = 0$$
    which has an interpretation in terms of curvature, at least when the infinitesimal Einstein deformations vanish at the singular points by \eqref{obstructino prol H4 a tous les pts} and \eqref{estimée somme lambda 1}.
\end{rem}    
\end{proof}
\begin{rem}
    The same obstruction result also holds if we consider desingularizations by \emph{smooth} Ricci-flat ALE of $A_1$, $D_k$ or $E_k$ singularities by \cite[Lemme 9]{biq2}. An obstruction is also satisfied for $A_k$ singularities but it also involves derivatives of the curvature \cite[Lemme 12]{biq2}. It also extends to Kähler Ricci-flat ALE manifolds by taking their hyperkähler finite cover.
\end{rem}
We also believe that the nondegeneracy assumption is purely technical. Understanding higher order obstructions should yield the following conjectural statement which already holds under a stability assumption.
\begin{conj}\label{conj Ricci plat}
    Let $(M_n,\mathbf{g}_n)_n$ be a sequence of compact smooth Ricci-flat metrics converging in the Gromov-Hausdorff sense towards a Ricci-flat orbifold $(M_o,\mathbf{g}_o)$ while bubbling Eguchi-Hanson metrics. Then the curvature of $\mathbf{g}_o$ at its singular points is either selfdual or anti-selfdual depending on the orientation of the Eguchi-Hanson metrics.
\end{conj}

\section{Ricci-flat modulo obstructions desingularizations of $\mathbb{T}^4\slash \mathbb{Z}_2$}\label{section obst desing T4/Z2}

Let us now work with the flat orbifold $\mathbb{T}^4\slash\mathbb{Z}_2$. We will study the question of whether or not it is possible to desingularize it by Ricci-flat metrics obtained by gluing Eguchi-Hanson metrics in \emph{different} orientations at its singular points.

\subsection{The orbifold $\mathbb{T}^4\slash \mathbb{Z}_2$ and its deformations}

We define $\mathbb{T}^4\slash \mathbb{Z}_2$ as the quotient of $\mathbb{R}^4\slash(2L\mathbb{Z}^4)$ for some $L\in GL(4,\mathbb{R})$ by the action of $\{\textup{Id},-\textup{Id}\}$. A metric on $\mathbb{T}^4\slash \mathbb{Z}_2$ can be seen as an $2L\mathbb{Z}^4$-invariant and $\mathbb{Z}_2$-invariant metric on $\mathbb{R}^4$. It is therefore determined by its values on any cube $L([a,a+1]^4)$ for $a\in \mathbb{R}$. We will denote $(\mathbb{T}^4\slash \mathbb{Z}_2,\mathbf{g}_L)$ the orbifold obtained with the matrix $L$.

All of the infinitesimal Ricci-flat deformations of a flat $\mathbb{T}^4\slash \mathbb{Z}_2$ are integrable. More precisely, all of the Ricci-flat deformations of $\mathbb{T}^4\slash \mathbb{Z}_2$ are flat and can be seen as either varying the above invertible matrix $L$, or as modifying the scalar product, that is by adding a constant symmetric $2$-tensor of norm smaller than $1$ to the metric $\mathbf{g}_L$.

This orbifold has 16 singular points with $\mathbb{R}^4\slash\mathbb{Z}_2$ singularities. We will denote 
$$S:= \{L(\epsilon_1,\epsilon_2,\epsilon_3,\epsilon_4), \epsilon_i\in \{0,1\}\} + 2\mathbb{Z}^4 $$
its singular set on $\mathbb{R}^4$ before the taking the quotients by $2L\mathbb{Z}^4$ and $\{\mathrm{Id},-\mathrm{Id}\}$.

\begin{rem}
    Since this orbifold is flat, it clearly satisfies our obstruction \eqref{2e obst}.
\end{rem}

\subsection{Partial hyperkähler desingularizations and estimates}
In the rest of the article, we consider $(M,g^D_{L,\zeta})$ a naïve desingularization of $(\mathbb{T}^4\slash\mathbb{Z}_2,\mathbf{g}_L)$ for $L\in GL(4,\mathbb{R})$ by positively oriented Eguchi-Hanson metrics $(\mathbf{eh}_{\zeta_i^+})_{i\in S_+}$ for $\zeta_i\in \mathbb{R}^3\backslash\{0\}$ at a subset $S_+$ of the $16$ singular points of $\mathbb{T}^4\slash\mathbb{Z}_2$ and by negatively oriented $(\mathbf{eh}_{\zeta_j^-})_{j\in S_-}$ for $\zeta_j\in \mathbb{R}^3\backslash\{0\}$ at a subset $S_-$  of the singular points. We denote $\zeta = ((\zeta_i^+)_{i\in S_+},(\zeta_j^-)_{j\in S_-})$. We will first be interested in its \emph{partial} hyperkähler desingularizations obtained by perturbation of the metric only desingularized by Eguchi-Hanson metrics in the same orientation.

\subsubsection{Hyperkähler partial desingularizations of $\mathbb{T}^4\slash\mathbb{Z}_2$ and approximations}
Let us define $(M_\pm,g^D_\pm) := (M_{S_\pm},g^D_{L,\zeta^\pm})$ the partial desingularizations where the only points which are desingularized are $S_\pm$ by the above Eguchi-Hanson metrics $(\mathbf{eh}_{\zeta_i^\pm})_{i\in S_\pm}$. They can be perturbed to hyperkähler orbifolds. 
\begin{lem}\label{lem def bf g pm}
    Let $g_\pm^D:= g^D_{L,\zeta^\pm}$ be a naïve desingularization of $\mathbb{T}^4\slash\mathbb{Z}_2$ thanks to Eguchi-Hanson metrics \eqref{eguchihanson} glued in the \emph{same orientation} at some or all singular points.
    
    Then, the Einstein modulo obstructions perturbation of $ g^D_\pm $, denoted $\mathbf{g}_\pm:= \mathbf{g}_{L,\zeta^\pm}$ and defined in \cite{ozu2} is hyperkähler.
\end{lem}
\begin{proof}[Sketch of proof]
    It is well-known (or can be proven from \cite{don}) that the moduli space of the Ricci-flat and even hyperkähler (orbifold) metrics on $M_\pm$ is of maximal dimension $3|S_\pm|+9$ which is the same as the space of Ricci-flat modulo obstructions metrics of \cite{ozu2}. Using a connectedness argument similar to that of \cite[Proposition 5.70]{ozuthese} for the moduli space of Kronheimer's gravitational instantons, we conclude that any Einstein modulo obstruction perturbation of a metric $g^D_\pm$ is actually Ricci-flat (and even hyperkähler).
\end{proof}

To obtain more information about these hyperkähler metrics $\mathbf{g}_\pm$, it suffices to construct metrics which are approximately Ricci-flat by \cite{ozu2}.
\begin{rem}
    We will in particular need to have a good approximation of the Riemannian curvature tensor of $\mathbf{g}_\pm$ at the remaining singular points. We will see that we have $|\mathbf{R}_{\mathbf{g}_\pm}|\sim |\zeta^\pm|^2$, and the previous control \eqref{est Phi } in $\mathcal{O}(|\zeta^\pm|^{\frac{3}{2}-\frac{\beta}{4}})$ is not sufficient for our purpose.
\end{rem}

\subsubsection{Approximation of the partial hyperkähler desingularizations}

We want to find good enough approximations to be able to control the curvature of our hyperkähler partial desingularizations. Consider the following construction.

On the orbifold $(\mathbb{T}^4\slash\mathbb{Z}_2,\mathbf{g}_L)$, for any harmonic homogeneous symmetric $2$-tensor on $\mathbb{R}^4\slash\mathbb{Z}_2$ with $|H^4|_{\mathbf{e}}\sim r_\mathbf{e}^{-4}$ and a singular point $i\in S$, we can solve the equation
    \begin{equation}
  \left\{
      \begin{aligned}
        P_L \overline{h}^4 &= 0, \text{ and } B_L\overline{h}^4  = 0,\\
        \overline{h}^4 &= H^4 +0+ \mathcal{O}(r^{-\epsilon}) \text{ for any } \epsilon>0 \text{ at the singular point } i,\\
        \overline{h}^4&\perp \mathbf{O}(\mathbf{g}_L).\label{eq prol H4 T4}
      \end{aligned}
    \right.
\end{equation}
somewhat explicitly (notice that there is no obstruction contrarily to \eqref{eq prol H4}). Indeed, as in \cite{bk}, one would expect the periodization of the sum of the symmetric $2$-tensors $H$ to be a solution. At $x\in\mathbb{R}^4$, this would look like
\begin{equation}
    \sum_{a\in\mathbb{Z}^4} H^4(L(2a)+x-i).\label{somme divergente}
\end{equation}
However except in the most symmetric cases, like the one treated in \cite{bk}, this periodization diverges (it is a sum of terms in $r^{-4}$ over a $\mathbb{Z}^4$ grid). One way around it is simply to substract to each term of the sum the mean value of it over a period, namely, we define
\begin{equation}
    \overline{h}^4(x) =  \sum_{a\in\mathbb{Z}^4} \big(H^4(L(2a)+x-i) - H^4_{av}(L,i,2a)\big),\label{expression extension orbifold}
\end{equation}
where $H_{av}^4(L,i,2a)$ is the average of $H^4(x)$ on the set $i+L(2a + [-1,1]^4)$. This sum converges because we have $\nabla^l H^4 = \mathcal{O}(r_\mathbf{e}^{-4-l})$ which is summable on $\mathbb{Z}^4\backslash\{0\}$ for $l>0$. In particular, the curvature is summable.
Note moreover that, by construction, $\overline{h}^4$ of \eqref{expression extension orbifold} is indeed orthogonal to the infinitesimal deformations of $\mathbf{g}_L$ which are constant symmetric $2$-tensor and that $B_{\mathbf{g}_L}\overline{h}^4=0$ as well as  $\textup{tr}_{\mathbf{g}_L}\overline{h}^4=0$.
\begin{rem}\label{courbure hzeta}
    The point is that the symmetric $2$-tensors we have substracted are constant and therefore, the curvature induced are the same as those of the sum \eqref{somme divergente}.
\end{rem}
Let, $H^4(\zeta_i^\pm)$ for $i\in S_\pm$ be the $r^{-4}$-asymptotic terms of the metric $\mathbf{eh}_{\zeta_i^\pm}$, see \eqref{dvp eh}. Consider the unique symmetric $2$-tensor $\overline{h}^4(\zeta^\pm_i)$ as in \eqref{expression extension orbifold} which is asymptotic to $ H^4(\zeta_i^\pm) $ at $i$ and bounded everywhere else while satisfying $\overline{h}^4(\zeta^\pm_i)\perp_{\mathbf{g}_L}\mathbf{O}(\mathbf{g}_L)$.

Consider now a singular point $i\in S_\pm$ and denote $|\zeta^\pm| = \max_{i'\in S_\pm}|\zeta_{i'}^\pm|$. We have the following expansion of the sum $\sum_{i'\neq i}\overline{h}^4(\zeta_{i'}^\pm) = H_0(i) + H_2(i) + \mathcal{O}(d(i,.)^4)$ at $i$ with $|H_0(i)|=\mathcal{O}(|\zeta^\pm|^2)$ a constant symmetric $2$-tensor, and $|H_2(i)|=\mathcal{O}(|\zeta^\pm|^2)r_\mathbf{e}^2$. For some linear isomorphism $\phi_{i}$ with, $\|\phi_{i}-\textup{Id}\|= \mathcal{O}(|\zeta^\pm|^2)$, we have $\mathbf{e}+H_0(i) = \phi_{i}^*\mathbf{e}$. In the coordinates of \eqref{eguchihanson}, the metric $\phi_{i}^*\mathbf{eh}_{\zeta_{i}^\pm}$ (see Remark \ref{meaning diffeo} for the meaning of this action of the diffeomorphism) is asymptotic to $\mathbf{e}+H_0(i)$. Consider now $\underline{h}_2(i)$ a solution to
$P_{\phi_{i}^*\mathbf{eh}_{\zeta_{i}^\pm}}\underline{h}_2(i) \in \mathbf{O}(\phi_{i}^*\mathbf{eh}_{\zeta_{i}^\pm}),$
with $\underline{h}_2(i) = H_2(i) ++  \mathcal{O}(r^{-2+\epsilon})$ for all $\epsilon>0$ from \eqref{def lambda k}. Since the perturbation $\textbf{g}_\pm$ is Ricci-flat, we actually have 
\begin{equation}
    P_{\phi_{i}^*\mathbf{eh}_{\zeta_{i}^\pm}}\underline{h}_2(i) =0.\label{eq overline h 2 tore}
\end{equation}
Note that $\underline{h}_2(i)$ it is determined up to $\mathbf{O}(\phi_{i}^*\mathbf{eh}_{\zeta_{i}^\pm})$. We choose the solution which is orthogonal to $\mathbf{O}(\phi_{i}^*\mathbf{eh}_{\zeta_{i}^\pm})$ for the $L^2$-product induced by $\phi_{i}^*\mathbf{eh}_{\zeta_{i}^\pm}$ on $\mathbb{S}^2$ as in \eqref{def lambda k}.

\begin{defn}[Approximate hyperkähler metric $g^A_\pm$]\label{def gApm}
    We define the approximate hyperkähler metric $g^A_\pm$ in a way similar to Definition \ref{def naive desing}): we glue the different $\phi_{i}^*\mathbf{eh}_{\zeta_{i}^\pm} + \underline{h}_2(i)$ for $i\in S_\pm$ to the metric $ \mathbf{g}_L + \sum_{i\in S_\pm} \overline{h}^4(\zeta_i^\pm)$ at the points $i\in S_\pm$ with cut-off at distance at $|\zeta_i^\pm|^\frac{1}{6}$, that is with the following metric on the gluing region:
    $$ \chi(|\zeta_i^\pm|^{-\frac{1}{6}}r_e)\Big(\phi_{i}^*\mathbf{eh}_{\zeta_{i}^\pm} + \underline{h}_2(i)\Big) + \big(1-\chi(|\zeta_i^\pm|^{-\frac{1}{6}}r_e)\big)\Big(\mathbf{g}_L + \sum_{i\in S_\pm} \overline{h}^4(\zeta_i^\pm)\Big).$$ 
\end{defn}

\subsubsection{Control of the partial hyperkähler desingularizations}

We now justify the qualifier ``approximate'' for the metric $g^A_\pm$.
\begin{prop}\label{controle gApm}
    We have the following control on the metric: for all $k\in\mathbb{N}$, there exists $C = C(k,L)>0$, such that
    \begin{equation}
        \|g^A_\pm - \mathbf{g}_\pm\|_{C^{k}_{\beta,*}(g^D_\pm)}\leqslant C |\zeta^\pm|^{2+\frac{2}{3}-\frac{\beta}{6}}.\label{est gApm bf gpm}
    \end{equation}
    In particular, at each remaining singular point $p$, we have
    \begin{equation}
        \left|\mathbf{R}_{g^A_\pm} - \mathbf{R}_{\mathbf{g}_\pm}\right|_{g^D_\pm} \leqslant C|\zeta^\pm|^{2+\frac{2}{3}-\frac{\beta}{6}}.\label{est Rgpm RgApm}
    \end{equation}
\end{prop}
\begin{rem}
    We a priori had $\mathbf{R}_{\mathbf{g}_\pm} = \mathcal{O}(|\zeta^\pm|^2)$ so this is indeed a good approximation of $\mathbf{R}_{\mathbf{g}_\pm}$ as $\zeta^\pm\to 0$. The controls of \eqref{est Phi } in $|\zeta|^{\frac{3}{2}-\frac{\beta}{4}}$ are not precise enough here.
\end{rem}
\begin{proof}
    Thanks to \cite[Proof of Theorem 4.51]{ozuthese}, it is enough to control the values of $\mathbf{\Phi}_{g^D_\pm}(g^A_\pm)$ in $r_D^{-2}C^{k,\alpha}_\beta(g^D_\pm)$ in order to obtain \eqref{est gApm bf gpm}, where for two metrics $g$ and $g'$, we define $$\mathbf{\Phi}_{g}(g'):= \Ric(g') + \delta^*_{g'}B_{g}g'$$
    as in the first section. Estimates similar to \cite[Lemma 14.2]{biq1} and \cite[Proposition 3.4]{bk}, yield the following controls. For all $k\in \mathbb{N}$, there exists $C_k>0$ for which for a given singular point $i_0$, in the neighborhood of $i_0$
    \begin{enumerate}
        \item in the region where $r_D>2|\zeta_{i_0}^\pm|^\frac{1}{6}$, we have
        $g^A_\pm = \mathbf{g}_L + \overline{h}^4$ and $g^D_\pm = \mathbf{g}_L$ and therefore, since the linear terms vanish, that is $d_{\mathbf{g}_L}\mathbf{\Phi}_{\mathbf{g}_L} \overline{h}^4 = 0$, the error is at least quadratic. Since $r_D^k|\nabla^k\overline{h}^4|_{g^D_\pm}\leqslant C |\zeta^\pm|^2 r_D^{-4}$, we have 
        $$r_D^k|\nabla^k\mathbf{\Phi}_{g^D_\pm}(g^A_\pm)|_{g^A_\pm}\leqslant C_k |\zeta^\pm|^4r_D^{-10},$$
        
        \item in the region where $r_D<|\zeta_{i_0}^\pm|^\frac{1}{6}$, we have $g^A_\pm = \mathbf{eh}_{\zeta_{i}^\pm} + \underline{h}_2$ and $g^D_\pm = \mathbf{eh}_{\zeta_{i}^\pm}$. Like above, by \eqref{eq overline h 2 tore} we also only have to control the nonlinear terms, and since $r_D^k|\nabla^k\underline{h}_2|_{g^D_\pm}\leqslant C |\zeta^\pm|^2r_D^{2}$, we find
         $$r_D^k|\nabla^k\mathbf{\Phi}_{g^D_\pm}(g^A_\pm)|_{g_\pm^A}= |\zeta^\pm|^4r_D^2,$$
        \item in the region where $|\zeta_{i_0}^\pm|^\frac{1}{6}<r_D<2|\zeta_{i_0}^\pm|^\frac{1}{6}$, $g^A_\pm$ and $g^D_\pm$ are respectively an interpolation between $\mathbf{g}_L + \overline{h}^4$ and $\mathbf{eh}_{\zeta_{i}^\pm} + \overline{h}_2 $ and an interpolation between $\mathbf{g}_L$ and $\mathbf{eh}_{\zeta_{i}^\pm}$ thanks to a cut-off function. Now we have $$ r_D^k\Big|\nabla^k\Big(\big(\mathbf{g}_L + \overline{h}^4\big)-\big(\mathbf{eh}_{\zeta_{i}^\pm} + \underline{h}_2 \big)\Big)\Big|_\mathbf{e} \leqslant C |\zeta^\pm|^2 r_D^4 + |\zeta^\pm|^4r_D^{-8}, \text{ and }$$
        $$ r_D^k\Big|\nabla^k\big(\mathbf{g}_L - \mathbf{eh}_{\zeta_{i}^\pm} \big)\Big|_\mathbf{e} \leqslant C |\zeta^\pm|^2 r_D^2 + |\zeta^\pm|^2r_D^{-4}, $$
        and since $|\zeta^\pm|^{1/6}<r_D<2|\zeta^\pm|^{1/6}$, by controlling the cut-off function, we find
        $r_D^{2+k}|\nabla^k\mathbf{\Phi}_{g^D_\pm}(g^A_\pm)|_{g^A_\pm}\leqslant C|\zeta^\pm|^{2+\frac{2}{3}}$.
\end{enumerate}
    Finally, since the control is straightforward away from the singular points, we obtain 
    $$\|\mathbf{\Phi}_{g^D_\pm}(g^A_\pm)\|_{r_D^{-2}C^{k,\alpha}_{\beta}(g^D_\pm)}\leqslant C|\zeta^\pm|^{2+\frac{2}{3}-\frac{\beta}{6}}, $$
    which gives the estimate \eqref{est gApm bf gpm}, see the argument of \cite[Theorem 4.51]{ozuthese}.
\end{proof}

\subsubsection{Approximation of the infinitesimal deformations}
Let us approximate the kernel of $P_{\mathbf{g}_\pm}$, $\mathbf{O}(\mathbf{g}_\pm)$ thanks to the following symmetric $2$-tensors.
\begin{defn}\label{obst approchées}
    We define $\tilde{\mathbf{O}}(g^A_\pm)$ the linear space spanned by the infinitesimal variations of $(L,\zeta^\pm)\mapsto g^A_\pm = g^A_{L,\zeta^\pm}$.
\end{defn}

Let us start by recalling a general lemma for operators between Banach spaces which will be useful to approximate the elements of the kernel of our operator on the hyperkähler partial desingularizations.
\begin{lem}[\cite{ozu2}]\label{approx kernel} Let $P,P':X\to Y$ be two operators between two Banach spaces $X$ and $Y$ for which there exists $C>0$ $\frac{1}{100 C}>\epsilon>0$ and a finite-dimensional linear subspace $K'$ and $S'$ a supplement of $K'$ in $X$ such that we have
    \begin{enumerate}
        \item for any $x\in X$, $\|(P-P')x\|_Y\leqslant \epsilon\|x\|_X,$
        \item for any $x \in S'$, $\|x\|_X\leqslant C\|P'x\|_Y,$
        \item for any $x\in K'$, 
        $\|P'x\|_Y\leqslant \epsilon\|x\|_X,$
        \item $dim(\ker P) = dim (K')$.
    \end{enumerate}
    Then, for any element $k'$ of $K'$, there exists an element $k$ of $\ker P$ such that $$\|k-k'\|_X\leqslant \frac{2C\epsilon}{1-C\epsilon} \|k'\|_X.$$
\end{lem}
\begin{proof}[ ]
    
\end{proof}

\begin{prop}\label{estim obst oApm}
    Let $ \mathbf{O}(\mathbf{g}_\pm) $ be the kernel of $P_{\mathbf{g}_\pm}$. Then, for any $\mathbf{o}_\pm\in \mathbf{O}(\mathbf{g}_\pm) $, there exist $ \mathbf{o}^A_\pm \in \tilde{\mathbf{O}}(g_\pm^A) $ such that we have
    $$ \|\mathbf{o}_\pm-\mathbf{o}^A_\pm\|_{C^{2,\alpha}_{\beta,*}(g^D_\pm)}\leqslant C |\zeta^\pm|^{2+\frac{2}{3}-\frac{\beta}{6}}. $$
\end{prop}
\begin{proof}
    Let us apply Lemma \ref{approx kernel} with $P = P_{\mathbf{g}_\pm}$, $P' = P_{g^A_\pm}$, $K' = \tilde{\mathbf{O}}(g_\pm^D)$ and $S' = \tilde{\mathbf{O}}(g_\pm^A)^\perp$.
    We have the following controls: there exists $C>0$ independent of $|\zeta^\pm|$ small enough
    \begin{enumerate}
        \item for any $h\in C^{2,\alpha}_{\beta,*}(g^D_\pm)$, by Proposition \ref{controle gApm}, we have $$\|(P_{\mathbf{g}_\pm}- P_{g^A_\pm})h\|_{r_D^{-2}C^{\alpha}_\beta(g^D_\pm)}\leqslant C |\zeta^\pm|^{2+\frac{2}{3}-\frac{\beta}{6}}\|h\|_{C^{2,\alpha}_{\beta,*}(g^D_\pm)},$$
        \item for any $h\in C^{2,\alpha}_{\beta,*}(g^D_\pm)\cap \tilde{\mathbf{O}}(g^A_\pm)^\perp$, by \cite{ozu2}, we have 
        $$ \|h\|_{C^{2,\alpha}_{\beta,*}(g^D_\pm)}\leqslant C \|P_{g^A_\pm}h\|_{r_D^{-2}C^\alpha_\beta(g^D_\pm)},$$
        \item
        by an estimation very similar to that of Proposition \ref{controle gApm}, which we therefore omit, for any $\mathbf{o}^A\in \tilde{\mathbf{O}}(g^A_\pm)$,
        $$ \|P_{g^A_\pm}\mathbf{o}^A\|_{r_D^{-2}C^\alpha_\beta(g^D_\pm)}\leqslant C |\zeta^\pm|^{2+\frac{2}{3}-\frac{\beta}{6}}\|\mathbf{o}^A\|_{C^{2,\alpha}_{\beta,*}(g^D_\pm)},$$
        \item and, $dim\mathbf{O}(\mathbf{g}_\pm) =dim \tilde{\mathbf{O}}(g_\pm^A)$ since every gluing configuration can be perturbed to a Ricci-flat manifold by Lemma \ref{lem def bf g pm}.
    \end{enumerate}
    We can therefore apply Lemma \ref{approx kernel} and obtain the stated estimate.
\end{proof}

Using the control of $\mathbf{o}_\pm$ of Proposition \ref{est einst op gAv zeta}, we moreover have the following estimate of the infinitesimal variation of $\mathbf{R}_{\mathbf{g}}$ in the direction of $\mathbf{o}_\pm$: for $C>0$, we have
\begin{equation}
    \left|\partial_{\mathbf{o}_\pm}\mathbf{R}_{\mathbf{g}_{\pm}}-\partial_{\mathbf{o}^A_\pm}\mathbf{R}_{g^A_\pm}\right|_{g^A_\pm}\leqslant C|\zeta^\pm|^{2+\frac{2}{3}-\frac{\beta}{6}}\|\mathbf{o}_\pm\|_{L^2(\mathbf{g}_\pm)}\label{control part opmRgpm}
\end{equation}
where we denoted $\partial_{\mathbf{o}_\pm}\mathbf{R}_{\mathbf{g}_{\pm}}$ the differential of $\mathbf{g}_{\pm}\mapsto\mathbf{R}_{\mathbf{g}_{\pm}}$ at $\mathbf{g}_{\pm}$ in the direction $\mathbf{o}_\pm$.

\subsection{Total desingularizations modulo obstructions}

Let us start by using a notation which will be convenient for this section.

\begin{defn}
    For a section $s$ on $(\mathbb{R}^4\slash\mathbb{Z}_2)\backslash\{0\}$, we will write $s\propto r_\mathbf{e}^k$ if $s\in r_\mathbf{e}^k C^l_{-\epsilon}(\mathbb{R}^4\slash\mathbb{Z}_2)$ for all $l\in\mathbb{N}$ and $\epsilon>0$.
    
    For $k\in \mathbb{Z}$ and $f:(\mathbb{R}^3)^{16}\mapsto \mathbb{R}^+_*$, we will write $s\propto f(\zeta) r_\mathbf{e}^k$ if as the parameters $\zeta\to 0$, we have $s = (f(\zeta)+o(f(\zeta)) s'$ for $s'\propto r_\mathbf{e}^k$.
\end{defn}

\begin{rem}
    For any $l\geqslant 0$, we have $r_\mathbf{e}^k\log^l(r_\mathbf{e})\propto r_\mathbf{e}^k$.
\end{rem}

Let us now consider the \emph{total} desingularization of the \emph{partial} hyperkähler desingularizations $\mathbf{g}_\pm$ by gluing the Eguchi-Hanson metrics $\mathbf{eh}_{\zeta_j^\mp}$ at the remaining singular points $j\in S_\mp$. Here, the hyperkähler orbifold $(M_\pm,\mathbf{g}_\pm)$ will play the role of $(M_o,\mathbf{g}_o)$ in Section \ref{section new section}. From the obstruction result of Theorem \ref{nouvelleobst}, we can already say that in order to desingularize the orbifold $(M_\pm,\mathbf{g}_\pm)$ by a nondegenerate sequence of Einstein manifolds in the Gromov-Hausdorff sense, it is necessary that the condition $\mathbf{R}_{\mathbf{g}_\pm}=0$ is satisfied at every remaining singular point. This was done thanks to an approximate development at order $4$, see Definition \ref{def gAv zeta}.

In this section, we are interested in the more challenging situation of a sequence of desingularization of a \emph{sequence} of metrics $\mathbf{g}_{\pm,n}$, and this requires an approximation at an even higher order.

\subsubsection{An approximate Einstein modulo obstructions metric}

Let $L\in GL(4,\mathbb{R})$ and $\zeta = \left((\zeta_i^\pm)_{i\in S_\pm},(\zeta_j^\mp)_{j\in S_\mp}\right)\in (\mathbb{R}^3)^{|S_\pm|}\times (\mathbb{R}^3)^{|S_\mp|}\sim (\mathbb{R}^3)^{16}$.
\begin{defn}[Total desingularizations, $g^D_{L,\zeta}$]
    We define $g^D_{L,\zeta}$ the naïve gluing at all $j\in S_\mp$ of the metrics $ \mathbf{eh}_{\zeta_j^\mp}$ to the metric $\mathbf{g}_\pm$. Recall the notation $|\zeta^\pm| = \max_{i\in S_\pm}|\zeta_i|$.
\end{defn}

As in Section \ref{section new section}, the goal of the present section is to construct better approximate Einstein modulo obstructions metric in order to identify the obstructions.

\begin{prop}\label{prop construction gALzeta}
    Let $L_0\in GL(4,\mathbb{R})$. Then there exists $C>0$ and a neighborhood $\mathcal{U}$ of $(L_0,0)\in GL(4,\mathbb{R})\times (\mathbb{R}^3)^{16}$ such that there exists a smooth family of metrics $ (g^A_{L,\zeta})_{(L,\zeta)\in \mathcal{U}}$
    for which if we denote $\tilde{\mathbf{O}}(g^A_{L,\zeta})$ the space of the infinitesimal deformations of $(L,\zeta)\mapsto g^A_{L,\zeta}$, we have
    \begin{enumerate}
        \item $\|g^A_{L,\zeta}-g^D_{L,\zeta}\|_{C^{2,\alpha}_{\beta,*}(g^D_{L,\zeta})} \to 0 $ as $\zeta\to 0$,
        \item $g^A_{L,\zeta}-g^D_{L,\zeta}\perp_{L^2(g^A_{L,\zeta})}  \tilde{\mathbf{O}}(g^A_{L,\zeta})$, and
        \item $\|\mathbf{\Phi}_{g^A_{L,\zeta}} (g^A_{L,\zeta})- \mathbf{o}^A_{L,\zeta}\|_{r_D^{-2}C^\alpha_\beta(g^D_{L,\zeta})}\leqslant C |\zeta^\pm|^{1+\frac{\beta}{12}}|\zeta^\mp|^{3-\frac{\beta}{4}} $ for $\mathbf{o}^A_{L,\zeta}\in \tilde{\mathbf{O}}(g^A_{L,\zeta})$ explicited in the proof.
    \end{enumerate}
\end{prop}

Let us consider the development of the metric $\mathbf{g}_\pm$ at the singular point $j\in S_\mp$, 
\begin{equation}
    \mathbf{g}_\pm=\mathbf{e}+\sum_{m\in \mathbb{N}^*} H_{j,2m}, \text{ with $H_{j,2m}\propto |\zeta^\pm|^2 r_\mathbf{e}^{2m}$}\label{dvp gpm}
\end{equation}
 by Proposition \ref{controle gApm}, and similarly, for each $j$, at infinity we have
\begin{equation}
    \mathbf{eh}_{\zeta_j^\mp} = \mathbf{e} + \sum_{m\in \mathbb{N}^*}H^{4m}(\zeta_j^{\mp}), \text{ with $H^{4m}(\zeta_j^{\mp}) \propto |\zeta^\mp|^{2m}r_\mathbf{e}^{-4m}$.}\label{dvp eh zetajmp}
\end{equation}

\subsubsection{Extension of the first asymptotic terms}

As in Section \ref{section new section}, we define $2$-tensors satisfying the following equations for $(\mathbf{o}_{j,k})_k$ a basis of $\mathbf{O}(\mathbf{eh}_{\zeta_j^\pm})$ and $\lambda_j^k\in \mathbb{R}$ as in \eqref{def lambda k zeta}:
\begin{equation}
\left\{\begin{aligned}
  P_{\mathbf{eh}_{\zeta_j^\mp}}\underline{h}_{j,2} &= \sum_k \lambda_k^j\mathbf{o}_{j,k}\\
  \underline{h}_{j,2}&= H_{j,2}+\mathbf{O}(|\zeta^\pm|^2\cdot|\zeta_j^\mp|^{2}r_\mathbf{e}^{-2+\epsilon}),\\
    \int_{\mathbb{S}^2}\langle\underline{h}_{j,2},\mathbf{o}_{j,k}\rangle_{\mathbf{eh}_{\zeta_j^\mp}} &dv_{\mathbf{eh}_{\zeta_j^\mp|\mathbb{S}^2}}= 0 \text{ for all  } k\in\{1,2,3\}
\end{aligned}\right.
\end{equation}
for $H_{j,2}$ as in \eqref{dvp gpm} and we have a converging development in a neighborhood of infinity.
\begin{equation}
    \underline{h}_{j,2}= H_{j,2}+H_{j,2}^4 +\sum_{m\geqslant 2} H_{j,2}^{4m}, \text{ with $ H_{j,2}^{4m}\propto |\zeta^\pm|^2|\zeta^\mp|^{2m}r_\mathbf{e}^{2-4m} $.}\label{def H^4_2 h_2}
\end{equation}

We then define the extension of the obstructions $\overline{\mathbf{o}}_{j,m}^4$ for $m\in\{1,2,3\}$ as the unique solutions to the following equation as in \eqref{eq prol O4}:
\begin{equation}
  \left\{
      \begin{aligned}
        P_{\mathbf{g}_\pm} \overline{\mathbf{o}}_{j,m}^4 &\in \mathbf{O}(\mathbf{g}_o),\\
        \overline{\mathbf{o}}_{j,m}^4 &= O_{j,m}^{4} +\mathcal{O}(r_\mathbf{e}^{-2-\epsilon}) \text{ for any } \epsilon>0 \text{ at the singular point } j\\
        \overline{\mathbf{o}}_{j,m}^4&\perp \mathbf{O}(\mathbf{g}_o)
      \end{aligned}
    \right.
\end{equation}
for $ O_{j,m}^{4}\propto |\zeta_j^\mp|r_\mathbf{e}^{-4}$ and $O_{j,m,2}^{4} \propto |\zeta^\pm|^2|\zeta_j^\mp|r_\mathbf{e}^{-2}$ the asymptotic term of the infinitesimal variation of $\underline{h}_2$ of \eqref{def H^4_2 h_2} as $\mathbf{eh}_{\zeta_j^\mp}$ is perturbed in the direction $\mathbf{o}_{j,m}$ like in Section $2$.

For the first approximation term on $M_\pm$,
\begin{equation}
    \left\{\begin{aligned}
  P_{\mathbf{g}_{\pm}}\overline{h}^4 &= \sum_k \lambda_k^j \overline{\mathbf{o}}_{j,k}^{4}+\sum_l\nu_l\mathbf{v}_l,\\
  \overline{h}^4&=H^4(\zeta_j^\mp)+H^4_{j,2} + \mathcal{O}(|\zeta^\pm|^2\cdot|\zeta^\mp|^{2}r_\mathbf{e}^{-\epsilon}) \text{ for any } \epsilon>0 \text{ at the singular point } j,
  \\
  \overline{h}^4&\perp \mathbf{O}(\mathbf{g}_o)
\end{aligned}\right.\label{eq h^4 T4Z2}
\end{equation}
for $H^4(\zeta_j^\mp)$ as in \eqref{dvp eh zetajmp} with $H^4_{j,2}$ determined in \eqref{def H^4_2 h_2}
and we have a converging development 
\begin{equation}
    \overline{h}^4= H^4(\zeta_j^\mp)+H^4_{j,2} + \sum_{m\geqslant 2}H^4_{j,2m} \label{def H^4_2m}
\end{equation}
 in a neighborhood of each singular point $j$, with $H^4_{j,2m}\propto|\zeta^\mp_j|^2|\zeta^\pm|^2r_\mathbf{e}^{2m-4}$.

Let us come back to the development \eqref{dvp gpm} of $\mathbf{g}_\pm$ at a singular point $j\in S_\mp$ and the term $H_{j,4}$. We know that $H_{j,4}\propto |\zeta^\pm|^2r_\mathbf{e}^4$ and satisfies 
$$ P_\mathbf{e}H_{j,4} + Q_\mathbf{e}^{(2)}(H_{j,2},H_{j,2})=0. $$
Since $H_{j,2}\propto |\zeta^\pm|^2 r_\mathbf{e}^2$, one would therefore expect by homogeneity to have $H_{j,4}\propto |\zeta^\pm|^4r_\mathbf{e}^4$ instead of the above $H_{j,4}\propto |\zeta^\pm|^2r_\mathbf{e}^4$. The reason is that $H_{j,4}$ might have a harmonic part in $|\zeta^\pm|^2r_\mathbf{e}^4$ while its part compensating the quadratic terms in $H_{j,2}$ is $\propto|\zeta^\pm|^4r_\mathbf{e}^4$. By Lemma \ref{resolution equation laplacien} in the Appendix, there exists  $\mathring{H}_{j,4}$ such that $\mathring{H}_{j,4}\propto |\zeta^\pm|^2r_\mathbf{e}^4$ such that $ H_{j,4}-\mathring{H}_{j,4}\propto|\zeta^\pm|^4r_\mathbf{e}^4 $ and satisfying $ P_\mathbf{e}\mathring{H}_{j,4}=0 $. 
The crucial thing to note is that the extension of the harmonic part $\mathring{H}_{j,4}$ on $\mathbf{eh}_{\zeta_j^\mp}$ does not induce any obstruction by \cite[(25), proof of Lemma 9]{biq2}.
This will ensure that we can indeed ``see'' the second obstruction $\mu_{j,1}=0$ for all $j$. More precisely, there exists $\mathring{\underline{h}}_{j,4}$ for which we have:
\begin{equation}
\left\{\begin{aligned}
  P_{\mathbf{eh}_{\zeta_j^\mp}}\mathring{\underline{h}}_{j,4} &= 0\\
  \mathring{\underline{h}}_{j,4}&= \mathring{H}_{j,4}+\mathbf{O}(|\zeta^\pm|^2\cdot|\zeta_j^\mp|^{2}r_\mathbf{e}^{\epsilon}), \text{ for any } \epsilon>0\\
    \int_{\mathbb{S}^2}\langle\mathring{\underline{h}}_{j,4},\mathbf{o}_{j,k}\rangle_{\mathbf{eh}_{\zeta_j^\mp}} &dv_{\mathbf{eh}_{\zeta_j^\mp|\mathbb{S}^2}}= 0 \text{ for all  } k\in\{1,2,3\}.
\end{aligned}\right.
\end{equation}
We can therefore define the following $2$-tensors on each $(T^*\mathbb{S}^2,\mathbf{eh}_{\zeta_j^\mp})$:
\begin{equation}
\left\{\begin{aligned}
  P_{\mathbf{eh}_{\zeta_j^\mp}}\underline{h}_{j,4} &+ Q_{\mathbf{eh}_{\zeta_j^\mp}}^{(2)}(\underline{h}_{j,2},\underline{h}_{j,2}) = \sum_k \mu_k^j\mathbf{o}_{j,k}\\
  \underline{h}_{j,4}&= H_{j,4}+H_{j,4}^4+\mathbf{O}(|\zeta^\pm|^2\cdot|\zeta_j^\mp|^{2}r_\mathbf{e}^{-2+\epsilon}) \text{ for any } \epsilon>0, \\
    \int_{\mathbb{S}^2}\langle\underline{h}_{j,4},\mathbf{o}_{j,k}\rangle_{\mathbf{eh}_{\zeta_j^\mp}} &dv_{\mathbf{eh}_{\zeta_j^\mp|\mathbb{S}^2}}= 0 \text{ for all  } k\in\{1,2,3\}
\end{aligned}\right.\label{eq h^4j T4}
\end{equation}
where $H^4_{j,4}$ is defined in \eqref{def H^4_2m} and satisfies 
$$P_\mathbf{e}H^4_{j,4} + Q_\mathbf{e}^{(2)}(H_{j,4},H^4(\zeta^\mp_j))+Q_\mathbf{e}^{(2)}(H_{j,2},H^4_{j,2})+Q_\mathbf{e}^{(3)}(H_{j,2},H_{j,2},H^4(\zeta_j^\mp))=0$$
which ensures that \eqref{eq h^4j T4} has a unique solution as in Section \ref{section new section}. The coefficients $\mu_k^j$ are the same as in \eqref{def mu k zeta}, and one has
\begin{equation}
    \Big\|\sum_k\mu_k^j\mathbf{o}_{j,k}\Big\|_{L^2(\mathbf{eh}_{\zeta^\mp_j})}= \mathcal{O}(|\zeta^\pm|^4|\zeta^\mp|^2).\label{controle mukj T4Z2}
\end{equation}
Remark that the factor $|\zeta^\pm|^4$ comes from the previous discussion about extending a \emph{harmonic} tensor $\mathring{H}_4$ and the difference between $H^4-\mathring{H}_4 \propto |\zeta^\pm|^4r_\mathbf{e}^4$.

We will need to extend higher order terms of the orbifold and ALE in order to produce a better approximation of the Einstein modulo obstructions perturbation of $g^D_{\zeta,v}$. We have to obtain good enough controls to ``see'' an obstruction of the same order as in \eqref{controle mukj T4Z2}.

\subsubsection{Extension of the higher order terms.}

We then also extend the obstructions of the orbifold $\mathbf{g}_\pm$ to the Eguchi-Hanson metrics to ensure a good approximation. Consider the development of $\mathbf{v}_l\in \mathbf{O}(\mathbf{g}_\pm)$ at the singular point $j$: $\mathbf{v}_l=V_{j,l,0}+V_{j,l,2}+V_{j,l,4}+V_{j,l,6}+\mathcal{O}(\|\mathbf{v}_l\|_{L^2}\cdot|\zeta^\pm|^2r_\mathbf{e}^8)$, we define the extensions
$\underline{\mathbf{v}_{j,l,}}_0=V_{j,l,0} + \mathcal{O}(|\zeta^\mp|^2r_\mathbf{e}^{-4+\epsilon})$,
$\underline{\mathbf{v}_{j,l,}}_2=V_{j,l,2} + \mathcal{O}(|\zeta^\pm|^2|\zeta^\mp|^2r_\mathbf{e}^{-2+\epsilon})$ and
$\underline{\mathbf{v}_{j,l,}}_4=V_{j,l,4} + \mathcal{O}(|\zeta^\pm|^2|\zeta^\mp|^2r_\mathbf{e}^\epsilon)$ for all $\epsilon>0$
satisfying:
$$P_{\mathbf{eh}_{\zeta_j}}\underline{\mathbf{v}_{j,l,}}_{0}=0,$$
where there are no obstructions by Lemma \ref{def h0},
$$P_{\mathbf{eh}_{\zeta_j}}\underline{\mathbf{v}_{j,l,}}_{2} + Q_{\mathbf{eh}_{\zeta_j}}^{(2)}(\underline{h}_2,\underline{\mathbf{v}_{j,l,}}_{0})\in \mathbf{O}(\mathbf{eh}_{\zeta_j^\mp}) \text{ and }$$
$$P_{\mathbf{eh}_{\zeta_j}}\underline{\mathbf{v}_{j,l,}}_{4}+Q_{\mathbf{eh}_{\zeta_j}}^{(2)}(\underline{h}_{j,2},\underline{\mathbf{v}_{j,l,}}_{2})+Q_{\mathbf{eh}_{\zeta_j}}^{(2)}(\underline{h}_{j,4},\underline{\mathbf{v}_{j,l,}}_{0})+Q_{\mathbf{eh}_{\zeta_j}}^{(3)}(\underline{h}_{j,2},\underline{h}_{j,2},\underline{\mathbf{v}_{j,l,}}_{0})\in \mathbf{O}(\mathbf{eh}_{\zeta_j^\mp}),$$
and we will take the solutions which are orthogonal to $\mathbf{O}(\mathbf{eh}_{\zeta_j^\mp})$ for the $ L^2$-product on $\mathbb{S}^2$ induced by $\mathbf{eh}_{\zeta_j^\mp}$.

Like in the decomposition for $H_{j,4}$, from the $2$-tensor $H_{j,6}$ of \eqref{dvp gpm} there exists $\mathring{H}_{j,6}\propto |\zeta^\pm|^2r_\mathbf{e}^6$ for which we have $H_{j,6}-\mathring{H}_{j,6}\propto|\zeta^\pm|^4r_\mathbf{e}^6$, and
\begin{equation}
    P_{\mathbf{e}}\mathring{H}_{j,6}=0,\label{def H60}
\end{equation}
Indeed, considering the quadratic $2$-tensor $H_{j,6}$  satisfying 
$$P_{\mathbf{e}}H_{j,6} + Q_{\mathbf{e}}^{(2)}(H_{j,2},H_{j,4})+Q_{\mathbf{e}}^{(3)}(H_{j,2},H_{j,2},H_{j,2}) =0,$$
as $|\zeta^\pm|\to 0$, one finds a $2$-tensor $ \mathring{H}_{j,6}$ satisfying \eqref{def H60} and $H_{j,6}-\mathring{H}_{j,6}\propto|\zeta^\pm|^4r_\mathbf{e}^6$, since $H_{j,2}\propto |\zeta^+|^2r_\mathbf{e}^2$ and $H_{j,4}\propto |\zeta^+|^2r_\mathbf{e}^4$ by Lemma \ref{resolution equation laplacien} in the Appendix. 
\begin{lem}\label{h_j60 def}
There exists a unique $2$-tensor $ \mathring{\underline{h}}_{j,6}$ satisfying the following identity:
\begin{equation}
\left\{\begin{aligned}
  P_{\mathbf{eh}_{\zeta_j^\mp}}\mathring{\underline{h}}_{j,6}&-\sum_l\nu_l\underline{\mathbf{v}_{j,l,}}_{0}= 0 \\
  \mathring{\underline{h}}_{j,6}&= \mathring{H}_{j,6}+\mathring{H}_{j,6}^4+\mathbf{O}(|\zeta^\pm|^2\cdot|\zeta_j^\mp|^{2}r_\mathbf{e}^{-4+\epsilon}) \text{ for any } \epsilon>0 \\
    \int_{\mathbb{S}^2}\langle\mathring{\underline{h}}_{j,6},\mathbf{o}_{j,k}\rangle_{\mathbf{eh}_{\zeta_j^\mp}} &dv_{\mathbf{eh}_{\zeta_j^\mp|\mathbb{S}^2}}= 0 \text{ for all  } k\in\{1,2,3\},
\end{aligned}\right.\label{def h60 ale}
\end{equation}
where $\mathring{H}_{j,6}^4 \propto |\zeta^\pm|^2|\zeta^\mp|^2r_\mathbf{e}^2$ satisfies
$ P_\mathbf{e} \mathring{H}_{j,6}^4 = -Q_\mathbf{e}^{(2)}(\mathring{H}_{j,6}, H^4(\zeta_j^\mp)) + \sum_l \nu_lV_{j,l,0} $
while the restriction of the coefficients of $\mathring{H}_{j,6}^4$ in a basis of the covering $\mathbb{R}^4$, once restricted to $\mathbb{S}^3$ are $L^2(\mathbb{S}^3)$-orthogonal to the eigenfunctions associated to the second eigenvalue of the Laplacian on $\mathbb{S}^3$.
\end{lem}
\begin{rem}
    Note that $\mathring{H}^4_{j,6}$ is different from $H^4_{j,6}$ defined in \eqref{def H^4_2m}. They are both $\propto|\zeta^\pm|^2|\zeta^\mp|^2r_\mathbf{e}^2$, and so is their difference. 
\end{rem}
\begin{proof}
    Let $\mathring{H}_{j,6}$ be as in \eqref{def H60}. The term $H_{j,6}^4$ in the development \eqref{def H^4_2m} satisfies for all $\epsilon>0$:
    \begin{equation}
        P_\mathbf{e}H_{j,6}^4 +Q_\mathbf{e}^{(2)}(\mathring{H}_{j,6},H^4(\zeta^\mp)) -\sum_l\nu_lV_{j,l,0} =  \mathcal{O}(|\zeta^\pm|^{4}|\zeta^\mp|^2r_\mathbf{e}^\epsilon),
    \end{equation}
    since we have $H_{j,6}-\mathring{H}_{j,6}\propto |\zeta^\pm|^{4}r_\mathbf{e}^6$, $H_{j,2}\propto |\zeta^\pm|^{2}r_\mathbf{e}^2$ and $H_{j,4}\propto |\zeta^\pm|^{2}r_\mathbf{e}^4$. Thanks to Lemma \ref{resolution equation laplacien}, we therefore find $\mathring{H}^4_{j,6}$ with $\mathring{H}^4_{j,6}\propto|\zeta^\pm|^2|\zeta^\mp|^2r_\mathbf{e}^2$ satisfying 
    \begin{equation}
        P_\mathbf{e}\mathring{H}^4_{j,6} +Q_\mathbf{e}^{(2)}(\mathring{H}_{j,6},H^4(\zeta^\mp)) -\sum_l\nu_lV_{j,l,0} = 0.\label{equation H^4j,6,0}
    \end{equation}
 The $2$-tensor $\mathring{H}^4_{j,6}$ is defined by \eqref{equation H^4j,6,0} up to the harmonic quadratic $2$-tensors on $\mathbb{R}^4\slash\mathbb{Z}_2$, and we choose the unique solution such that the restriction of the coefficients of $\mathring{H}_{j,6}^4$ in a basis of the covering $\mathbb{R}^4$, once restricted to $\mathbb{S}^3$ are $L^2(\mathbb{S}^3)$-orthogonal to the eigenfunctions associated to the second eigenvalue of the Laplacian on $\mathbb{S}^3$.
    \\
    
    Without loss of generality by rescaling and by acting on $H_{j,6}$, $\mathring{H}_{j,6}$ and $V_{j,0}$ by an element of $O(4)\backslash SO(4)$, we will study the existence of $\mathring{\underline{h}}_{j,6}$ and the vanishing of the associated obstructions on $\mathbf{eh}$, that is, as if we had $ \zeta_j^\pm = (1,0,0)^+$. 
    
    \begin{rem}
        This may change the orientation and in particular makes the curvature of $\mathring{H}_{j,6}$ become selfdual. This will let us use the more convenient usual Eguchi-Hanson metric in our computations.
    \end{rem}
    
    For a smooth cut-off function $\chi$ supported at the infinity of $\mathbf{eh}$ we have $P_\mathbf{eh}(\chi(\mathring{H}_{j,6}+\mathring{H}^4_{j,6})) - \sum_l\nu_l\underline{\mathbf{v}_{j,l,}}_{0} = \mathcal{O}(|\zeta^\pm|^2r_\mathbf{e}^{-4})$. By \cite[Proposition 2.1]{biq1}, one can find a smooth $2$-tensor $h'$ decaying at infinity such that  
    $$P_\mathbf{eh}(\chi(\mathring{H}_{j,6}+\mathring{H}^4_{j,6})+h')-\sum_l\nu_l\underline{\mathbf{v}_{j,l,}}_{0}\in \mathbf{O}(\mathbf{eh}), \text{ with}$$
    $$\big\|P_\mathbf{eh}\left(\chi(\mathring{H}_{j,6}+\mathring{H}^4_{j,6})+h'\right)-\sum_l\nu_l\underline{\mathbf{v}_{j,l,}}_{0}\big\|_{r_\mathbf{eh}^{-2}C^\alpha_\beta(\mathbf{eh})} \leqslant C |\zeta^\pm|^2$$ for some $C>0$ depending on the metric $\mathbf{eh}$ only.
    \\
    
    There now remains to prove that the obstruction actually vanishes. For this, again by \cite[Proposition 2.1]{biq1}, we need to prove that for all $k\in\{1,2,3\}$
    $$\int_{T^*\mathbb{S}^2}\big\langle P_\mathbf{eh}(\chi(\mathring{H}_{j,6}+\mathring{H}^4_{j,6}))-\sum_l\nu_l\underline{\mathbf{v}_{j,l,}}_{0}\;,\; \mathbf{o}_k\big\rangle_{\mathbf{eh}}dv_\mathbf{eh} = 0.$$
    Let us use the formalism of \cite{biq2} in order to show this. Let us denote $\Omega$ the closed anti-selfdual $2$-form generating of the $L^2$-cohomology of $\mathbf{eh}$ as in \cite[(5)]{biq1}. We will need its asymptotics at infinity:
    \begin{equation}
    \begin{aligned}
       \Omega &= \frac{\theta_1^-}{r^4} - \frac{\theta_1^-+\frac{1}{2}\omega_1^+}{r^8} + \mathcal{O}(r_\mathbf{e}^{-12})\\
       &= d\Big(d^C\Big(\frac{1}{r^2}\Big) + \frac{\alpha_1}{4r^6} +\mathcal{O}(r_\mathbf{e}^{-11})\Big),\label{asympt Omega}
    \end{aligned}
    \end{equation}
    where $d^C$ is the operator defined as $d^C = \frac{1}{2i\pi} (\partial - \overline{\partial})$ in complex geometry. 
    
    As in the first Section, we use the identification of traceless $2$-tensors and (commuting) compositions of a selfdual and an anti-selfdual $2$-form, see Remark \ref{identification traceless 2tensor 2 formes}. We will use this on both the flat metric $\mathbf{e}$ for which a basis of $\Omega^+(\mathbf{e})$ is $(\omega_k^+)_{k\in\{1,2,3\}}$ and for the Eguchi-Hanson metric, we will use the basis $(\tilde{\omega}_k^+)_{k\in\{1,2,3\}}$ of $\Omega^+(\mathbf{eh})$ satisfying 
    $\tilde{\omega}_k^+ = \omega_k^+ +\mathcal{O}(r_\mathbf{e}^{-4})$ at infinity. We can therefore rewrite 
    \begin{itemize}
        \item $ \mathring{H}_{j,6} = \sum_k\tilde{\phi}_{j,6}^k\circ \tilde{\omega}_k^+$ for $\tilde{\phi}_{j,6}^k\in \Omega^-(\mathbf{eh})$, as well as $ \mathring{H}_{j,6} = \sum_k\phi_{j,6}^k\circ \omega_k^+$ for $\phi_{j,6}^k\in \Omega^-(\mathbf{e})$,
        \item $ \mathring{H}_{j,6}^4 = \sum_k\tilde{\phi}_{j,6}^{4,k}\circ \tilde{\omega}_k^+$ for $\tilde{\phi}_{j,6}^{4,k}\in \Omega^-(\mathbf{eh})$, as well as $ \mathring{H}_{j,6}^4 = \sum_k\phi_{j,6}^{4,k}\circ \omega_k^+$ for $\phi_{j,6}^{4,k}\in \Omega^-(\mathbf{e})$,
        \item $\sum_l\nu_l\underline{\mathbf{v}_{j,l,}}_{0} = \sum_k \tilde{V}_{j,0}^k\circ \tilde{\omega}_k^+$, for $\tilde{V}_{j,0}^k\in \Omega^-(\mathbf{eh})$ as well as $\sum_l\nu_lV_{j,l,0} = \sum_k V_{j,0}^k\circ 
        \omega_k^+$, for $V_{j,0}^k\in \Omega^-(\mathbf{e})$,
        \item $\mathbf{o}_k = \Omega \circ \tilde{\omega}_k^+$ see \cite{biq2},
        \item recall also that $H^4 =- \frac{\theta_1^-\circ\omega_1^+}{2r_\mathbf{e}^4}$.
    \end{itemize}
    This lets us rewrite several of the above identities:
    \begin{itemize}
        \item the equality \eqref{def H60} becomes
        \begin{equation}
          (d_-d^*)_\mathbf{e}\phi_{j,6}^k = 0, \text{ or equivalently } (dd^*)_\mathbf{e}\phi_{j,6}^k \in \Omega^+(\mathbf{e}) \label{reecriture def H60}
        \end{equation}
        \item the equality \eqref{equation H^4j,6,0} becomes:
        \begin{equation}
            (d_-d^*)_{\mathbf{eh}}(\phi_{j,6}^{k}+\phi_{j,6}^{4,k}) - \tilde{V}_0^k = \mathcal{O}(r_\mathbf{e}^{-4})\label{reecriture def H460}
        \end{equation}
        \item a direct extension of \cite[Lemme 3]{biq2} to the computation of the bilinear terms of the selfdual part of the Riemannian curvature. Let $\phi:=\sum_k\phi_k\circ\omega_k^+$ and $\psi:=\sum_k\psi_k\circ\omega_k^+$ be infinitesimal Ricci-flat deformations of $\mathbf{e}$ which satisfy $B_\mathbf{e}(\phi) = B_\mathbf{e}(\psi)= 0$. Assume additionally that $\textup{tr}_\mathbf{e} \phi=0$ and $ d*d\phi_k = 0 $, then the second variation of the selfdual part of the curvature at $\mathbf{e}$ in the directions $\phi$ and $\psi$, is:
        $$ \mathbf{R}^{+,(2)}_{\mathbf{e}}(\phi,\psi) = - \frac{1}{2} [a_\mathbf{e}(\phi),a_\mathbf{e}(\psi)]_+,$$
        where $a_\mathbf{e}(\phi)$ and $a_\mathbf{e}(\psi)$ are respectively the first variations of the connection on the bundle of selfdual $2$-forms at $\mathbf{e}$ in the directions $\phi$ and $\psi$. In the particular case of $\phi = H^4$ and $\psi = \mathring{H}_{j,6}$, since $a_\mathbf{e}(H^4) = -*d\big(\frac{\theta_1^-}{2r_\mathbf{e}^4}\big) = 0$, we find
        \begin{equation}
            \mathbf{R}^{+,(2)}_{\mathbf{e}}(\mathring{H}_{j,6},H^4) = 0.\label{pas de partie autoduale courbure}
        \end{equation}
        \item the equalities \eqref{equation H^4j,6,0} and \eqref{pas de partie autoduale courbure} imply that in a neighborhood of infinity we have
        \begin{equation}
           \begin{aligned} dd^*_\mathbf{eh}\big(\tilde{\phi}_{j,6}^k + \tilde{\phi}_{j,6}^{4,k}\big) - \tilde{V}_0^k &= (d_+d^*)_\mathbf{e}\phi_{j,6}^k + (d_+d^*)_\mathbf{e}\phi_{j,6}^{4,k}+ \mathcal{O}(r_\mathbf{e}^{-4})\\
           &=-d_\mathbf{e}\mathbf{R}^+(\mathring{H}_{j,6})- d_\mathbf{e}\mathbf{R}^+(\mathring{H}^4_{j,6})+ \mathcal{O}(r_\mathbf{e}^{-4}).\label{partie autoduale}
           \end{aligned}
        \end{equation}
    \end{itemize}
    
    Following the integration by parts of the proof of \cite[Lemma 7]{biq2}, since $\Omega \in \Omega^-(\mathbf{eh})$, by \eqref{asympt Omega} and by \eqref{partie autoduale}, we have
    \begin{align}
        \frac{1}{\|\mathbf{o}_k\|_{L^2}}&\int_{T^*\mathbb{S}^2}\big\langle P_\mathbf{eh}(\chi(\mathring{H}_{j,6}+\mathring{H}^4_{j,6}))-\sum_l\nu_l\underline{\mathbf{v}_{j,l,}}_{0}\;,\; \mathbf{o}_k\big\rangle_{\mathbf{eh}}dv_\mathbf{eh} \nonumber\\
        =&\; -\frac{1}{\|\Omega\|_{L^2}}\int_{T^*\mathbb{S}^2} \Omega\wedge\big( dd^*_\mathbf{eh} \big(\tilde{\phi}_{j,6}^k + \tilde{\phi}_{j,6}^{4,k}\big) - \tilde{V}_0^k\big)\nonumber \\
        =&\; -\frac{1}{\|\Omega\|_{L^2}}\int_{T^*\mathbb{S}^2} d\Big(d^C\Big(\frac{1}{r_\mathbf{e}^2}\Big) + \frac{\alpha_1}{4r_\mathbf{e}^6} +\mathcal{O}(r_\mathbf{e}^{-11})\Big)\wedge \left(dd^*_\mathbf{eh} \big(\tilde{\phi}_{j,6}^k + \tilde{\phi}_{j,6}^{4,k}\big) -  \tilde{V}_0^k\nonumber\right)\\
        =&\;-\frac{1}{\|\Omega\|_{L^2}}\lim_{r\to+\infty} \int_{\{r_\mathbf{e} = r\}}\Big(d^C\Big(\frac{1}{r_\mathbf{e}^2}\Big) + \frac{\alpha_1}{4r_\mathbf{e}^6}\Big)\wedge \big((d_+d^*)_\mathbf{e}\phi_{j,6}^k + (d_+d^*)_\mathbf{e}\phi_{j,6}^{4,k}\big)\nonumber\\
        =&\;+\frac{1}{\|\Omega\|_{L^2}}\lim_{r\to+\infty} r^4\int_{\{r_\mathbf{e} = 1\}}d^C\Big(\frac{1}{r_\mathbf{e}^2}\Big)\wedge d_\mathbf{e}\mathbf{R}^+(\mathring{H}_{j,6})(\omega_k^+) \nonumber\\
         &\;+\frac{1}{\|\Omega\|_{L^2}}\lim_{r\to+\infty} \int_{\{r_\mathbf{e} = 1\}}  \frac{\alpha_1}{4r_\mathbf{e}^6}\wedge d_\mathbf{e}\mathbf{R}^+(\mathring{H}_{j,6})(\omega_k^+)\nonumber\\
        &\;+ \frac{1}{\|\Omega\|_{L^2}}\lim_{r\to+\infty}\int_{\{r_\mathbf{e} = 1\}}d^C\Big(\frac{1}{r_\mathbf{e}^2}\Big)\wedge d_\mathbf{e}\mathbf{R}^+(\mathring{H}^4_{j,6})(\omega_k^+).\label{dvp obst h60}
    \end{align}
    
    Now, the coefficients of $d_\mathbf{e}\mathbf{R}^+(\mathring{H}_{j,6})(\omega_k^+)$ are harmonic functions in $r_\mathbf{e}^4$ and their restrictions to a sphere $\{r_\mathbf{e} = r\}= r\cdot\mathbb{S}^3\slash\mathbb{Z}_2$ are eigenfunctions of the Laplacian associated to the second eigenvalue on $\mathbb{S}^3\slash\mathbb{Z}_2$ (fourth eigenvalue on $\mathbb{S}^3$). They are therefore $L^2(\mathbb{S}^3\slash\mathbb{Z}_2)$-orthogonal to the coefficients of the restriction of $d^C\big(\frac{1}{r_\mathbf{e}^2}\big)$ or $ \frac{\alpha_1}{4r_\mathbf{e}^6}$ to $\mathbb{S}^3\slash\mathbb{Z}_2$.
    
    Similarly, by the two formulas of \cite[Proposition 2.1]{biq1} and \cite[Lemme 7]{biq2}, we know that
    $\int_{\{r_\mathbf{e} = r\}}d^C\big(\frac{1}{r_\mathbf{e}^2}\big)\wedge d_\mathbf{e}\mathbf{R}^+(\mathring{H}^4_{j,6})(\omega_k^+)$ is proportional to 
    $ \int_{\mathbb{S}^3\slash\mathbb{Z}_2}\big\langle \mathring{H}^4_{j,6}, O_k^4 \big\rangle dv_{\mathbb{S}^3\slash\mathbb{Z}_2} $, where $O_k^4$ is explicited in \eqref{def O^4_k} and once restricted to a sphere $\{r_\mathbf{e}=r\}$, $O_k^4$ has eigenfunctions of the Laplacian on $\mathbb{S}^3\slash\mathbb{Z}_2$ associated to the first eigenvalue as coefficients. By construction, $\mathring{H}^4_{j,6}$ is orthogonal to this eigenspace. The obstructions therefore vanish.
\end{proof}

We then extend the next term of the development of $\mathbf{o}_m$, for $m\in\{1,2,3\}$ to the orbifold. Recall that 
$\mathbf{o}_{j,m} = O_{j,m}^{4}+O_{j,m}^{8}+\mathcal{O}(|\zeta|^5r_\mathbf{e}^{-12}),$
with $O_{j,m}^{4}\propto |\zeta^\mp|r_\mathbf{e}^{-4}$ and $O_{j,m}^{8}\propto |\zeta^\mp|^3r_\mathbf{e}^{-8}$. We can therefore find by the general method of \cite[Section 10]{biq1} $\overline{\mathbf{o}}_{j,m}^8$ satisfying
\begin{equation}
    \left\{\begin{aligned}
  P_{\mathbf{g}_{\pm}}\overline{\mathbf{o}}_{j,m}^8 &+ Q_{\mathbf{g}_{\pm}}^{(2)}(\overline{\mathbf{o}}_{j,m}^4,\overline{\mathbf{o}}_{j,m}^4)+Q_{\mathbf{g}_{\pm}}^{(2)}(\overline{\mathbf{o}}_{j,m}^4,\overline{h}^4)\in \mathbf{O}(\mathbf{g}_\pm)\\
  \overline{\mathbf{o}}_{j,m}^8&=O_{m}^8(\zeta_j^\mp)+ \mathcal{O}(|\zeta^\pm|^2|\zeta^\mp|^4r_\mathbf{e}^{-6-\epsilon}) \text{ for any } \epsilon>0 \text{ at the point } j\\
  \overline{\mathbf{o}}_{j,k}^8&\perp \mathbf{O}(\mathbf{g}_o).
\end{aligned}\right.
\end{equation}
We then need to come back to the orbifold in order to find:
\begin{equation}
\left\{\begin{aligned}
  P_{\mathbf{g}_{\pm}}\overline{h}^8 &+ Q_{\mathbf{g}_{\pm}}^{(2)}(\overline{h}^4,\overline{h}^4)-\sum_k\lambda_k^j\overline{\mathbf{o}}_{j,m}^8 \in\mathbf{O}(\mathbf{g}_{\pm})\\
  \overline{h}^8&=H^8(\zeta_j^\mp)+H^8_{j,2} +H^8_{j,4}+\mathring{H}^8_{j,6}+\mathring{H}^8_{j,8}+ \mathcal{O}(|\zeta^\pm|^2|\zeta^\mp|^4r_\mathbf{e}^{-2-\epsilon}), \; \forall \epsilon>0 \text{ at } j\\
  \overline{h}^8&\perp\mathbf{O}(\mathbf{g}_o)
\end{aligned}\right.\label{def h^8}
\end{equation}
with $H^8(\zeta_j^\mp)$, $H^8_2$, $H^8_4$ and $\mathring{H}^8_6$ from \eqref{dvp gpm}, \eqref{def H^4_2 h_2}, and the asymptotic terms of \eqref{eq h^4j T4} and \eqref{def h60 ale} and $ \mathring{H}^8_{j,8} \propto |\zeta^\pm|^2|\zeta^\mp|^4$. We moreover have a converging development 
$$ \overline{h}^8 = H^8(\zeta_j^\mp)+H^8_{j,2} +H^8_{j,4}+\mathring{H}^8_{j,6}+\mathring{H}^8_{j,8} +\sum_{m\geqslant 5} H^8_{j,2m} $$ 
with $ H^8_{j,2m}\propto |\zeta^\pm|^2|\zeta^\mp|^4 r_\mathbf{e}^{2m-8} $ in a neighborhood of $j$. We moreover have the control 
\begin{equation}
    \big\|P_{\mathbf{g}_{\pm}}\overline{h}^8 + Q_{\mathbf{g}_{\pm}}^{(2)}(\overline{h}^4,\overline{h}^4)-\sum_k\lambda_k^j\overline{\mathbf{o}}_{j,m}^8\big\|_{L^2(\mathbf{g}_\pm)} = \mathcal{O}(|\zeta^\pm|^2\cdot|\zeta_j^\mp|^4)\label{controle suffisant orbifold}
\end{equation}
which will be sufficient on the orbifold.

As for $\mathring{H}_{j,6}$ defined in \eqref{def H60} from $H_{j,6}$, one finds $\mathring{H}_{j,8}$ from $H_{j,8}$ of \eqref{dvp gpm}. More precisely, we have $H_{j,8} - \mathring{H}_{j,8} \propto|\zeta^\pm|^{4}r_\mathbf{e}^8$ where $\mathring{H}_{j,8}\propto |\zeta^\pm|^{2}r_\mathbf{e}^8$ satisfies
$ P_\mathbf{e}\mathring{H}_{j,8} = \sum_l\nu_lV_{j,l,2}. $

As in Lemma \ref{h_j60 def}, we prove that one can extend $\mathring{H}_{j,8}$ on $\mathbf{eh}_{\zeta_j^\mp}$ without obstruction.
\begin{lem}
    There exists a unique $2$-tensor $ \mathring{\underline{h}}_{j,8}$ satisfying the following identity:
\begin{equation}
\left\{\begin{aligned}
  P_{\mathbf{eh}_{\zeta_j^\mp}}\mathring{\underline{h}}_{j,8}&-\sum_l\nu_l\underline{\mathbf{v}_{j,l,}}_{2}= 0 \\
  \mathring{\underline{h}}_{j,8}&= \mathring{H}_{j,8}+\mathring{H}_{j,8}^4+ \mathring{H}_{j,8}^8+\mathcal{O}(|\zeta^\pm|^2\cdot|\zeta_j^\mp|^{2}r_\mathbf{e}^{-6+\epsilon}) \text{ for any } \epsilon>0\\
    \int_{\mathbb{S}^2}\langle\mathring{\underline{h}}_{j,8},\mathbf{o}_{j,k}\rangle_{\mathbf{eh}_{\zeta_j^\mp}} &dv_{\mathbf{eh}_{\zeta_j^\mp|\mathbb{S}^2}}= 0 \text{ for all  } k\in\{1,2,3\},
\end{aligned}\right.
\end{equation}
where $\mathring{H}_{j,8}^4\propto |\zeta^\pm|^2|\zeta^\mp|^2r_\mathbf{e}^4$ satisfies
$ P_\mathbf{e} \mathring{H}_{j,8}^4 = -Q_\mathbf{e}^{(2)}(\mathring{H}_{j,8}, H^4(\zeta_j^\mp)) + \sum_l \nu_lV_{j,l,2}, $
and where $\mathring{H}_{j,8}^8\propto|\zeta^\pm|^2|\zeta^\mp|^4$ satisfies
$$ P_\mathbf{e} \mathring{H}_{j,8}^8 = -Q_\mathbf{e}^{(2)}(\mathring{H}_{j,8}, H^8(\zeta_j^\mp)) -Q_\mathbf{e}^{(2)}(\mathring{H}_{j,8}, H^4(\zeta_j^\mp), H^4(\zeta_j^\mp))+ \sum_l\nu_lV_{j,l,2}^4, $$
for $V_{j,l,2}^4$ defined by $\underline{\mathbf{v}_{j,l,}}_2 = V_{j,l,2}+V_{j,l,2}^4 + o(|\zeta^\pm|^2|\zeta^\mp|^4r_\mathbf{e}^{-6})$
\end{lem}
\begin{proof}
    The proof of the existence of $\mathring{\underline{h}}_{j,8}$ satisfying
\begin{equation}
\left\{\begin{aligned}
  P_{\mathbf{eh}_{\zeta_j^\mp}}\mathring{\underline{h}}_{j,8}&-\sum_l\nu_l\underline{\mathbf{v}_{j,l,}}_{2}\in \mathbf{O}(\mathbf{eh}) \\
  \mathring{\underline{h}}_{j,8}&= \mathring{H}_{j,8}+\mathring{H}_{j,8}^4+ \mathring{H}_{j,8}^8+\mathcal{O}(|\zeta^\pm|^2\cdot|\zeta_j^\mp|^{6}r_\mathbf{e}^{-6+\epsilon}) \text{ for any } \epsilon>0\\
    \int_{\mathbb{S}^2}\langle\mathring{\underline{h}}_{j,8},\mathbf{o}_{j,k}\rangle_{\mathbf{eh}_{\zeta_j^\mp}} &dv_{\mathbf{eh}_{\zeta_j^\mp|\mathbb{S}^2}}= 0 \text{ for all  } k\in\{1,2,3\},
\end{aligned}\right.\label{def h80 ale}
\end{equation}
is very similar to that of $\mathring{\underline{h}}_{j,6}$ of Lemma \ref{h_j60 def} once $\mathring{H}_{j,8}^4$ and $ \mathring{H}_{j,8}^8$ are identified like $\mathring{H}_{j,6}^4$ is in Lemma \ref{h_j60 def}. We will therefore omit the proof.

Let us focus on showing that the obstructions vanish. Again, after rescaling and acting by a rotation in $O(4)\backslash SO(4)$, we will work on the usual metric $\mathbf{eh}$. 
Mimicking the computation of \eqref{dvp obst h60}, and dropping the terms not contributing to the boundary term, we find
\begin{align}
\Big\langle P_{\mathbf{eh}}\big(\chi(\mathring{H}_{j,8}+\mathring{H}_{j,8}^4+ \mathring{H}_{j,8}^8)\big)&-\sum_l\nu_l\underline{\mathbf{v}_{j,l,}}_{2}\;,\; \mathbf{o}_{k}\Big\rangle\\
    =\lim_{r\to+\infty} &\left(\int_{\{r_\mathbf{e} = r\}}\Big(d^C\Big(\frac{1}{r_\mathbf{e}^2}\Big) + \frac{\alpha_1}{4r_\mathbf{e}^6}\Big)\wedge d_\mathbf{e}\mathbf{R}^+(\mathring{H}_{j,8})(\omega_k^+)\right. \nonumber\\
        &\;\;+\left.\int_{\{r_\mathbf{e} = r\}}d^C\Big(\frac{1}{r_\mathbf{e}^2}\Big)\wedge d_\mathbf{e}\mathbf{R}^+(\mathring{H}^4_{j,8})(\omega_k^+)\right)\nonumber\\
        = \lim_{r\to+\infty}&\left( r^6\int_{\{r_\mathbf{e} = 1\}}d^C\Big(\frac{1}{r_\mathbf{e}^2}\Big)\wedge d_\mathbf{e}\mathbf{R}^+(\mathring{H}_{j,8})(\omega_k^+) \nonumber\right.\\
        &\;\;+ r^2\int_{\{r_\mathbf{e} = 1\}}\frac{\alpha_1}{4r_\mathbf{e}^6}\wedge d_\mathbf{e}\mathbf{R}^+(\mathring{H}_{j,8})(\omega_k^+) \nonumber\\
        &\;\;+ \left. r^2\int_{\{r_\mathbf{e} = 1\}} d^C\Big(\frac{1}{r_\mathbf{e}^2}\Big)\wedge d_\mathbf{e}\mathbf{R}^+(\mathring{H}^4_{j,8})(\omega_k^+)\right)\nonumber
\end{align}
because of the homogeneity of each term: 
\begin{itemize}
    \item $d^C\Big(\frac{1}{r_\mathbf{e}^2}\Big)\propto r_\mathbf{e}^{-3}$,
    \item $\frac{\alpha_1}{4r_\mathbf{e}^6}\propto r_\mathbf{e}^{-7}$
    \item $d_\mathbf{e}\mathbf{R}^+(\mathring{H}_{j,8})(\omega_k^+) \propto |\zeta^\pm|^{2}r_\mathbf{e}^6$
    \item $d_\mathbf{e}\mathbf{R}^+(\mathring{H}^4_{j,8})(\omega_k^+)\propto |\zeta^\pm|^{2}|\zeta^\mp|^2r_\mathbf{e}^2$.
\end{itemize}
We see that the obstruction either vanishes or is infinite. However $\big\langle P_{\mathbf{eh}}\big(\chi(\mathring{H}_{j,8}+\mathring{H}_{j,8}^4+ \mathring{H}_{j,8}^8)\big)-\sum_l\nu_l\underline{\mathbf{v}_{j,l,}}_{2}, \mathbf{o}_{k}\big\rangle$ is finite because both $P_{\mathbf{eh}}\big(\chi(\mathring{H}_{j,8}+\mathring{H}_{j,8}^4+ \mathring{H}_{j,8}^8)\big)-\sum_l\nu_l\underline{\mathbf{v}_{j,l,}}_{2}$ and $\mathbf{o}_{k}$ are in $L^2(\mathbf{eh})$. The obstruction therefore necessarily vanishes.

\end{proof}

Similarly, there exists $\mathring{\underline{h}}_{j,10}$ satisfying: for $\mathring{H}_{j,10}$ the harmonic part of $H_{j,10}$ from \eqref{dvp gpm} and $\mathring{H}_{j,10}^4$ and $\mathring{H}_{j,10}^8$ associated from the development of $\overline{h}^4$ of \eqref{eq h^4 T4Z2} and $\overline{h}^8$ of \eqref{def h^8}:
\begin{equation}
\left\{\begin{aligned}
  P_{\mathbf{eh}_{\zeta_j^\mp}}\mathring{\underline{h}}_{j,10} &- \sum_{l}\nu_l\underline{\mathbf{v}_{j,l}}_{4} \in \mathbf{O}(\mathbf{eh}_{\zeta_j^\mp}) \\
  \mathring{\underline{h}}_{j,10}&= \mathring{H}_{j,10}+\mathring{H}_{j,10}^4+\mathring{H}_{j,10}^8+\mathbf{O}(|\zeta^\pm|^2\cdot|\zeta_j^\mp|^{6}r_\mathbf{e}^{-2+\epsilon}) \text{ for any } \epsilon>0\\
    \int_{\mathbb{S}^2}\langle\mathring{\underline{h}}_{j,10},\mathbf{o}_{j,k}\rangle_{\mathbf{eh}_{\zeta_j^\mp}} &dv_{\mathbf{eh}_{\zeta_j^\mp|\mathbb{S}^2}}= 0 \text{ for all  } k\in\{1,2,3\},
\end{aligned}\right.
\end{equation}

\begin{equation}
    \text{with }\|P_{\mathbf{eh}_{\zeta_j^\mp}}\mathring{\underline{h}}_{j,10} - \sum_{l}\nu_l\underline{\mathbf{v}_{j,l}}_{4}\|_{L^2(\mathbf{eh}_{\zeta_j^\mp})}=\mathcal{O}( |\zeta^\pm|^2|\zeta^\mp|^{5}).\label{control obst h10}
\end{equation}
\begin{rem}
    The obstruction may not vanish here, but it is small enough for our present purpose: it is better than $\mu_1\sim |\zeta^\pm|^4|\zeta^\mp|^2$ when $\zeta^\pm\sim \zeta^\mp$ as for \emph{transverse} or \emph{nondegenerate} desingularizations.
\end{rem}

 Let us define $$\mathbf{eh}_{\zeta_j^\mp}^{[10]}:=\mathbf{eh}_{\zeta_j^\mp}+ \underline{h}_{j,2} + \underline{h}_{j,4}+ \sum_{m=3}^{5} \mathring{\underline{h}}_{j,2m}$$
 which is a metric for $r_\mathbf{e} \ll 1$, and 
 $$\mathbf{g}_{\pm}^{[8]}:=\mathbf{g}_\pm + \overline{h}^4+\overline{h}^8$$
 which is a metric for $|\zeta_j|^\frac{1}{2}\ll r_\mathbf{e}$. On the region where 
 $|\zeta_j|^\frac{1}{2}\ll r_\mathbf{e}\ll 1$, we have
 \begin{equation}
     r_\mathbf{e}^k\Big|\nabla^k \big(\mathbf{eh}_{\zeta_j^\mp}^{[10]}-\mathbf{g}_\pm^{[8]}\big)\Big|\leqslant C \Big( |\zeta^\pm|^2r_\mathbf{e}^{12}+|\zeta_j^\mp|^6r_\mathbf{e}^{-12} + |\zeta^\pm|^2| \zeta_j^\mp|^2r^2\Big),\label{difference metrics}
 \end{equation}
 
which is minimized at $r_\mathbf{e} = |\zeta^\pm|^{-\frac{1}{12}}\cdot|\zeta_j^\mp|^{\frac{1}{4}}$, in a neighborhood of the singular point $j$.

\begin{proof}[Proof of Proposition \ref{prop construction gALzeta}]
We define an approximate desingularization $g^A_{L,\zeta}$ as the gluing of $\mathbf{eh}_{\zeta_j^\mp}^{[10]}$ to $\mathbf{g}_{\pm}^{[8]}$ at the singular point $j$ and with a gluing scale at $|\zeta^\pm|^{-\frac{1}{12}}\cdot|\zeta_j^\mp|^{\frac{1}{4}}<r_\mathbf{e} <2 |\zeta^\pm|^{-\frac{1}{12}}\cdot|\zeta_j^\mp|^{\frac{1}{4}}$. As in Definition \ref{def tilde O gAv zeta}, we will denote $\tilde{\mathbf{O}}(g^A_{L,\zeta})$ the space of the infinitesimal deformations of $(L,\zeta_i^\pm,\zeta_j^\mp)\mapsto g^A_{L,\zeta}$.
    \begin{rem}
        The space $\tilde{\mathbf{O}}(g^A_{L,\zeta})$ is different from the space $\tilde{\mathbf{O}}(g^D)$ of \cite{ozu2} but yields better estimates. The results of \cite{ozu2} with  $\tilde{\mathbf{O}}(g^D)$ therefore hold with $\tilde{\mathbf{O}}(g^A_{L,\zeta})$.
    \end{rem}
    
    Denote $\mathbf{o}^A_{L,\zeta}$ as the element of $\tilde{\mathbf{O}}(g^A_{L,\zeta})$ associated to the same variations of $L$ and the $\zeta_i^\pm$ as $ \sum_l \nu_l\mathbf{v}_l(L,\zeta) \in \mathbf{O}(\mathbf{g}_\pm)$ defined in \eqref{eq h^4 T4Z2}, and to the variations of each $\zeta_j^\mp$ associated to $ \sum_k (\lambda_k^j(L,\zeta) + \mu_k^j(L,\zeta) )\mathbf{o}_{j,k} $. From the construction of the $\mathbf{eh}_{\zeta_j^\mp}^{[10]}$ and $\mathbf{g}_{\pm}^{[8]}$ we find:
    \begin{enumerate}
        \item on the region where $g^A_{L,\zeta} = \mathbf{eh}_{\zeta_j^\mp}+\underline{h}_{j,2}+\underline{h}_{j,4}+\sum_{m=3}^{5} \mathring{\underline{h}}_{j,2m}$, we have:
        \begin{align*}
            (\Ric-\Lambda)(g^A_{L,\zeta}) =&\; \sum_k \big(\lambda_k^j + \mu_k^j + \mathcal{O}(|\zeta^\pm|^2|\zeta^\mp|^5)\big) \mathbf{o}_{j,k} \\
            &+ \sum_l (\nu_l+ \mathcal{O}(|\zeta^\pm|^2|\zeta^\mp|^4)) (\underline{\mathbf{v}_{j,l,}}_{0}+\underline{\mathbf{v}_{j,l,}}_{2}+\underline{\mathbf{v}_{j,l,}}_{4}) \\
            &+ \mathcal{O}(|\zeta^\pm|^4r_\mathbf{e}^4 + |\zeta^\pm|^2r_\mathbf{e}^{10} ),
        \end{align*}
        \item on the region where $g^A_{L,\zeta} = \mathbf{g}_\pm + \overline{h}^4+\overline{h}^8$, we have:
        \begin{align*}
            (\Ric-\Lambda)(g^A_{L,\zeta}) =& \sum_k (\lambda_k^j + \mu_k^j+ \mathcal{O}(|\zeta^\pm|^2|\zeta^\mp|^5)\big) (\overline{\mathbf{o}}_{j,k}^4+\overline{\mathbf{o}}_{j,k}^8)  \\
        &+ \sum_l (\nu_l+ \mathcal{O}(|\zeta^\pm|^2|\zeta^\mp|^4))\mathbf{v}_{l}\\
        &+ \mathcal{O}( |\zeta^\mp|^{6}r_\mathbf{e}^{-14})
        \end{align*}
       \item on the gluing region at $j$, we find:
       \begin{align*}
           (\Ric-\Lambda)(g^A_{L,\zeta})=&\sum_{k}(\lambda_k^j + \mu_k^j+ \mathcal{O}(|\zeta^\pm|^2|\zeta^\mp|^5)\big)\mathbf{o}_{j,k} \\
       &+\sum_l(\nu_l+ \mathcal{O}(|\zeta^\pm|^2|\zeta^\mp|^4))\mathbf{v}_l \\
       &+ \mathcal{O}(|\zeta^\pm|^4r_\mathbf{e}^4 + |\zeta^\pm|^2r_\mathbf{e}^{10} + |\zeta^\mp|^{6}r_\mathbf{e}^{-14}).
       \end{align*}
    \end{enumerate}
    
Let us finally mention the controls on the elements of $\tilde{\mathbf{O}}(g^A_{L,\zeta})$ which let us conclude. Recall that for $\zeta_{j,(k)}\in \Omega^\mp$, we have $\mathbf{o}_{j,k}=\partial_{\zeta_{j,(k)}}\mathbf{eh}_{\zeta_j^\mp}$, and consider $\mathbf{v}_l$ one of the elements of the orthonormal basis of $\mathbf{O}(\mathbf{g}_o)$. On the region where $g^A_{L,\zeta} =  \mathbf{eh}_{\zeta_j^\mp}^{[10]}$, where $r_D<|\zeta^\pm|^{-\frac{1}{12}}\cdot|\zeta_j^\mp|^{\frac{1}{4}}$:
$r_D^2\big|\partial_{\zeta_{j,(k)}}g^A_{L,\zeta} - \mathbf{o}_{j,k}\big|_\mathbf{\mathbf{eh}_{\zeta_j^\mp}}\leqslant C |\zeta^\pm|^2|\zeta^\mp|^2r_D^{-2}, $
on the region where $g^A_{L,\zeta} = \mathbf{g}_\pm^{[8]}$, where $r_D>2|\zeta^\pm|^{-\frac{1}{12}}\cdot|\zeta_j^\mp|^{\frac{1}{4}}$
$$r_D^2\big|\partial_{\zeta_{j,(k)}}g^A_{L,\zeta} - \big(\overline{\mathbf{o}}_{j,k}^4+\overline{\mathbf{o}}_{j,k}^8\big)\big|_{\mathbf{g}_{\pm}} < C\big(|\zeta^\pm|^6 r_D^{-12} + |\zeta^\pm|^2|\zeta^\mp|^2r_D^{-2}\big).$$
We have similar controls between $\partial_{\mathbf{v}_l}g^A_{L,\zeta}$ and $\mathbf{v}_l$ where $g^A_{L,\zeta}= \mathbf{g}_\pm^{[8]}$ and between $\partial_{\mathbf{v}_l}g^A_{L,\zeta}$ and $\underline{\mathbf{v}_{j,l,}}_{0}+\underline{\mathbf{v}_{j,l,}}_{2}+\underline{\mathbf{v}_{j,l,}}_{4}$ where $g^A_{L,\zeta}= \mathbf{eh}_{\zeta_j^\mp}^{[10]}$.
    
    As a consequence, using the fact that on the gluing region where $|\zeta^\pm|^{-\frac{1}{12}}|\zeta_j^\mp|^{\frac{1}{4}}\leqslant r_D\leqslant|\zeta^\pm|^{-\frac{1}{12}}\cdot|\zeta_j^\mp|^{\frac{1}{4}}$, the difference between the metrics is controlled by \eqref{difference metrics} we globally find
    $$\|\mathbf{\Phi}_{g^A_{L,\zeta}}(g^A_{L,\zeta}) - \mathbf{o}^A_{L,\zeta}\|_{r_D^{-2}C^\alpha_\beta}\leqslant C |\zeta^\pm|^{1+\frac{\beta}{12}}|\zeta^\mp|^{3-\frac{\beta}{4}}. $$
    which proves the statement of Proposition \ref{prop construction gALzeta}.
\end{proof}

\subsubsection{Controls on the approximate Einstein modulo obstructions metric}

    By \cite{ozu2}, we can always solve the Einstein equation modulo obstructions.
    
    \begin{defn}[Total Einstein modulo obstructions desingularization $\hat{g}_{L,\zeta}$]
        For $L\in GL(4,\mathbb{R})$, and $\zeta$ small enough depending on $L$ (as determined in \cite{ozu2}), we define $\hat{g}_{L,\zeta}$ as the unique metric satisfying for some $\epsilon=\epsilon(L)>0$ independent on $\zeta$ determined in \cite{ozu2}:
        \begin{itemize}
            \item $\|\hat{g}_{L,\zeta}-g^A_{L,\zeta}\|_{C^{2,\alpha}_{\beta,*}(g^A_{L,\zeta})}<\epsilon$,
            \item $\mathbf{\Phi}_{g^A_{L,\zeta}}(\hat{g}_{L,\zeta})\in \tilde{\mathbf{O}}(g^A_{L,\zeta})$, and
            \item $\hat{g}_{L,\zeta}-g^A_{L,\zeta}$ is $L^2(g^A_{L,\zeta})$-orthogonal to $\tilde{\mathbf{O}}(g^A_{L,\zeta})$.
        \end{itemize}
    \end{defn}
        
First, we have a good control on the projection on $\tilde{\mathbf{O}}(g^A_{L,\zeta})$. The proof is analogous to the proof of Lemma \ref{proj contre obst} above so we omit it, compare also with \cite[Lemma 13.2]{biq1}.
\begin{lem}\label{produit contre oA}
    Let $h$ be a symmetric $2$-tensor on $M$, and let $\mathbf{o}^A\in \tilde{\mathbf{O}}(g^A_{L,\zeta})$, then we have the following estimate, there exists $C>0$ such that
    $$ \left|\left\langle P_{g^D_{L,\zeta}}h,\mathbf{o}^A \right\rangle_{L^2(g^D_{L,\zeta})}\right|\leqslant C|\zeta^\pm|^2|\zeta^\mp|^{1-\frac{\beta}{2}} \|h\|_{C^{0}_{\beta,*}}\|\mathbf{o}^A\|_{L^2} . $$
\end{lem}
\begin{proof}[  ]
    
\end{proof}

Replacing the control \cite[(4.46), Lemma 4.46]{ozuthese} by the above Lemma \ref{produit contre oA} in the proof of \cite[Proposition 5.1]{ozuthese}, we get the following statement as a direct consequence.

\begin{lem}\label{meilleure approx and obst}
    There exists $\epsilon>0$ and $C>0$ such that we have
    \begin{equation}
        \|\hat{g}_{L ,\zeta} -g^A_{L ,\zeta} \|_{C_{\beta,*}^{2,\alpha}(g^A_{L ,\zeta})}\leqslant C\big\|\mathbf{\Phi}_{g^A_{L ,\zeta}} (g^A_{L ,\zeta} ) - \mathbf{o}^A_{L ,\zeta} \big\|_{r^{-2}_DC_{\beta}^{\alpha}(g^A_{L ,\zeta})}, \label{Meilleure approximation}
    \end{equation}
    and moreover if the metric $(M, \Hat{g}_{L ,\zeta})$ is an Einstein metric, then, we have
    \begin{align}
    \|\mathbf{o}^A_{L ,\zeta}\|_{L^2(g^A_{L ,\zeta})}\leqslant& \; C \big(\| \mathbf{\Phi}_{g^A_{L ,\zeta}} (g^A_{L ,\zeta})-\mathbf{o}^A_{L ,\zeta} \|_{r^{-2}_DC_{\beta}^{\alpha}(g^A_{L ,\zeta})}+ |\zeta^\mp|^{1-\frac{\beta}{2}}|\zeta^\pm|^2\big)\nonumber\\
        &\times \| \mathbf{\Phi}_{g^A_{L ,\zeta}} (g^A_{L ,\zeta} ) -\mathbf{o}^A_{L ,\zeta} \|_{r^{-2}_DC_{\beta}^{\alpha}(g^A_{L ,\zeta})}.\label{obstruction explicite}
    \end{align}
\end{lem}
\begin{proof}[ ]
    
\end{proof}

    Together with the control of Proposition \ref{prop construction gALzeta}, Lemma \ref{meilleure approx and obst} implies both a control of the metric and the obstructions.
    
\begin{rem}\label{almost orth obst}
    For small $\zeta$, we can essentially treat the obstructions coming from variations of different $\zeta_i^\pm$, $\zeta_j^\mp$ or $L$ as orthogonal to each other. Indeed, for the different infinitesimal variations, we have the following controls showing where each of the infinitesimal deformation has most of its mass:
    \begin{itemize}
        \item $ |\partial_{z_j^\mp}g^A_{L_k,\zeta_k}|_{g^A_{L_k,\zeta_k}} \leqslant C |\zeta^\mp_j| \|\partial_{z_j^\mp}g^A_{L_k,\zeta_k}\|_{L^2} \big(|\zeta_j^\mp|^\frac{1}{2}+d_{g^A_{L_k,\zeta_k}}(j,.)\big)^{-4}$
        \item  $|\partial_{z_i^\pm}g^A_{L_k,\zeta_k}|_{g^A_{L_k,\zeta_k}}\leqslant C |\zeta^\pm_i|\|\partial_{z_i^\pm}g^A_{L_k,\zeta_k}\|_{L^2}\big(|\zeta_i^\pm|^\frac{1}{2}+d_{g^A_{L_k,\zeta_k}}(i,.)\big)^{-4}$
        \item  $|\partial_{K}g^A_{L_k,\zeta_k}|_{g^A_{L_k,\zeta_k}}\leqslant C\|\partial_{K}g^A_{L_k,\zeta_k}\|_{L^2}$.
    \end{itemize}
    Estimating the $L^2$-product of two of the above infinitesimal, we see that as $|\zeta|\to 0$, they become arbitrarily close to being orthogonal.
\end{rem}

    \begin{cor}\label{control obst desing T4}
    We have the following controls:
    \begin{itemize}
        \item $\|\hat{g}_{L ,\zeta}-g^A_{L ,\zeta}\|_{C^{2,\alpha}_{\beta,*}(g^A_{L ,\zeta})}\leqslant C |\zeta^\pm|^{1+\frac{\beta}{12}}|\zeta^\mp|^{3-\frac{\beta}{4}}$,
        \item and $ \|\mathbf{o}^A_{L ,\zeta}\|_{L^2(g^A_{L ,\zeta})} \leqslant C|\zeta ^\pm|^{3+\frac{\beta}{12}} |\zeta ^\mp|^{4-\frac{3\beta}{4}}$ which rewrites as, for all $j\in S_\mp$ and $k\in\{1,2,3\}$, for any $0<\beta<1$ (chosen small enough)
        \begin{equation}
            \lambda_k^j(L,\zeta) + \mu_k^j(L,\zeta) = \mathcal{O}(|\zeta^{\pm}|^2|\zeta^{\mp}|^5 +|\zeta ^\pm|^{3+\frac{\beta}{12}} |\zeta ^\mp|^{4-\frac{3\beta}{4}}),\label{controle obst eh T4}
        \end{equation}
        with $|\zeta^{\pm}|^2|\zeta^{\mp}|^5$ from \eqref{control obst h10}, and for all $l$,
        \begin{equation}
            \nu_l(L,\zeta) = \mathcal{O}( |\zeta^\pm|^{2}|\zeta^\mp|^{4}+ |\zeta ^\pm|^{3+\frac{\beta}{12}} |\zeta ^\mp|^{4-\frac{3\beta}{4}}),\label{control obst nu T4}
        \end{equation}
        where the additional error in $ |\zeta^\pm|^{2}|\zeta^\mp|^{4}$ comes from \eqref{controle suffisant orbifold}.
    \end{itemize}
    \end{cor}
\begin{rem}
    Recall that a priori, we only had $ \lambda_k^j(L,\zeta) = \mathcal{O}(|\zeta^\pm|^2|\zeta^\mp|)$, $ \mu_k^j(L,\zeta)= \mathcal{O}(|\zeta^\pm|^4|\zeta^\mp|^2) $ and $\nu_l=\mathcal{O}(|\zeta^\pm|^2|\zeta^\mp|^2)$.
\end{rem}
\subsection{Obstruction to the total desingularization}
    We will now prove that there are obstructions to the Gromov-Hausdorff desingularization of $ \mathbb{T}^4\slash\mathbb{Z}_2 $ by Einstein metrics. 
    
\begin{prop}\label{controle obstructions T4Z2}
    Assume that there exists a sequence of Einstein modulo obstructions metrics $(\hat{g}_{L_n,\zeta_n}= \mathbf{g}_{L_n,\zeta_n})_{n\in\mathbb{N}}$ which are actually Einstein with $\zeta_n\to 0$ and $L_n\to L_\infty$ where $\det(L_\infty)>0$. Up to taking a subsequence, we consider the following limits:
    $ \frac{\zeta_{n,j}^\mp}{|\zeta_{n}^\mp|}\to \zeta_{\infty,j}^\mp$, $ \frac{\zeta_{n,i}^\pm}{|\zeta_{n}^\pm|}\to \zeta_{\infty,i}^\pm$ and $\frac{\zeta_{n,j}^\mp}{|\zeta_{n,j}^\mp|}\to \phi_{\infty,j}^\mp$, as well as the limit  
    \begin{equation}
        \mathbf{R}_{L_\infty,\zeta_{\infty}^\pm}:= \lim_{n\to\infty}\frac{1}{|\zeta_{n}^\pm|^2}\mathbf{R}_{\mathbf{g}_{\pm,n}}=\lim_{n\to\infty}\frac{1}{|\zeta_{n}^\pm|^2}\mathbf{R}_{g^A_{\pm,n}}.
    \end{equation}
    
    Then,
    \begin{enumerate}
        \item one always has 
    \begin{equation}
        \mathbf{R}_{L_\infty,\zeta_{\infty}^\pm}(j)(\phi_{\infty,j}^\mp) = 0,\label{obst R(xi)= 0}
    \end{equation}
    for all the deformations $(0,z^\pm)$ of $(L_\infty,\zeta_{\infty}^\pm)$, we have
    \begin{equation}
        \sum_{j\in S_\pm}\left\langle \partial_{(0,z^\pm)}\mathbf{R}_{L_\infty,\zeta_{\infty}^\pm}(j)(\zeta_{\infty,j}^\mp), \zeta_{\infty,j}^\mp \right\rangle = 0,\label{obst dR(zeta,zeta)=0}
    \end{equation}
    and for all the deformations $(K, 0)$ of $(L_\infty,\zeta_{\infty}^\pm)$, and $(\mathbf{v}_l)_l$ an orthonormal basis of $\mathbf{O}(\mathbf{g}_\pm)$ and $\nu_l(L_\infty,\zeta_\infty)$ the obstructions of \eqref{eq h^4 T4Z2}, we have 
    \begin{equation}
        \sum_l\langle \partial_{(K,0)}\mathbf{g}_{\pm},\mathbf{v}_l \rangle_{L^2(\mathbf{g}_\pm)}\nu_l(L_\infty,\zeta_\infty) = 0.\label{obst nu tore}
    \end{equation} 
        \item If we assume that the sequence $( \mathbf{g}_{L_n,\zeta_n})_{n\in\mathbb{N}}$ is a \emph{nondegenerate} sequence of Einstein desingularizations (see Definition \ref{non degenerate and transverse einstein}), then, one moreover has:
    \begin{equation}
        \mathbf{R}_{L_\infty,\zeta_{\infty}^\pm}(j) = 0\text{ at every $j\in S_\mp$.}\label{obst R=0}
    \end{equation}
    \end{enumerate}
\end{prop}
\begin{rem}
    The assumption $\det(L_\infty)>0$ for $L_\infty\in GL(4,\mathbb{R})$ ensures that there is no collapsing and a bounded diameter in our degeneration and therefore that we stay in the context of \cite{and,bkn,ozu1,ozu2}.
\end{rem}
\begin{rem}
    When \eqref{obst R(xi)= 0} is satisfied, from \eqref{estimée somme lambda 1}, we have the following rewriting of the obstruction \eqref{obst nu tore}:
    \begin{equation}
        \sum_{j\in S_\pm}\left\langle \partial_{(K, 0)}\mathbf{R}_{L_\infty,\zeta_{\infty}^\pm}(j)(\zeta_{\infty,j}^\mp), \zeta_{\infty,j}^\mp \right\rangle = 0.\label{obst dKR(zeta,zeta)=0}
    \end{equation}
\end{rem}
\begin{proof}[Proof of Proposition \ref{controle obstructions T4Z2}]

Let us assume that there exists a sequence of Einstein modulo obstructions $(\hat{g}_{L_n,\zeta_n}= \mathbf{g}_{L_n,\zeta_n})_{n\in\mathbb{N}}$ which are Einstein with $\zeta_n\to 0$ and $L_n\to L_\infty=\textup{Id}$. By compactness, up to taking a subsequence, for all $j\in S_\mp$ and all $i\in S_\pm$, we have $ \frac{\zeta_{n,j}^\mp}{|\zeta_{n,j}^\mp|}\to \zeta_{\infty,j}^\mp$, $ \frac{\zeta_{n,i}^\pm}{|\zeta_{n}^\pm|}\to \zeta_{\infty,i}^\pm$ and  $\frac{\zeta_{n,j}^\mp}{|\zeta_{n,j}^\mp|}\to \phi_{\infty,j}^\mp$ as $n\to +\infty$. By \eqref{controle obst eh T4} and \eqref{control obst nu T4}, 
for any $0<\beta<1$ (chosen small enough), we have
        \begin{equation}
            \lambda_k^j(L_n,\zeta_n) + \mu_k^j(L_n,\zeta_n) = \mathcal{O}(|\zeta_n^\pm|^{2}|\zeta_n^\mp|^5+|\zeta_n ^\pm|^{3+\frac{\beta}{12}} |\zeta_n ^\mp|^{4-\frac{3\beta}{4}}),\label{controle obst eh T4 n}
        \end{equation}
        and for all $l$,
        \begin{equation}
            \nu_l(L_n,\zeta_n) = \mathcal{O}( |\zeta_n^\pm|^{2}|\zeta_n^\mp|^{4}+ |\zeta_n ^\pm|^{3+\frac{\beta}{12}} |\zeta_n ^\mp|^{4-\frac{3\beta}{4}}),\label{controle obst orb T4 n}
        \end{equation}
    We directly see from \eqref{controle obst eh T4 n} that 
    $$ \lim_{n\to\infty}\frac{\lambda_k^j(L_n,\zeta_n)}{|\zeta_n^\pm|^2|\zeta_n^\mp|}= 0 $$
    because $ \mu_k^j(L_n,\zeta_n) = \mathcal{O}(|\zeta_n^\pm|^4|\zeta_n^\mp|^2) $ which proves \eqref{obst R(xi)= 0} by the value the obstruction \eqref{valeur obst lambda zeta} and the control of the curvature of Proposition \ref{controle gApm}.
    
    Similarly, from \eqref{controle obst orb T4 n} we find
    $$ \lim_{n\to\infty}\frac{\nu_l(L_n,\zeta_n)}{|\zeta_n^\pm|^2|\zeta_n^\mp|^2}= 0 $$ which proves \eqref{obst nu tore}. In order to prove \eqref{obst dR(zeta,zeta)=0}, we note that the constant part (the limit at the singular point) of $\partial_{(0,z^\pm)}\mathbf{g}_\pm$ scales like $|\zeta^\pm_n|^2$. Thanks to the expression \eqref{valeur nu l} of the obstruction, this tells us that the obstruction induced by the constant part (the limit at the singular point) scales like $|\zeta_n^\pm|^4|\zeta_n^\mp|^2$ because we have $H^4_2\propto |\zeta_n^\pm|^2|\zeta_n^\mp|^2r_\mathbf{e}^{-2}$ in the development \eqref{def H^4_2 h_2}. In the limit, we are left with the obstruction coming from the nonconstant part of $\partial_{(0,z^\pm)}\mathbf{g}_\pm$  whose interpretation in \eqref{estimée somme lambda 1} yields the obstruction \eqref{obst dR(zeta,zeta)=0} using the control of the variations of the curvature \eqref{control part opmRgpm}.
    \\
     We conclude exactly as in the proof of Theorem \ref{nouvelleobst} by using \eqref{controle obst eh T4 n} as a replacement for \eqref{contrôle obst}. More precisely, if we assume towards a contradiction that the obstruction \eqref{obst R=0} is not satisfied, that is $\lim_{n\to\infty}\frac{\mu_{j,1}(L_n,\zeta_n)}{|\zeta_n^\pm|^4|\zeta_n^\mp|^2}= 0$ is not satisfied, then since we have $ |\mathbf{R}_{\mathbf{g}_{\pm}}|\sim |\zeta_n^\pm|^2 $, the estimates \eqref{lambda1variations} and \eqref{borne inf lambda 2 3} become 
    $$
        \lambda_1^j(L_\infty,\phi_{n,j}) =  \mathcal{O}(|\zeta_n^\pm|^2|\phi_{n,j}-\phi_{\infty,j}|^2),\text{ and}$$
    $$
        \big|\big(\lambda_2^j(L_\infty,\phi_{n,j}),\lambda_3^j(L_\infty,\phi_{n,j})\big)\big|\geqslant c|\zeta_n^\pm|^2|\phi_{n,j}-\phi_{\infty,j}|, \text{ for some $c>0$}.$$
    As in the proof of Theorem \ref{nouvelleobst}, we find a contradiction using \eqref{controle obst eh T4 n} for $k = 2,3$ and then $k= 1$. Similarly, as in the proof of Theorem \ref{nouvelleobst} one can use the eigenvalue estimate of \cite{biq3} to deal with the \emph{stable} situation thanks to the estimate of Corollary \ref{control obst desing T4}.
\end{proof}

\begin{rem}
    Allowing some Eguchi-Hanson metrics to be glued at negligible scale might not preserve the obstruction $\mathbf{R}_\pm=0$ at the associated point and might reduce the number of obstructions by $2$ per such point. Indeed, $\mathbf{R}^\pm$ is in the $5$-dimensional space of traceless symmetric $2$-tensors and the first obstructions $\lambda_k = 0$, give $3$ independent equations. But we will see that the scale being negligible means that we lose the $3$ degrees of freedom. These situations should therefore be more obstructed even without the obstruction $\mathbf{R}^\pm=0$.
\end{rem}

\section{Explicit obstructions to the desingularization of $\mathbb{T}^4\slash\mathbb{Z}_2$}\label{section obst}

Let us now make the obstructions of Proposition \ref{controle obstructions T4Z2} more explicit as the vanishing of polynomial functions of the $\zeta^\pm$, whose coefficients depend on the flat metric $\mathbf{g}_L$ alone.

Recall from \eqref{courbure variation h4} that the curvature of the Eguchi-Hanson metric $ \mathbf{eh}_{\zeta^+} $ and by their first asymptotic term $\frac{\zeta^+\circ\rho(\zeta)^-}{r^4}$ has the form $12 \pi_{\textup{tr}}\frac{\rho(\zeta)^-\otimes \rho(\zeta)^-}{r^6}$ where $\rho(\zeta)^-$ is defined in Definition \ref{def rho -}.

\begin{note}
    During all of this section, we will often identify $\mathbb{R}^3$ and $\Omega^+$ or $\Omega^-$ thanks to the bases $(\omega_i^\pm)_{i\in\{1,2,3\}}$. This will make it more convenient to find identities relating different obstructions in Lemma \ref{same obstruction} for instance. That is why we are defining the selfdual or anti-selfdual curvature-like quantities $b_x$ and $B_x^L$ on $\mathbb{R}^3$. We will also use the identification $\mathbb{R}^3\otimes\mathbb{R}^3\approx \Omega^\pm\otimes\Omega^\pm\approx(\Omega^\pm)^*\otimes\Omega^\pm\approx End(\Omega^\pm)$.
\end{note}

For $x\in \mathbb{R}^4$ and $\zeta,\zeta'\in \mathbb{R}^3\backslash\{0\}$, we define $b_x(\zeta, \zeta'):= 12\frac{ \rho_x(\zeta)^-\otimes\rho_x(\zeta')^- }{|x|^6}$, where $\rho_x$ for $x\in \mathbb{R}^4$ is the rotation explicited in \eqref{rotation 2 forms}. Note that $\rho_{-x} = (\rho_x)^t = (\rho_x)^{-1}=\rho_x$. For any $z,z'\in \mathbb{R}^3$, we have
\begin{equation}
    \big\langle b_x(\zeta, \zeta')z,z'\big\rangle:= 12\frac{\langle \rho_x(\zeta),z \rangle\langle \rho_x(\zeta'),z' \rangle}{|x|^{6}}.\label{explicite bx}
\end{equation}
We also define the following bilinear form with values in $\mathbb{R}^3\otimes\mathbb{R}^3$:
\begin{equation}
    B_{x}^L(\zeta, \zeta') := \pi_{tr}\sum_{a\in \mathbb{Z}^4}b_{L(x-2a)}(\zeta, \zeta')\label{def BxL}
\end{equation}
for $x\in \mathbb{R}^4$ and $L$ a $4\times 4$ matrix, where $\pi_{tr}$ is the projection on the traceless part.

\begin{lem}\label{curvature gA-}
    Let $(M_\pm,g_\pm^A)$ be one of the partial desingularizations of $\mathbb{T}^4\slash\mathbb{Z}_2$ of Definition \eqref{def gApm}. Then, for any $j_0\in S_-$ and any $\zeta\in \mathbb{R}^3$ and the associated $\zeta^-\in \Omega^-$ (thanks to the basis $(\omega_i^-)_{i\in\{1,2,3\}}$), we have 
    \begin{equation}
        \mathbf{R}^-_{g^A_+}(j_0)\zeta^- = \sum_{i\in S_+} \big(B_{j_0-i}^L(\zeta_i, \zeta_i)\zeta\big)^-\label{obst R-=0}
    \end{equation}
    where again $B_{j_0-i}^L(\zeta_i, \zeta_i)\zeta\in \mathbb{R}^3$ and $\big(B_{j_0-i}^L(\zeta_i, \zeta_i)\zeta\big)^-\in \Omega^-$.
\end{lem}
\begin{proof}
    Away from the glued Eguchi-Hanson metrics, $g^A_+$ equals $\mathbf{g}_L+\sum_{i\in S_+}\overline{h}^4_{\zeta_i^+}$ where the $\overline{h}^4_{\zeta_i^+}$ are explicited in \eqref{expression extension orbifold}. By Remark \ref{courbure hzeta} and the formula \eqref{courbure variation h4}, the induced curvature is therefore given by \eqref{obst R-=0}.
\end{proof}
We have the following properties for $ B_{x}^L(\zeta, \zeta') $ for any $x\in \mathbb{R}^4$ and $\zeta, \zeta'\in \mathbb{R}^3$
\begin{enumerate}
    \item $B_{x}^L(\zeta, \zeta') = B_{x}^L(\zeta', \zeta),$
    \item $B_{x}^L(\zeta, \zeta') = B_{-x}^L(\zeta, \zeta') $,
    \item $B_{x+a}^L(\zeta, \zeta') = B_{x}^L(\zeta, \zeta')$ for $a\in \mathbb{Z}^4$,
\end{enumerate}
\begin{cor}\label{list obst}
    The obstruction \eqref{obst R(xi)= 0} rewrites for all $j_0\in S_-$,
    \begin{equation}
        \sum_{i\in S_+} \left\langle B_{j_0-i}^L(\zeta_i, \zeta_i) \zeta_{j_0} , \zeta_{j_0} \right\rangle = 0.\label{obst exp R(xi)=0}
    \end{equation}
    and from 
    \eqref{obst dR(zeta,zeta)=0} for $j_0\in S_-$ and an infinitesimal deformation of $\zeta_{j_0}$ in the direction $ z_{j_0}$:
    \begin{equation}
        \sum_{i\in S_+} \left\langle B_{{j_0}-i}^L(z_{j_0}, \zeta_{j_0}) \zeta_{i} , \zeta_{i} \right\rangle = 0.\label{obst exp dzR(xi,xi)=0}
    \end{equation}
    For the conjectural value of the obstructions \eqref{obst dKR(zeta,zeta)=0}, against a deformation $K$ of the matrix $L$, denoting $\partial_KB_{j-i}^L$ the linearization of $L\mapsto B_{j-i}^L$ for $i\in S_+$ and $j\in S_-$ at $L$ and in the direction $K$, we have
    \begin{equation}
        \sum_{i\in S_+}\sum_{j\in S_-} \left\langle \partial_KB_{j-i}^L(\zeta_i, \zeta_i) \zeta_{j} , \zeta_{j} \right\rangle = 0.\label{obst exp dKR(xi,xi)=0}
    \end{equation}
    Finally, the obstruction \eqref{obst R=0} rewrites for any $j_0\in S_-$:
    \begin{equation}
       \sum_{i\in S_+} B_{j_0-i}^L(\zeta_i, \zeta_i) = 0 \label{obst exp R=0}
    \end{equation}
    
    The analogous obstructions hold for the other orientation.
\end{cor}

\begin{rem}
    The obstruction identified in \cite{bk} corresponds to the nonvanishing of the one obstruction
    $$\sum_{i\in S_+}\sum_{j\in S_-} \left\langle B_{j-i}^L(\zeta_i, \zeta_i) \zeta_{j} , \zeta_{j} \right\rangle = 0 $$
    (which is clearly implied by \eqref{obst exp R(xi)=0}),
    in the particular situation where
    \begin{itemize}
        \item $L=\textup{Id}$,
        \item the sets $S_\pm$ follow a chessboard pattern (see \cite{bk} or Section \ref{section regular torus}),
        \item for all $i\in S_+$, $j\in S_-$, $\zeta_i=(1,0,0)$ and $\zeta_j=(1,0,0)$.
    \end{itemize}
\end{rem}

\begin{rem}\label{comptage obst}
    If the Eguchi-Hanson metrics are glued at comparable scales, then:
    \begin{itemize}
        \item the obstructions \eqref{obst exp R(xi)=0} correspond to $3|S_+|$ equations for $3|S_-| + 9$ parameters given by the $|S_-|$ vectors $\zeta_j^-$ and the matrix $L$, 
        \item the obstructions  \eqref{obst exp R=0} correspond to $5|S_+|$ equations since $\mathbf{R}$ is symmetric trace-free for $3|S_-| + 9$ parameters given by the $|S_-|$ vectors $\zeta_j^-$ and the matrix $L$,
        \item the obstructions \eqref{obst exp dKR(xi,xi)=0} are $9$ equations, for the different deformations of the torus, for $3|S_+|+3|S_-|+9 = 57$ parameters,
        \item the obstructions \eqref{obst exp dzR(xi,xi)=0} are $3|S_-|$ equations for $3|S_+|+9$ parameters.
    \end{itemize}
    Since we have analogous obstructions in the other orientation, we see that we a priori have many more obstructions to satisfy than parameters. We will reduce the number of these equations in the next section.
\end{rem}

\subsection{Equivalent obstructions}

Let us show that many of the above equations are actually equivalent. Let us therefore just assume that we have the obstructions \eqref{obst exp R=0} at every singular point in both orientations and prove that some of the other obstructions are already satisfied.

\subsubsection{Against the deformation of the Eguchi-Hanson metrics}
\begin{lem}\label{same obstruction}
    We have the following equality for all $i\in S_+$, $j\in S_-$, $\zeta\in (\mathbb{R}^3)^{|S_+|}\times (\mathbb{R}^3)^{|S_-|}$ and $z\in \mathbb{R}^3$, $$\big\langle B_{j-i}^L(\zeta_{i}, z)\zeta_{j} , \zeta_{j}  \big\rangle = \big\langle B_{j-i}^L(\zeta_{j}, \zeta_{j})\zeta_{i}, z \big\rangle,$$
    and therefore the above obstruction \eqref{obst exp R=0} at all singular points in the opposite orientation implies that the obstructions \eqref{obst exp dzR(xi,xi)=0} vanish.
\end{lem}
\begin{rem}
    This shows that the first obstructions coming from the deformations of the Eguchi-Hanson metrics are seen whether we first desingularize the positive ones or the negative ones.
\end{rem}
\begin{proof}
    Given the expression of $B_{j-i }^L$ of \eqref{def BxL}, we will study for $a\in \mathbb{Z}^4$ a term of the above expressions. Because of the symmetries of $B$, we can simply study $\langle \pi_{tr}b_x(\zeta_{i } ,z )\zeta_j ,\zeta_j  \rangle$. We will first consider $z$ proportional to $\zeta_{i } $ to obtain 
\begin{align} 
     \big\langle \pi_{tr}b_x(\zeta_{i },\zeta_{i })\zeta_j ,\zeta_j  \big\rangle &= \big\langle \big(12\rho_x(\zeta_{i }) \otimes\rho_x(\zeta_{i })  - 4|\zeta_{i }|^2 \mathrm{I_3} \big) \zeta_{j} ,\zeta_{j}  \big\rangle \\
     &= 12\big\langle \rho_x(\zeta_{i }) , \zeta_{j}  \big\rangle^2 - 4 |\zeta_{i }|^2|\zeta_{j} |^2\\
     &= 12\big\langle \big(\rho_x(\zeta_{j})\otimes\rho_x(\zeta_{j}) - 4 |\zeta_{j} |^2\mathrm{I_3} \big) \zeta_{i },\zeta_{i } \big\rangle\\
     &= \big\langle \pi_{tr}b_x(\zeta_j ,\zeta_j )\zeta_{i },\zeta_{i } \big\rangle,
\end{align}
where we used the fact that $\rho_{-x} = (\rho_x)^T = (\rho_x)^{-1}=\rho_x$. For $z \perp\zeta_{i } $, 
\begin{align} 
     \langle \pi_{tr}b_x(\zeta_{i } ,z )\zeta_j ,\zeta_j  \rangle&=\big\langle \big(12\rho_x(\zeta_{i }) \otimes\rho_x(z )  \big) \zeta_{j} ,\zeta_{j}  \big\rangle\\
     &= 12\big\langle \rho_x(\zeta_{i }) , \zeta_{j}  \big\rangle\big\langle \rho_x(z ) , \zeta_{j}  \big\rangle\\
     &= 12\big\langle \big(\rho_x(\zeta_{j}) \otimes\rho_x(\zeta_{j})  - 4|\zeta_{j} |^2 \mathrm{I_3} \big) \zeta_{i } ,z  \big\rangle \\
     &= \langle \pi_{tr}b_x(\zeta_j ,\zeta_j )\zeta_{i } ,z  \rangle. 
\end{align}
This proves the statement.
\end{proof}

\subsubsection{Against the flat deformations of the torus}

The last deformations to consider are the flat deformations of the torus which are equivalent to variations of the matrix $L$ used in the definition of $B^L_x$. For the conjectural form of the obstruction \eqref{obst dKR(zeta,zeta)=0}, we compute the following expression. We omit the proof as the obstruction itself is only conjectural.

\begin{lem}
    Let us consider an infinitesimal variation of the flat torus seen as an infinitesimal variation of the matrix $L$ in the direction $KL^{-1}$, and assume that the obstruction \eqref{obst exp R=0} is satisfied in both orientations.
    
    Then, the conjectural obstruction \eqref{obst dKR(zeta,zeta)=0} for a variation $KL^{-1}$ of $L$ is equivalent to
    \begin{align}
    \sum_{i\in S_+}\sum_{j\in S_-}\sum_{a\in  \mathbb{Z}^4}\frac{ \langle L(i-j-2a), K(i-j-2a) \rangle}{|L(i-j-2a)|^2} \big\langle \pi_{tr}b_{L(i-j-2a)}(\zeta_i,\zeta_i) \zeta_j ,\zeta_j  \big\rangle = 0
    \end{align}
\end{lem}
\begin{proof}[ ]
    
\end{proof}

\begin{rem}
    The deformations of the matrix $\textup{Id}$ other than those in the space spanned by the $\pi_{tr}(e_k\otimes e_k)$ for $(e_k)_{k\in\{1,2,3,4\}}$ the usual basis of $\mathbb{R}^4$, are equivalent to rotations of the torus and therefore to the rotation of the Eguchi-Hanson metrics and have already been considered above.
\end{rem}

\begin{cor}\label{cor freedom constraints}
    The obstructions of Corollary \ref{list obst} reduce to $84$ polynomial equations in the $\zeta_i$ and $\zeta_j$ with coefficients depending on $L$ only.
\end{cor}

\subsection{Some obstructed situations}

We therefore see from Remark \ref{comptage obst} that even though some obstructions are equivalent, there are many remaining ones that outnumber the degrees of freedom we have. Indeed, we have only ruled out the last type of obstructions corresponding to \eqref{obst exp dzR(xi,xi)=0}.

Let us now give an example of configuration of desingularization which cannot yield Ricci-flat (or even Einstein) metrics thanks to the obstructions identified in Corollary \ref{list obst}.

\begin{thm}\label{obst un EH}
    There does not exist a sequence Einstein metrics $(M,\mathbf{g}_n)_{n\in \mathbb{N}}$ converging in the Gromov-Hausdorff sense to  $(\mathbb{T}^4\slash\mathbb{Z}_2,\mathbf{g}_{L_\infty})$ with $L_\infty=\textup{Id}$ by bubbling out \emph{exactly} 1 positively oriented Eguchi-Hanson metric and 15 negatively oriented one.
\end{thm}
\begin{rem}
    This is an obstruction to any Gromov-Hausdorff desingularization under a topological assumption.
\end{rem}
\begin{proof}
    Let $(M,\mathbf{g}_n)_{n\in\mathbb{N}}$ be a sequence of Einstein metrics converging to $(\mathbb{T}^4\slash\mathbb{Z}_2,\mathbf{g}_{L_\infty})$ with $L_\infty=\textup{Id}$ in the Gromov-Hausdorff sense while bubbling out Eguchi-Hanson metrics and let us assume without loss of generality that $S_+=\{0\}$ and $S_-$ is the complement of $\{0\}$ among the singular points of $\mathbb{T}^4\slash\mathbb{Z}_2$. Then, according to \cite{ozu2}, up to taking a subsequence, for all $n$, there exist $L_n$, $\zeta_{n,0}^+$ and $\zeta_{n,j}^-$ for $j\in S_-$, such that the metrics $\mathbf{g}_n$ are isometric to Einstein modulo obstructions perturbations of $g^D_{L_n,\zeta_n}$.
    
    Let us show that the limit rescaled curvature $\mathbf{R}^-_{L_\infty,\zeta_{\infty}^+}(j)$ defined in Proposition \ref{controle obstructions T4Z2} is invertible at a singular point $j\in S_-$. This will lead to a contradiction by the obstruction \eqref{obst R(xi)= 0}. 
    
    By Lemma \ref{curvature gA-}, we have an expression of $\mathbf{R}^-_{L_\infty,\zeta_{\infty}^+}(j)$: it is a nonvanishing multiple of
    \begin{equation}
        \sum_{a\in\mathbb{Z}^4} \pi_{tr}\frac{ \theta_1^{-}((1,0,0,0)+2a)\otimes \theta_1^{-}((1,0,0,0)+2a)}{|(1,0,0,0)+2a|^6}.\label{somme courbure 1eh}
    \end{equation}
    The value of the summand at $a=0$ and $a=(-1,0,0,0)$ is
    \begin{equation}
        \pi_{tr} (\omega_1^{-}\otimes \omega_1^{-}) = \frac{1}{3}(2\omega_1^{-}\otimes \omega_1^{-}-\omega_2^{-}\otimes \omega_2^{-}-\omega_3^{-}\otimes \omega_3^{-}),\label{premier terme}
    \end{equation}
    which is invertible on $\Omega^-$. For the sum of the norms of the remaining terms, that is,
    $$\sum_{a\in\mathbb{Z}^4\backslash\{0,(-1,0,0,0)\}} \Big|\pi_{tr}\frac{ \theta_1^{-}((1,0,0,0)+2a)\otimes \theta_1^{-}((1,0,0,0)+2a)}{|(1,0,0,0)+2a|^6}\Big|,$$
    we numerically have the control
    $$ \sum_{a\in\mathbb{Z}^4\backslash\{0,(-1,0,0,0)\}}\frac{1}{|(1,0,0,0)+2a|^6}\approx 0.19<\frac{2}{3}.$$
    This is not enough to make the sum \eqref{somme courbure 1eh} non invertible given the expression \eqref{premier terme}. The result follows by \eqref{obst exp R(xi)=0}. 
\end{proof}

\begin{rem}
    The above proof works for many other limits than $L_\infty=\operatorname{Id}$ and it is likely that it is true for any $L_\infty$, but it requires some control of the compactness of the lattice generated by $L_\infty$. Theorem \ref{obst un EH} particular holds even with the most compact lattices like the $D_4$ lattice. By compactness, we loosely mean how close the points of the lattice are to each other.
\end{rem}
\begin{conj}\label{conj 1 eh+}
    The statement of Theorem \ref{obst un EH} holds even without the assumption $L_\infty=\textup{Id}$, but only assuming that $\det(L_\infty)>0$.
\end{conj}
For \emph{stable} or \emph{nondegenerate} Ricci-flat deformations, one does not simply have $\R^-_{L_\infty,\zeta_\infty^+}(j)$ invertible, but $\R^-_{L_\infty,\zeta_\infty^+}(j)=0$. Since it is a sum of terms of the form \eqref{somme courbure 1eh} based at each singular point, it is direct to see that $\R^-_{L_\infty,\zeta_\infty^+}(j)=0$ requires at least $3$ Eguchi-Hanson metrics in this orientation.
\begin{cor}\label{obst 3 positive generale stable}
    There does not exist a sequence of \emph{stable Ricci-flat} metrics $(M,\mathbf{g}_n)_{n\in \mathbb{N}}$ converging in the Gromov-Hausdorff sense to  $(\mathbb{T}^4\slash\mathbb{Z}_2,\mathbf{g}_{L_\infty})$ with $L_\infty=\textup{Id}$ by bubbling out \emph{at most} 3 positively oriented Eguchi-Hanson metric and the rest of negatively oriented ones.
\end{cor}

\subsection{A family of desingularizations satisfying the obstructions}\label{section regular torus}

Let us now specialize our discussion to the regular torus with $L=\textup{Id}$ and with the so-called chessboard pattern considered in \cite{bk}. We study the $48$-dimensional space of gluings to the (fixed) regular torus and test it against the $ 16 \times 5 + 4 = 84$ (the last $4$ only having a conjectural expression in terms of curvature) obstruction equations identified in Section \ref{section obst}.

Perhaps surprisingly, in this most symmetric situation, we will find a family of solutions to all of our $84$ equations. This family of solutions is $14$-dimensional.

\begin{rem}
    It will be clear from our proof and computations in this section that in the configuration considered in \cite{bk}, many of these $84$ equations are not satisfied.
\end{rem}

The \emph{chessboard} configuration already considered in \cite{bk} is defined as: 
$$S_-=\big(\{(a_1,a_2,a_3,a_4)\in \mathbb{Z}^4, a_1+a_2+a_3+a_4\in 2\mathbb{Z}\}\slash\mathbb{Z}^4\big)\slash\mathbb{Z}_2 \text{ and}$$
$$S_+=\big(\{(a_1,a_2,a_3,a_4)\in \mathbb{Z}^4, a_1+a_2+a_3+a_4\in 1+2\mathbb{Z}\}/\mathbb{Z}^4\big)\slash\mathbb{Z}_2.$$

\begin{rem}\label{remark premiers voisins}
    In this chessboard configuration, let $i\in S_\pm$. Then, we have the following configuration of other singular points:
    \begin{itemize}
        \item at distance $0$ is $i\in S_\pm$ 
        \item at distance $1$, there are $4$ points in $S_\mp$
        \item at distance $\sqrt{2}$, there are $6$ points in $S_\pm$,
        \item at distance $\sqrt{3}$, there are $4$ points in $S_\mp$, and
        \item at distance $2$, there is a point in $S_\pm$.
    \end{itemize}
    
    We will denote $i^c$ the opposite of $i$ which is the point at maximal distance $2$ from $i$.
\end{rem}
Here we will work under the assumption that for all $i$,
$$ \zeta_i=\zeta_{i^c}. $$
We will denote $e_1=(1,0,0,0)\in S_+$ and similarly $e_i\in S_+$ for $i\in \{1,2,3,4\}$ the other vectors of the canonical basis of $\mathbb{R}^4$.
Let us introduce the following notations for the coordinates of our gluing parameters $\zeta_{e_i}$: $$ \zeta_{e_i}=:(x_i,y_i,z_i), $$ for $i\in\{1,2,3,4\}$. 
Let us denote $x=(x_1,x_2,x_3,x_4)$,  $y=(y_1,y_2,y_3,y_4)$ and $z=(z_1,z_2,z_3,z_4)$ the three vectors in $\mathbb{R}^4$ we obtain.

\begin{prop}
    Using the above notations, let us assume that the family $(x,y,z)$ of vectors of $\mathbb{R}^4$ forms an orthogonal family of $\mathbb{R}^4$ with constant length. Then, the obstructions \eqref{obst dKR(zeta,zeta)=0} and the obstructions \eqref{obst R=0} at every $j\in S_-$ are satisfied.
\end{prop}
\begin{proof}[Sketch of proof (some computer-based arguments involved)]
    Let us consider a desingularization configuration as above.
    
    Let us study the curvature at the singular point $ 0\in S_-$ induced by the other Eguchi-Hanson metrics glued at the points of $S_+$. From \eqref{obst R-=0}, we know that this exactly equals 
    \begin{equation}
        \sum_{i=1}^4B^\textup{Id}_{e_i}((x_i,y_i,z_i),(x_i,y_i,z_i))+B^\textup{Id}_{e_i^c}((x_i,y_i,z_i),(x_i,y_i,z_i)). \label{courbure cas qui marche}
    \end{equation}
    
    Numerically, we obtain the following form for $B^\textup{Id}_{e_i}((x_i,y_i,z_i),(x_i,y_i,z_i))+B^\textup{Id}_{e_i^c}((x_i,y_i,z_i),(x_i,y_i,z_i))$:
    $$
    \begin{bmatrix}
a\cdot (2x_i^2-y_i^2-z_i^2) & b\cdot x_iy_i & b\cdot x_iz_i\\
b\cdot x_iy_i & a\cdot (-x_i^2+2y_i^2-z_i^2) & b\cdot y_iz_i\\
b\cdot x_iz_i & b\cdot y_iz_i & a\cdot (-x_i^2-y_i^2+2z_i^2)
\end{bmatrix},$$
    for some real numbers $a\approx 0.69$ and $b\approx 2.0$, and similar formulas for the other $i$ by symmetries. Therefore, the equation 
    \begin{equation}
        \sum_{i=1}^4B^\textup{Id}_{e_i}((x_i,y_i,z_i),(x_i,y_i,z_i))+B^\textup{Id}_{e_i^c}((x_i,y_i,z_i),(x_i,y_i,z_i)) =0. \label{courbure cas qui marche annulation}
    \end{equation}
    coming from the obstruction \eqref{obst exp R=0} reduces to having both $$ \sum_{i= 1}^4x_i^2 =\sum_{i= 1}^4y_i^2=\sum_{i= 1}^4z_i^2, \text{  and}$$ $$ \sum_{i = 1}^4x_iy_i=\sum_{i = 1}^4x_iz_i=\sum_{i = 1}^4y_iz_i = 0. $$ This means that we are looking for an orthogonal family of $3$ vectors $x= (x_i)_i$, $y= (y_i)_i$ and $z= (z_i)_i$ with same length in $\mathbb{R}^4$. This is given by an element of $\mathbb{R}^+\times\mathbb{R}P^3\times O(3)$. More precisely, the element of $\mathbb{R}$ is equal to $\sum_{i= 1}^4x_i^2 =\sum_{i= 1}^4y_i^2=\sum_{i= 1}^4z_i^2$, the element of $\mathbb{R}P^3$ is a line orthogonal to $(x,y,z)$ in $\mathbb{R}^4$ and $O(3)$ is the orthonormal basis $ (x/|x|,y/|y|,z/|z|) $ in the orthogonal of the element of $\mathbb{R}P^3$. This set is $7$-dimensional.
    
    Similarly, we can study the variations of the curvature when $L=\textup{Id}$ in the directions $e_i\otimes e_i$ for $i\in\{ 1,2,3,4\}$. We again find numerically that the conjectural obstructions \eqref{obst dKR(zeta,zeta)=0} are of the form:
    $$
    \begin{bmatrix}
c\cdot (2x_i^2-y_i^2-z_i^2) & d\cdot x_iy_i & e\cdot x_iz_i\\
d\cdot x_iy_i & c\cdot (-x_i^2+2y_i^2-z_i^2) &f\cdot y_iz_i\\
e\cdot x_iz_i & f\cdot y_iz_i & c\cdot (-x_i^2-y_i^2+2z_i^2)
\end{bmatrix},$$
    for some nonvanishing real numbers $c,d,e,f$ which depend on the direction in which the deformation is done. Therefore, they also vanish under the conditions $ \sum_{i= 1}^4x_i^2 =\sum_{i= 1}^4y_i^2=\sum_{i= 1}^4z_i^2 $ and $ \sum_{i = 1}^4x_iy_i=\sum_{i = 1}^4x_iz_i=\sum_{i = 1}^4y_iz_i = 0 $.
    
    Moreover, because we have similar equations and degrees of freedom in the other orientation, we obtain a $14$-dimensional family of candidates satisfying all of the obstructions of Corollary \ref{list obst}.
\end{proof}

    \begin{exmp}\label{exemple solution}
        Some examples of solutions to the above equations are
        \begin{equation}
            \begin{aligned}
                &\zeta_1 = (1,1,1), \;&\zeta_2 =(1,-1,-1),\;& \zeta_3 = (-1,1,-1) \text{ and }\; & \zeta_4 = (-1,-1,1) \text{ and }
            \end{aligned}
        \end{equation}
       \begin{equation}
            \begin{aligned}
                &\zeta_1 = (1,0,0), \;&\zeta_2 =\left(0,1/\sqrt{2},0\right),\;& \zeta_3 = \left(0,0,1/\sqrt{2}\right) \text{ and }\; & \zeta_4 = (1,0,0).
            \end{aligned}
        \end{equation}
    \end{exmp}
\appendix

\section{Function spaces}\label{function spaces}

    For a tensor $s$, a point $x$, $\alpha>0$ and a Riemannian manifold $(M,g)$. The Hölder seminorm is defined as
$$ [s]_{C^\alpha(g)}(x):= \sup_{\{y\in T_xM,|y|< \textup{inj}_g(x)\}} \Big| \frac{s(x)-s(\exp^g_x(y))}{|y|^\alpha} \Big|_g.$$

For orbifolds, we consider a norm bounded for tensors decaying at the singular points.
\begin{defn}[Weighted Hölder norms on an orbifold]\label{norme orbifold}
    Let $\beta\in \mathbb{R}$, $k\in\mathbb{N}$, $0<\alpha<1$ and $(M_o,g_o)$ an orbifold. Then, for any tensor $s$ on $M_o$, we define
    \begin{align*}
        \| s \|_{C^{k,\alpha}_{\beta}(g_o)} &:= \sup_{M_o}r_o^{-\beta}\Big(\sum_{i=0}^k r_o^{i}|\nabla_{g_o}^i s|_{g_o} + r_o^{k+\alpha}[\nabla_{g_o}^ks]_{C^\alpha(g_o)}\Big).
    \end{align*}
\end{defn}

For ALE manifolds, we will consider a norm bounded for tensors decaying at infinity.

\begin{defn}[Weighted Hölder norms on an ALE orbifold]\label{norme ALE}
Let $\beta\in \mathbb{R}$, $k\in\mathbb{N}$, $0<\alpha<1$ and $(N,g_b)$ be an ALE manifold. Then, for all tensor $s$ on $N$, we define
   \begin{align*}
       \| s \|_{C^{k,\alpha}_{\beta}(g_b)}:= \sup_{N}r_b^\beta\Big( \sum_{i=0}^kr_b^{i}|\nabla_{g_b}^i s|_{g_b} + r_b^{k+\alpha}[\nabla_{g_b}^ks]_{C^\alpha({g_b})}\Big).
   \end{align*}
\end{defn}

On $M$, using a partition of unity 
$1= \chi_{M_o^t} + \sum_j \chi_{N_j^t}$
with $\chi_{M_o^t}$ equals to $1$ where $g^D_t=\mathbf{g}_o$ and $\chi_{N_j^t}$ equals to $1$ where $g^D_t = t_j \mathbf{g}_{b_j}$ (see \cite{ozu1}), we can define a global norm.
\begin{defn}[Weighted Hölder norm on a naïve desingularization]\label{norme a poids M}
		Let $\beta\in \mathbb{R}$, $k\in\mathbb{N}$ and $0<\alpha<1$. We define for $s\in TM^{\otimes l_+}\otimes T^*M^{\otimes l_-}$ a tensor $(l_+,l_-)\in \mathbb{N}^2$, with $l:= l_+-l_-$ the associated conformal weight,
		$$ \|s\|_{C^{k,\alpha}_{\beta}(g^D)}:= \| \chi_{M_o^t} s \|_{C^{k,\alpha}_{\beta}(g_o)} + \sum_j T_j^\frac{l}{2}\|\chi_{N_{j}^t}s\|_{C^{k,\alpha}_{\beta}(g_{b_j})}.$$
\end{defn}

\subsection{Decoupling norms.}

We actually need a last family of norms to get good analytic properties for our operators, see \cite{ozu2}. With the notations of Definition \ref{orb Ein}, denote for each singular point $k$, $A_k(t,\epsilon_0) := (\Phi_k)_*A_\mathbf{e}(\epsilon_0^{-1}\sqrt{t_k},\epsilon_0)$ and $B_k(\epsilon_0):=(\Phi_k)_*B_\mathbf{e}(0,\epsilon_0)$, as well as cut-off functions $\chi_{A_k(t,\epsilon_0)}$ and $\chi_{B_k(\epsilon_0)}$ respectively supported in $A_k(t,\epsilon)$ and $B_k(\epsilon_0)$, and equal to $1$ on $A_k(t,2\epsilon_0)$ and $B_k(\epsilon_0/2)$.

\begin{defn}[$C^{k,\alpha}_{\beta,*}$-norm on $2$-tensors]
    Let $ h $ be a $2$-tensor on $(M,g^D)$, $(M_o,g_o)$ or $(N,g_b)$. We define its $C^{k,\alpha}_{\beta,*}$-norm by
    $$\|h\|_{C^{k,\alpha}_{\beta,*}}:= \inf_{h_*,H_k} \|h_*\|_{C^{k,\alpha}_{\beta}} + \sum_k |H_k|_{g_e},$$
    where the infimum is taken on the $(h_*,(H_k)_k)$ satisfying $h= h_*+\sum_k \chi_{A_k(t,\epsilon_0)}H_k$ for $(M,g^D)$ or $h= h_*+\sum_k \chi_{B_k(\epsilon_0)}H_k$ for $(M_o,g_o)$ or $(N,g_b)$, where each $H_k$ is some constant and trace-free symmetric $2$-tensors on $\mathbb{R}^4\slash\Gamma_k$.
\end{defn}

\subsection{An application of the analysis on weighted spaces.}

Consider for $n\in\mathbb{N}$, for $\gamma\in \mathbb{N}\backslash\{3-n,\ldots,-1\}$ and $0<\beta<1$ the Fredholm operator 
$$P_\mathbf{e} : r_\mathbf{e}^\gamma C^{2,\alpha}_{-\beta}(\mathbb{R}^n\backslash\{0\})\mapsto r_\mathbf{e}^{\gamma-2}C^{\alpha}_{-\beta}(\mathbb{R}^n\backslash\{0\}), $$
where the norms  $C^{l,\alpha}_{-\beta}(\mathbb{R}^n\backslash\{0\})$ on symmetric $2$-tensors denote the norm on $\mathbb{R}^n\backslash\{0\}$ seen as orbifold at $0$ like in Definition \ref{norme orbifold} and ALE at infinity like in definition \ref{norme ALE}. This operator is Fredholm and its kernel is composed of homogeneous harmonic $2$-tensors in $r_\mathbf{e}^\gamma$.
Its $L^2$-cokernel is the kernel of 
$$P_\mathbf{e} : r_\mathbf{e}^{n+2-\gamma} C^{2,\alpha}_{\beta}(\mathbb{R}^n\backslash\{0\})\mapsto r_\mathbf{e}^{n-\gamma}C^{\alpha}_{\beta}(\mathbb{R}^n\backslash\{0\}), $$
which is reduced to $\{0\}$. As a direct consequence of the open mapping theorem, we have the following result.
\begin{lem}\label{resolution equation laplacien}
    Let $n\in\mathbb{N}$, for $\gamma\in \mathbb{N}\backslash\{3-n,\ldots,-1\}$ and $0<\beta<1$. Then, there exists $C>0$ such that for any symmetric $2$-tensor $v\in r_\mathbf{e}^{\gamma-2}C^{\alpha}_{-\beta}(\mathbb{R}^n\backslash\{0\})$, there exists a $2$-tensor $h$ satisfying
    $ P_\mathbf{e}h = v, $
    with $$\|h\|_{r_\mathbf{e}^\gamma C^{2,\alpha}_{-\beta}(\mathbb{R}^n\backslash\{0\})}\leqslant C \|v\|_{r_\mathbf{e}^{\gamma-2}C^{\alpha}_{-\beta}(\mathbb{R}^n\backslash\{0\})}.$$
\end{lem}
\begin{proof}[ ]

\end{proof}

\Addresses

\end{document}